\documentclass[12pt]{amsart}
\usepackage[margin=3cm, marginpar=3cm]{geometry}

\usepackage{amssymb}
\usepackage{braket}
\usepackage{mathrsfs}
\usepackage{ifthen}
\usepackage{graphicx}

\usepackage{comment} 
\usepackage{mleftright}
\usepackage[pagebackref,hypertexnames=false]{hyperref} 
\usepackage{cleveref}
 \usepackage[all]{xy}
 \usepackage{amscd}
 \usepackage{color}
 \usepackage{enumitem}
\newlist{steps}{enumerate}{1}
\setlist[steps, 1]{label = Step \arabic*:}
\usepackage[fontsize=10]{scrextend}
\usepackage[pagebackref,hypertexnames=false]{hyperref} 
\usepackage[alphabetic,backrefs,msc-links]{amsrefs}

\makeatletter
\DeclareRobustCommand\widecheck[1]{{\mathpalette\@widecheck{#1}}}
\def\@widecheck#1#2{%
   \setbox\z@\hbox{\m@th$#1#2$}%
   \setbox\tw@\hbox{\m@th$#1%
      \widehat{%
         \vrule\@width\z@\@height\ht\z@
         \vrule\@height\z@\@width\wd\z@}$}%
   \dp\tw@-\ht\z@
   \@tempdima\ht\z@ \advance\@tempdima2\ht\tw@ \divide\@tempdima\thr@@
   \setbox\tw@\hbox{%
      \raise\@tempdima\hbox{\scalebox{1}[-1]{\lower\@tempdima\box\tw@}}}%
   {\ooalign{\box\tw@ \cr \box\z@}}}
\makeatother

\theoremstyle{plain}
\newtheorem{thm}{Theorem}[section]
\crefname{thm}{Theorem}{Theorems}
\Crefname{thm}{Theorem}{Theorems}
\newtheorem{prop}[thm]{Proposition}
\crefname{prop}{Proposition}{Propositions}
\Crefname{prop}{Proposition}{Propositions}
\newtheorem{lem}[thm]{Lemma}
\crefname{lem}{Lemma}{Lemmas}
\Crefname{lem}{Lemma}{Lemmas}
\newtheorem{cor}[thm]{Corollary}
\crefname{cor}{Corollary}{Corollaries}
\Crefname{cor}{Corollary}{Corollaries}

\crefname{claim}{Claim}{Claims}
\Crefname{claim}{Claim}{Claims}

\crefname{property}{Property}{Properties}
\Crefname{property}{Property}{Properties}

\crefname{problem}{Problem}{Problems}
\Crefname{problem}{Problem}{Problems}
\newtheorem{ques}[thm]{Question}
\crefname{ques}{Question}{Questions}
\Crefname{ques}{Question}{Questions}

\theoremstyle{definition}
\newtheorem{defn}[thm]{Definition}
\crefname{defn}{Definition}{Definitions}
\Crefname{defn}{Definition}{Definitions}

\crefname{notation}{Notation}{Notations}
\Crefname{notation}{Notation}{Notations}

\crefname{convention}{Convention}{Conventions}
\Crefname{convention}{Convention}{Conventions}

\crefname{cond}{Condition}{Conditions}
\Crefname{cond}{Condition}{Conditions}

\crefname{assum}{Assumption}{Assumptions}
\Crefname{assum}{Assumption}{Assumptions}
\newtheorem{conj}[thm]{Conjecture}
\crefname{conj}{Conjecture}{Conjectures}
\Crefname{conj}{Conjecture}{Conjectures}

\theoremstyle{remark}
\newtheorem{rem}[thm]{Remark}
\crefname{rem}{Remark}{Remarks}
\Crefname{rem}{Remark}{Remarks}

\crefname{ex}{Example}{Examples}
\Crefname{ex}{Example}{Examples}

\crefname{section}{Section}{Sections}
\Crefname{section}{Section}{Sections}
\crefname{subsection}{Subsection}{Subsections}
\Crefname{subsection}{Subsection}{Subsections}
\crefname{figure}{Figure}{Figures}
\Crefname{figure}{Figure}{Figures}

\newtheorem*{acknowledgement}{Acknowledgement}

\newcommand{\Z}{\mathbb{Z}}

\newcommand{\Q}{\mathbb{Q}}

\newcommand{\fraks}{\mathfrak{s}}

\newcommand{\Coker}{\mathop{\mathrm{Coker}}\nolimits}
\newcommand{\im}{\operatorname{Im}}
\newcommand{\rank}{\mathop{\mathrm{rank}}\nolimits}
\newcommand{\Hom}{\mathop{\mathrm{Hom}}\nolimits}

\newcommand{\id}{\mathrm{id}}
\newcommand{\ind}{\mathop{\mathrm{ind}}\nolimits}

\newcommand{\C}{\mathbb{C}}
\newcommand{\s}{\mathfrak{s}}
\newcommand{\wh}{\widehat}

\newcommand{\pr}{\text{pr}}
\newcommand{\al}{\alpha}

\def\om{\omega}
\def\Om{\Omega}

\newcommand{\R}{\mathbb R}

\newcommand{\ctext}[1]{\raise0.2ex\hbox{\textcircled{\scriptsize{#1}}}}

\def\ker{\operatorname{Ker}}
\def\cok{\operatorname{Cok}}

\def\dim{\operatorname{dim}}
\def\rank{\operatorname{rank}}

\def\Hom{\operatorname{Hom}}

\def\id{\operatorname{Id}}

\def\ind{\operatorname{ind}}

\newcommand{\mbar}[1]{{\ooalign{\hfil#1\hfil\crcr\raise.167ex\hbox{--}}}}

\def\gr{\operatorname{gr}}
\def\cV{\mathcal{V}}

\def\cF{\mathcal{F}}

\def\cU{\mathcal{U}}

\def\wt{\widetilde}
\def\H{\mathbb{H}}
\def\Spinc{Spin$^c$ }

\title{Monopoles and Transverse Knots}

\author{Nobuo Iida}
\address{Tokyo Institute of Technology, Ookayama, Meguro-ku, Tokyo}
\email{iida.n.ad@m.titech.ac.jp}

\author{Masaki Taniguchi} 
\address{Department of Mathematics, Graduate School of Science, Kyoto University, Kitashirakawa Oiwake-cho, Sakyo-ku, Kyoto 606-8502, Japan}
\email{taniguchi.masaki.7m@kyoto-u.ac.jp}

\begin{document}

\begin{abstract}
We present a framework for studying transverse knots and symplectic surfaces utilizing the Seiberg--Witten monopole equation. Our primary approach involves investigating an equivariant Seiberg--Witten theory introduced by Baraglia--Hekmati on branched covers, incorporating invariant contact/symplectic structures.

Within this framework, we introduce a novel slice-torus invariant denoted as $q_M(K)$. This invariant can be viewed as the Seiberg--Witten analog of Hendricks--Lipshitz--Sarkar
's $q_\tau$ invariant, with a signature correction term. One property of the invariant $q_M(K)$ is an adjunction equality for properly embedded connected symplectic surfaces in the symplectic filling $D^4\# m \overline{\mathbb{C}P}^2$. The proof of this equality utilizes the equivariant version of the homotopical contact invariant introduced in \cite{IT20}, leading to a transverse knot invariant. Another ingredient of the proof involves constructing invariant symplectic structures on branched covering spaces branched along properly embedded symplectic surfaces in symplectic fillings. As an application of the invariant $q_M(K)$, we determine the value of any slice-torus invariant within a permissible deviation of $2$ for squeezed knots concordant to certain Montesinos knots. Additionally, we provide an obstruction to realizing second homology classes of $D^4 \#m \overline{\mathbb{C}P}^2$ as connected embedded symplectic surfaces with transverse knot boundary or connected embedded Lagrangian surfaces with collarable Legendrian knot boundary. Moreover, we introduce a new obstruction to certain Montesinos knots being quasipositive, which is described only in terms of slice genera and their signatures.
 
\end{abstract}

\maketitle
\setcounter{tocdepth}{1}
\tableofcontents

\section{Introduction}
\subsection{Contact Structures and Transverse Knots}
A contact structure on a $(2n+1)$-dimensional oriented manifold $M$ is a codimension-1 distribution $\xi=\ker \lambda \subset TM$, $\lambda \in \Omega^1(M)$ satisfying the contact condition
\[
\lambda\wedge (d\lambda)^n>0.
\]
We focus on contact structures on 3-dimensional manifolds, which have intensively interacted with low-dimensional topology.  
When a contact structure is given on an oriented 3-manifold, 
there are two distinguished kinds of knots:
Legendrian knots, which are tangent to the contact plane, and transverse knots, which are transverse to it. The study of Legendrian knots has utilized the Seiberg--Witten theory in various ways. 
For instance, Mrowka--Rollin \cite{MR06} proved a generalization of Bennequin inequality using Kronheimer--Mrowka's variant of the Seiberg--Witten invariants for 4-manifolds with contact boundary. 
\par
The purpose of this paper is to give a new framework to study transverse knots and symplectic surfaces using Seiberg--Witten theory. Our method is to develop Baraglia--Hekmati's equivariant Seiberg--Witten theory \cite{BH} under invariant contact/symplectic structures on branched covers. For a prime number $p$ and a knot $K$ in $S^3$, Baraglia--Hekmati introduced a formal desuspension of an $S^1\times \Z_p$-equivariant pointed homotopy type 
\[
SWF(\Sigma_p (K),  \s_0) 
\]
constructed from Seiberg--Witten equation and whose $\Z_p$-action comes from the covering transformation on the $p$-th branched covering space $\Sigma_p (K)$ along $K$. Here, $\s_0$ is the unique $\Z_p$-invariant spin structure on $\Sigma_p (K)$.
By applying $S^1\times \Z_p$-equivariant cohomology to $SWF(\Sigma_p(K),  \s_0)$, Baraglia--Hekmati defined \footnote{More precisely, Baraglia--Hekmati did not define metric-independent Seiberg--Witten Floer stable homotopy type $SWF(-\Sigma_p (K), \mathfrak{s}_0)$. Instead, they defined metric-dependent stable homotopy type and proved that its $ S^1 \times \Z_p $-equivariant cohomology including a suitable grading shift is independent of the metric. We use notation as if we have the metric-independent stable homotopy type $SWF(-\Sigma_p (K), \mathfrak{s}_0)$ to simplify the notation, but the metric-dependent one introduced by Baraglia--Hekmati is sufficient for all the arguments in this paper.} 
\[
\wt{H}^*_{S^1\times \Z_p} (SWF(\Sigma_p (K) , \s_0); \mathbb{F}_p).
\]
\par

One of the main key ingredients in this paper is based on an equivariant version of Seiberg--Witten Floer homotopy contact invariant introduced in \cite{IT20}, formulated as $\Z_p$-equivariant map 
\begin{align}\label{homotopy contact eq}
\mathcal{C}_{\Z_p }( \Sigma_p(K), \wt{\xi}) : S^0 \to \Sigma^{d_3(\wt{\xi}) + \frac{1}{2} } (SWF(-\Sigma_p (K), \s_0)) , 
\end{align}
where $\wt{\xi}$ is the natural $\Z_p$-invariant contact structure constructed in \cite{Gonz, Pla06} on the cyclic $p$-th branched covering $\Sigma_p(K)$ and $d_3(\wt{\xi})\in \Q$ is the Gompf's invariant \cite{Go98h} of the oriented 2-plane field  $\wt{\xi}$. 
 We conjecture that this invariant recovers the Heegaard Floer transverse knot invariant defined by Kang \cite{Ka18} based on Hendricks--Lipshitz--Sarkar's equivariant Heegaard Floer homology with respect to branched involutions \cite{HLS16}. Before focusing on the theoretical developments in this paper, we first state applications to symplectic surfaces and Lagrangian surfaces in symplectic fillings.  

\subsection{Applications to symplectic surfaces}
The symplectic surfaces bounded by transverse knots have been studied as a connection between knot theory and symplectic topology. Closed symplectic surfaces in closed symplectic 4-manifolds have been studied in various ways. The existence of connected immersed symplectic surfaces realizing a given homology class in a closed symplectic 4-manifold has been proven under a homotopical necessary condition \cite{Li08}. For higher dimensions, it is proven in \cite{Li08} that the homotopical condition is known to be enough to show the existence of embedded connected symplectic surfaces.  Hence, the remaining challenge lies in the existence problem of embedded connected closed symplectic surfaces realizing specific homology classes in four dimensions. 
For certain homology classes, several transcendental constructions of closed symplectic surfaces have been known \cite{Do96, Tau00, LL02}. Results on the existence of symplectic surfaces were used to classify symplectic structures on $\C P^2$ \cite{Tau00} up to symplectomorphisms. \footnote{This technique showing the uniqueness of symplectic structures has been developed to apply other symplectic 4-manifolds, see \cite{Sa13} for example. }  
In \cite{Ham11}, certain obstructions to the existence of closed symplectic surfaces have been provided for specific homology classes using Seiberg--Witten theory.

We focus on symplectic surfaces properly embedded with boundaries in symplectic fillings. Regarding symplectic surfaces bounded by knots in $S^3$, the existence problem of symplectic surfaces in symplectic fillings has been resolved \cite{BO01, Ru83} in the following manner: A given knot $K$ in $S^3$ is a quasipositive knot if and only if there exists a transverse representative of $K$ that bounds a symplectic surface in $(D^4\# m \overline{\C P}^2,\omega_{\text{std}})$, where $\omega_{\text{std}}$ denotes the standard symplectic structure on $D^4$ blown up $m$ times.
However, the homology classes that arise as symplectic surfaces are difficult to detect even for this simple situation. 
Equivariant Seiberg--Witten theory can be employed to obstruct homology classes of symplectic surfaces. We will begin by stating a result concerning symplectic surfaces in $D^4 \# m \overline{\mathbb{C}P}^2$.
\begin{thm}\label{Q-minimize statement} 
Let  $S \subset X$ be a properly embedded connected symplectic surface in a symplectic filling $(X, \om)$ of $(S^3, \xi_{\operatorname{std}})$ with $\partial S=K$, where $\xi_{\operatorname{std}}$ is the standard contact structure on $S^3$ and $K$ is a knot in $S^3$. Suppose the Ozv\'ath--Szab\'o's contact invariant $c(\Sigma_2 (K), \wt{\xi})$ with respect to the contact structure $\wt{\xi}$ on $\Sigma_2 (K)$ gives a minimum $\mathbb{Q}$-grading element in the Heegaard Floer homology $HF^+_*(-\Sigma_2 (K);\mathfrak{s}_0) $, where $\mathfrak{s}_0$ is the unique spin structure on $\Sigma_2(K)$. 
Suppose 
\[
-\frac{1}{8} [S]^2-\frac{1}{2}\eta([S]/2)
> \frac{1}{2}\langle c_1(TX), [S]\rangle, 
\]
where 
\[
\eta(x):=\min_{c\in H^2(X; \Z) }\left\{\ -(x+c)^2-b_2(X)\ \middle|\  c \equiv w_2(X) \text{ mod }2  \right\}. 
\]
Then $[S] \in H_2(X)$ is not divisible by $2$.
\end{thm}

As concrete applications, one can give the following sequence of examples. 
\begin{cor}\label{computation for CP}
Let $x$ be an element in $2 H_2(D^4 \# m\overline{\C P}^2, \partial ;\Z)$ with an expression 
\[
x = \sum x_i e_i \in H_2 ( D^4 \# m\overline{\C P}^2, \partial ; \Z), 
\]
where $e_i$ is the class represented by the exceptional curve in each summand of $\overline{\C P}^2$. 
If the homology class admits a properly embedded connected symplectic surface $S$ with 
\[
\partial S = \begin{cases}
    T(3,6n\pm 1) \\ 
    T(5, 7) , T(5,9)  
\end{cases}  , 
\]
then  
\[
x_i = 0 \text{ or }   -2 \text{ or } -4,  
\]
where $T(p,q)$ denotes the torus knot of type $(p,q)$.

\end{cor}

Although most torus knots do not satisfy the assumptions of \cref{Q-minimize statement}, under a more general assumption, we see another obstruction again from equivariant Seiberg--Witten theory.
\begin{thm}\label{main symp sur}
    Let $K$ be a transverse knot in $(S^3, \xi_{\operatorname{std}})$. Suppose 
    \[
    g_4 (K) + \frac{1}{2}\sigma(K)>0
    \]
    where $\sigma(K)$ denotes the knot signature with the convention $\sigma(T(2,3))=-2$ and $g_4 (K)$ denotes the smooth four genus of $ K$.
Then there is no a properly embedded connected symplectic surface $S$ in a symplectic filling $(X, \om)$ with $\partial S = K$ such that $[S]$ is divisible by $2$ and  
\[
g(S) - \frac{1}{4}[S]^2 + \frac{1}{2}\sigma (K)=0. 
\]
\end{thm}

\begin{cor}\label{ex for 510n+3}
Let $x$ be an element in $2 H_2(D^4 \# m\overline{\C P}^2,\partial ;\Z)$.  
Suppose 
\[
 g_4 (K) + \frac{1}{2}\sigma(K)>0.
 \]
If the homology class $x$ admits a properly embedded symplectic surface $S$ in $(D^4 \# m\overline{\C P}^2, \om_{\operatorname{std}})$ with 
$\partial S = K $ being a transverse knot with repsect to $\xi_{\mathrm{std}}$, we have 
\[
-\langle c_1(\om_{\om_{\operatorname{std}}}), [S]\rangle + \frac{1}{2}[S]^2+ \sigma(K) + \operatorname{sl} (K) +1 \neq 0, 
\]
where $\operatorname{sl}$ denotes the self-linking number. 
\end{cor}
For example, most torus knots are known to satisfy $g_4 (K) + \frac{1}{2}\sigma(K)>0$. 
As another class of examples, one can also treat the positively cusped Whitehead doubles. 
\begin{cor} \label{quasipos app}
 Let $K$ be a strongly quasipositive knot such that its positively cusped Whitehead double is a non-trivial knot. 
Then there is no properly embedded connected symplectic surface $S \subset D^4 \# \overline{\C P}^2 $ with 
$\partial S = \operatorname{Wh}(K)$ such that
\[
[S] =   -2 [\C P^1] \in H_2(D^4 \# \overline{\C P}^2, \partial; \Z) \cong H_2(\overline{\C P}^2; \Z). 
\]
\end{cor}

The same type of results holds for a collarable Lagrangian filling bounded by a Legendrian knot instead of a symplectic surface bounded by a transverse knot.
 \footnote{For the definition of collarable Lagrangian filling, see \cite{Ch12}.  }
This follows from the above theorems combined with Cao--Gallup--Hayden--Sabloff's result \cite[Lemma 4.1]{CGHS}, which is based on Eliashberg's work \cite[Lemma 2.3.A]{Eli95}.

\begin{cor}
    The same statements in \cref{computation for CP} and \cref{ex for 510n+3} hold for connected and embedded  Lagrangian surfaces with collarable Legendrian boundary. 
\end{cor}

\begin{rem}
    In the proof of \cref{Q-minimize statement} and \cref{main symp sur}, we did not use the classification result of symplectic fillings of $(S^3, \xi_{\operatorname{std}})$. Using the classification, one sees $X \cong  (D^4 , \om_{\operatorname{std}} ) \# _m \overline{\C P}^2$ \cite[p.311]{Gr85}\cite[Theorem 5.1]{Eli90}\cite[Theorem 1.7]{McD90}(See also \cite[Corollary 3.2]{GZ13}).
\end{rem}

\begin{rem}
    In the proof of \cref{Q-minimize statement} and \cref{main symp sur}, we actually obstruct not only symplectic surfaces but more generally smoothly and properly embedded connected surfaces which satisfy the adjunction equality. 
\end{rem}

\subsection{Quasipositivity of knots}
We also give some new obstructions to a knot being quasipositive. 
Note that the quasi-positivity is equivalent to bounding a connected symplectic surface in a filling of $(S^3, \xi_{\operatorname{std}})$. Detecting quasipositivity for a given knot remains a challenging problem.

\footnote{Since it has already been proven in \cite{Or20} that if the connected sum $K=K_1 \# K_2$ of knots is quasipositive, then $K_i$ is quasipositive for $i=1, 2$. Thus we only consider irreducible knots. } In this paper, we focus on Montesinos knots written by 
\[
M( b, (a_1, b_1), \cdots, (a_n, b_n) ), \footnote{Our convention of Montesinos knots is based on \cite{BH24}. }
\]
which is obtained by joining together $n$ rational tangles with slopes $a_1/b_1, \cdots , a_n/b_n $ together with $b$ half-twists. 
The link $M( b, (a_1, b_1), \cdots, (a_n, b_n) )$ is a knot if and only if either exactly one of $\{a_i\}$ is even or all of $a_1, \cdots ,a_n, b+ \sum b_i$ are odd. 
Quasipositive fibered Montesinos knots have been already classified in \cite[Proof of Theorem 2]{KA18I} combined with \cite{He10}. 
However, as posed in \cite[Question 6]{KA18I}, the question remains unanswered for non-fibered Montesinos knots, despite several partial results have been known (\cite{Ba21}, \cite[Section 5.2]{Stoi}). 
We give a new obstruction for non-fibered Montesinos knots. 
\begin{thm}\label{main Montesinos}
Let $K=M( b, (a_1, b_1), \cdots, (a_n, b_n) )$ be a Montesinos knot such that exactly one of $\{a_i\}$ is even. If 
\begin{align}\label{Mon eq}
g_4(K) + \frac{1}{2}\sigma(K) \geq 2, 
\end{align}
then $K$ is not smoothly concordant to any quasipositive knot. 
\end{thm}

\subsection{New concordance invariants and adjunction equalities}
Next, we shall discuss the slice torus invariant derived from Baraglia--Hekmati's equivariant Seiberg--Witten theory. 
We define a knot invariant $q_M  ( K ) \in \Z$
by 
\[
q_M(K) := \min \{ \gr^\Q ( x) |    x \in \wt{H}^*_{\Z_2} (SWF(-\Sigma_2(K)); \mathbb{F}_2): \text{homogeneous  and } Q^n x \neq 0  \text{ for all }n\geq 0\} - \frac{3}{4} \sigma(K), 
\]
where $SWF(\Sigma_2(K)))$ is $\Z_2$-equivariant Floer homotopy type introduced by Baraglia--Hekmati. 
The invariant  $q_M(K)$ is a Seiberg--Witten analog of Hendricks--Lipshitz--Sarkar's $q_\tau$-invariant introduced in \cite{HLS16} with a signature correction.
The following are the fundamental properties of $q_M$.
\begin{thm}\label{main slice-torus}
The invariant $2 q_M$ is a slice-torus invariant introduced in \cite{Li04,Le14}, i.e. 
   \begin{itemize}
   \item[(i)] $q_M$ is a smooth concordance invariant, 
       \item[(ii)] $q_M(K_1\# K_2  ) = q_M(K_1) + q_M (K_2)$,  \item[(iii)] $q_M(K) \leq    g_4(K) $
       \item[(iv)] $q_M(T(p,q)) = \frac{1}{2}(p-1) (q-1)$. 
   \end{itemize}
For any knot $K$ in $S^3$, we have 
\[
m(-K)-\frac{3}{4}\sigma(K) \leq q_M(K),
\]
where 
\begin{align}\label{mK}
m(K):=\min \{ gr^\Q (x) | 0\neq x\in  \widehat{HF}(\Sigma_2(K); \fraks_0), x \text{ is homogeneous}  \}. 
\end{align}
   Moreover, if the double-branched cover of $K$ is L-space, then 
   \[
   q_M(K) = -\frac{1}{2}\sigma(K).
   \]
   In particular, this holds for quasi-alternating knots. 
\end{thm}
 Baraglia and Hekmati \cite{BH2} use spectral sequences and computations of Heegaard Floer homology to calculate their invariants $\theta^{(p)}(K)$ for torus knots and give an alternative proof of the Milnor conjecture. Instead, we utilize our transverse knot invariant, which is essentially the $\mathbb{Z}_2$-equivariant version of the contact invariant introduced in the authors' prior work \cite{IT20}, to determine $q_M$ for torus knots.  In particular, this gives an alternative proof of the Milnor conjecture by $\Z_2$-equivariant monopole Floer homology (rather than $ S^1 \times \Z_2$ equivariant Floer cohomology).
 This computation of $q_M$ for torus knots can be regarded as a special case of the result for symplectic surfaces \cref{main theo:general} that we will state later. The invariant $q_M(K)$ can be extended to oriented links $q_M(L)$ with non-zero determinants and $q_M(L)$ is invariant under oriented $\chi$-concordance introduced in \cite{DO12}. See \cref{qMforlink} for more details. 
 \par
 We also have the following relation with Baraglia's $\theta^{(2)}$-invariant.
     \begin{thm}\label{qmtheta}
         For a knot $K$ in $S^3$, we have 
         \[
         q_M(K) \leq \theta^{(2)} (K). 
         \]
         Therefore, $ q_M(K)$ gives a lower bound of the smooth H-slice genus in a negative definite 4-manifold with $S^3$ boundary and with $H_1(X;\Z)=0$.
     \end{thm}

     \begin{rem}
         It is pointed out in \cite{BH} that $\theta^{(p)} (K)$ has analogous properties with the nu-invariant in Heegaard knot Floer theory \cite{OS10I}. In this viewpoint, $q_M(K)$ is formally similar to Ozv\'ath--Szabo's $\tau$-invariant $\tau(K)$ \cite{OS03}. Similar variants were found in instanton theory \cite{BS21, DISST22}. 
   Moreover, many constructions of slice-torus invariants have been known so far from different theories including Heegaard Floer theory, Khovanov homology theory, and instanton Floer theory.  See \cite{OS03, Wu09, BS21, DISST22, SS22}. 
It is interesting to ask which slice-torus invariant coincides with $q_M(K)$.
\end{rem}

Our invariant $q_M(K)$ can be easily computed if there is a symplectic surface bounded by $K$. 
\begin{thm}\label{main theo:general} 
 Let $(X, \om)$ be a weak symplectic filling of $(S^3, \xi_{\operatorname{std}})$. Suppose a transverse knot $K$ in $(S^3, \xi_{\operatorname{std}})$ bounds a properly embedded connected symplectic surface $S$ in $X$ divisible by $2$. 
 Then, one has 
\begin{align}\label{main eq}
q_M ( K) =   g(S)   +\frac{1}{2} \langle c_1 (\om), [S] \rangle - \frac{1}{2}[S]^2
  \end{align}
  and 
  \begin{align}\label{main eq1}
q_M ( K) = -d_3 (\Sigma_2(K), \wt{\xi}) - \frac{1}{2} + \frac{3}{4} \sigma (K)=\frac{1}{2}\operatorname{sl}(K)+\frac{1}{2},  
  \end{align}
  For the last equality, see \cite{Pla06, It17}.
 \end{thm}
Several techniques to calculate slice-torus invariants have been developed (See \cite{Li04,Le14, FLL22}). Since our concordance invariant $q_M$ is a slice-torus invariant, one can make use of such a result. From such a viewpoint, Feller--Lewark--Lobb \cite{FLL22} recently introduced a notion of {\it squeezed knots}, which are knots in $S^3$ arise as cross-sections of a genus minimizing knot cobordism between torus knots. This class includes all quasipositive knots and alternating knots and they proved any slice-torus invariant coincides with the Rasmussen invariant for a squeezed knot.  Combined this with  \cref{main slice-torus}, one can give the following computations. 

\begin{thm}\label{Rasmussen}
    Let $K$ be a squeezed knot. The following hold: 
    \begin{itemize}
        \item[(i)] If $\Sigma_2(K)$ is an L-space, then we have 
        \[
          f(K) =  -\sigma(K), 
        \]
        for any slice-torus invariant $f:  \mathcal{C} \to \Z$. (Here, our convention of slice-torus invariants is $f(T(p,q))=(p-1)(q-1)=2g_4(T(p,q))$ for coprime positive $p, q$.)
        \item[(ii)] If $K$ is smoothly concordant to a Montesinos knot $M( b, (a_1, b_1), \cdots, (a_n, b_n) )$ such that exactly one of $\{a_i\}$ is even, then we have
        \[
       | f(K) + \sigma(K)|\leq 2
        \]
        for any slice-torus invariant $f$.

        \item[(iii)]  For any slice-torus invariant $f$, we have 
        \[
        2m(-K)-\frac{3}{2}\sigma(K) \leq f(K),
        \]
        where $m(K)$ is given in \eqref{mK}.
    \end{itemize}
\end{thm}
Since the instanton $\wt{s}$-invariant introduced in \cite{DISST22} is a slice-torus invariant and useful to study the homology cobordism group, we have the following application. 
\begin{cor}\label{squeezed Montesinos}
    Let $K$ be a squeezed knot satisfying one of the following conditions: 
    \begin{itemize}
        \item[(i)] the double branched covering space $\Sigma_2(K)$ is a L-space and $\sigma(K) \leq -2$, 
        \item[(ii)] $K$ is Montesinos knot $M( b, (a_1, b_1), \cdots, (a_n, b_n) )$ such that exactly one of $\{a_i\}$ is even and $\sigma(K) \leq -4$, 
        \item[(iii)] $K$ satisfies 
        \[
        0 <2m(-K)-\frac{3}{2}\sigma(K).  
        \]
    \end{itemize}
    Then, $\{S^3_{1/n}(K)\}_{n \in\Z_{>0}}$ is linearly independent in the homology cobordism group of homology 3-spheres. 
\end{cor}
Note that for alternating knots and quasipositive knots, there are similar types of results \cite{BS22, DISST22}. \cref{squeezed Montesinos} extends it to a bigger class of knots. 
For example, a squeezed non-alternating, non-quasipositive and Montesinos knot  $9_{43}=M(b=0, (5, 3), (3, 1), (2, -1))$ with $\sigma=-4$ satisfy the assumption of \cref{squeezed Montesinos}. See \cite{Sa22} for a survey of the homology cobordism group. 

\subsection{Several theoretical results}
In addition to developing the study of symplectic surfaces using equivariant Seiberg--Witten theory, we further develop equivariant Seiberg--Witten Floer homology in several directions. 
One representative result is described as follows. 
\begin{thm}\label{rank-1}
For any knot $K$ in $S^3$, we have 
    \[
\operatorname{rank}_{H^*(B\Z_2; \mathbb{F}_2) } \wt{H}^*_{\Z_2}(SWF(\Sigma_2(K)); \mathbb{F}_2) =1.
    \]
\end{thm}
The idea of the proof of \cref{rank-1} uses the techniques of calculations of singular Donaldson invariants for surfaces developed by Kronheimer in \cite{Kr97}. \cref{rank-1} can be seen as a Seiberg--Witten theoretical analog of a result given by Daemi and Scaduto \cite[Theorem 10]{DS23} for singular instanton Floer theory. Also,  \cref{rank-1} can be seen as an analog of Hendricks--Lipshitz--Sarkar's result  \cite[Theorem 1.22]{HLS16} combined with \cite[Theorem 3]{LiTr16} in Heegaard Floer theory. The proof of \cref{main slice-torus} relies on \cref{rank-1}. 

The method to prove \cref{rank-1} can also imply the following structural theorem of Baraglia--Hekmati's $S^1\times \Z_2$-equivariant Floer cohomologies. 
\begin{thm}\label{towers}
Let $K \subset S^3$ be a knot.
The set of $U$-torsions in 
$\wt{H}^* _{S^1\times \Z_2}(SWF(\Sigma_2(K));  \mathbb{F}_2)$
is finite.  
\end{thm}
The theorems \cref{rank-1} and \cref{towers} are the statements about the equivariant Floer cohomology. These results are proven by using the following structure theorem for the cobordism maps for the equivariant Floer cohomologies. 
\begin{thm}\label{vanishing BF}  Let $S$ be a connected, oriented, and embedded surface cobordism in a spin$^c$ cobordism $(W, \fraks)$ from a pair of a homology $3$-sphere $Y$ and a knot $K$ in Y to another pair $(Y', K')$ whose homology class is divisible by $2$. 
Then, the $Q$-localized cohomological Bauer--Furuta invariant for the pull-backed spin$^c$ structure\footnote{We need $\Z_p$-invariant spin$^c$ structure on the $p$-th cyclic branched cover branched along $S$ in order to define the equivariant Bauer--Furuta invariant. See \cref{pull back spinc} for the precise meaning of the pull-backed spin$^c$ structure. }
\[
 BF^{*, Q\text{-loc}}_{(W, S)} :Q^{-1}\wt{H}^*_{S^1\times \Z_2} (SWF(\Sigma_2(K'))) \to Q^{-1} \wt{H}^*_{S^1\times \Z_2} (SWF(\Sigma_2(K)))  
\]
depends only on the homotopy class of $S$, precisely, for a pair of homotopic surfaces $S$ and $S'$ rel boundary, we have 
\[
 BF^{*, Q\text{-loc}}_{(W, S)}  =  BF^{*, Q\text{-loc}}_{(W, S')}. 
\]
Here $Q$ is the variable corresponding to the module structure of $\wt{H}^*_{S^1\times \Z_2} (SWF(\Sigma_2(K)))$ over $H^*(B(S^1\times \Z_2); \mathbb{F}_2)\cong  H^*(BS^1; \mathbb{F}_2) \otimes H^*(B \Z_2; \mathbb{F}_2)\cong \mathbb{F}_2[U] \otimes \mathbb{F}_2[Q]\cong \mathbb{F}_2[U, Q]$. 
In addition, the $S^1\times \Z_p$-equivariant stable homotopy Bauer--Furuta invariant for 2-knot in $S^4$
\[
BF_{(S^4, S)} : S^V \to S^{V} 
\]
is stably $S^1\times \Z_p$ homotopic to the  identity map up to sign, where $V$ is a suitable $S^1\times \Z_p$-representation space. 
\end{thm}
There are analogous results in singular Donaldson invariants \cite{Kr97} and Khovanv homology theory \cite{Ra05, Ta06}.   
\par
One of the main tools to prove \cref{main symp sur} is a  homological invariant 
\[
c_{(2)}(Y, \xi, K) \in \wt{H}^{S^1\times \Z_2}_* (SWF(-\Sigma_2(K)); \mathbb{F}_2)
\]
obtained by applying equivariant homology and the Borel construction to  \eqref{homotopy contact eq} for a given transverse knot $K$ in a contact homology 3-sphere $(Y, \xi)$.  
The following is a basic property of the invariant $c_{(2)}(S^3, \xi, K)$ used in the proof of \cref{main symp sur}: 
\begin{thm}\label{c in Utower:intro}
 For any transverse knot $K$ in $ S^3$ with any contact structure $\xi$, the equivariant transverse knot invariant $c_{(2)}(S^3, \xi, K)$ lies in the $U_\dagger$-tower, i.e. 
\begin{align}\label{c in U-tower}
c_{(2)}(S^3, \xi, K) \in \bigcap_{i \geq 0} \operatorname{Im} U^i_\dagger, 
 \end{align}
 where $U_\dagger$ is the induced degree-$(-2)$ module structure comes from the $S^1$-action on $SWF(-\Sigma_2(K))$.
\end{thm} 
In the standard (non-equivariant) contact invariant within monopole Floer homology or Heegaard Floer homology, the phenomenon described by \eqref{c in U-tower} appears to be rare. For instance, refer to \cite[Theorem 5.1]{MT22} for computations concerning Seifert homology 3-spheres.

Let us also provide some results on Baraglia's  knot concordance invariant
\[
\theta^{(p)} (K) \in \frac{1}{p-1}\Z_{\geq 0}
\]
for each prime number $p$. 
These invariants give a lower bound of the slice genus \cite{Ba22} and it was used to give an alternative proof of the Milnor conjecture \cite{BH2}.  
We prove a Bennquin-type inequality for $\theta^{(p)}$ which helps us calculating the invariants $\theta^{(p)}$. 

We state a Bennequin-type inequality for transverse knots in the standard contact 3-sphere. 
\begin{thm}\label{BP inequality}
Let $p$ be a prime number. 
Let $K$ be a transverse knot $(S^3 , \xi_{std})$. Then, we have 
\[
 \operatorname{sl} (K)  \leq 2\theta^{(p)}(K)-1 , 
\]
where $\operatorname{sl}(K)$ denotes the self-linking number of $K$. 
Moreover, the equality holds if $K$ is quasipositive. 
\end{thm}

\begin{rem}
By the standard technique of Legendrian push-off, the above inequality implies the Bennequin-type inequality for Legendrian knots. See \cref{Legendrian}. 
\end{rem}

For quasipositive knots, 
our lower bound for $\theta^{(p)}$  and upper bound given by the  4-ball genus coincide.
Thus we can determine the value of $\theta^{(p)}$ for quasipositive knots:
\begin{cor}\label{quasi theta}
    For any quasipositive knot $K$, we have 
    \[
    \theta^{(p)}(K)= g_4(K).
    \]
for any prime number $p$. 
    \qed
\end{cor}

\subsection{Symplectic filling structure on cyclic branched covers}
As the final result in the introduction, we state a result purely described in symplectic topology. 
In the computations of our transverse knot invariant, we shall make use of invariant symplectic structures on the branched covers. 
For this purpose, we prove a result on the existence of invariant symplectic structures on the branched covering spaces when the covering spaces are branched along properly embedded symplectic surfaces in symplectic fillings. 
For closed surfaces, this kind of result has been already known. See \cite{Go98, Wa97, Au00,Au05}. 

\begin{thm}\label{double branch symp}
Let $n\geq 2$ be an integer.
Let $(Y, \xi)$ be a contact integer homology 3-sphare and let $(X, \omega)$ be a weak symplectic filling of $(Y, \xi)$ with 
$H_1(X; \Z)=0$.
Let $S \subset (X, \omega)$ be a properly embedded connected symplectic surface satisfying $PD[S] \in n H^2(X; \Z)$ and $K:=\partial S$ is non-empty and connected. 
Let $\pi: \Sigma_n(S)\to X$ be the $\Z_n$-branched covering space along $S$. Then there exists a $\Z_n$-invariant symplectic form  $\wt{\omega}$ on $\Sigma_n(S)$  with $[\wt{\omega}]=\pi^* [\omega] \in H^2(\Sigma_n(S);\R)$ and 
\[
c_1(\wt{\omega})=\pi^*c_1(\omega)+(1-n) PD[\pi^{-1}(S)] \in H^2(\Sigma_n(S)).
\]
such that $(\Sigma_n(S), \wt{\omega})$ is a weak symplectic filling of $(\Sigma_n(K), \wt{\xi})$.

\end{thm}

\begin{acknowledgement}
We would like to extend our deep gratitude to Marco Golla for finding mistakes in earlier drafts on multiple occasions and for providing several examples of symplectic surfaces. 
We also would like to thank David Baraglia for pointing out several errors in an earlier draft. 
Our appreciation also goes to Hirofumi Sasahira and Matthew Stoffregen for sharing the draft of their forthcoming paper. We are thankful to Hisaaki Endo for bringing to our attention a paper about symplectic structures on branched covering spaces and for telling us various examples of branched covering spaces. 
Additionally, we express our gratitude to Tye Lidman, Mathew Hedden, Ko Honda, Tetsuya Itoh, and Mikio Furuta for their invaluable contributions to our discussions. Part of this research was conducted during the "Floer homotopy theory" program hosted at MSRI/SLMath in Fall 2022. The first author acknowledges support from JSPS KAKENHI Grant Number 22J00407, while the second author acknowledges partial support from JSPS KAKENHI Grant Numbers 20K22319, 22K13921, and the RIKEN iTHEMS Program.
\end{acknowledgement}

\section{Symplectic/contact structures on branched covering spaces} 

We first treat $\Z_n$-invariant symplectic structures on the cyclic $n$-th branched covering spaces along symplectic surfaces in symplectic fillings. 
\subsection{Preliminaries}
We state a fundamental result about surfaces in strong symplectic fillings. 
\begin{lem}\label{symp surf boundary is trans}
Let $(Y, \xi)$ be a closed contact 3-manifold and $(X, \omega)$ be a strong symplectic filling of $(Y, \xi)$.
Let $S\subset (X, \omega)$ be a properly embedded symplectic surface.
Then if we suitably deform $\omega$ near $Y$ up to isotopy, $K:=\partial S$ becomes a transverse link.
\end{lem}
\begin{proof}
From \cite[Lemma 2.1]{Et20}, by deforming the symplectic form up to isotopy, one can suppose there is a Liouvile vector field $v$ of $\om$ near $Y$ such that $v|_S$ is also a Liouvile vector field with respect to $\om|_S$.
The result follows because $v_p$ and $\dot{K}:=T_p (\partial S) $ are linearly independent at each point  $p \in K$ and thus
\[
(\iota_v(\om|_{S}))(\dot{K})   >0.
\]
\end{proof}
\subsection{Contact structures on $\Sigma_n(K)$}\label{Contact structures on branched coverring spaces}

In this paper, we mainly treat knots in $S^3$,
but the construction of contact structures on branched covering spaces works for every null-homologous knot in 3-manifolds. Therefore, we record a general argument here. Also, we work with $p$-th branched covering spaces for primes but the argument can be done with any $n$-th branched covering space for any given integer $n \geq 2$. We also note the link case is completely the same as the knot case.

Let $(Y, \xi)$ be a contact closed 3-manifold and $K$ be a null-homologous transverse knot in $(Y, \xi)$ and $\lambda$ be a positive contact form for $\xi$.
Let $n \geq 2$ be an integer.
(Later, in order to ensure that the branched covering is a rational homology 3-sphere, we assume $n$ is prime.)

We denote by 
\[
\pi: \Sigma_n (K) \to Y 
\]
the $n$-fold branched cover along $K$.
We have the standard covering $\Z_n$-action 
\[
\tau: \Sigma_n(K)\to \Sigma_n(K), \quad \tau^n= id.
\]
We shall do this construction taking into account the contact structure. The following construction of the contact structure on the branched cover $\Sigma_n(K)$ is standard. See \cite{Gonz, Pl06, HKP09, Ka18}.

Since a transverse knot is a contact submanifold, we can apply the contact neighborhood theorem to 
take a tubular neighborhood $N_K= S^1\times D^2$ of $ Y$ such that we have a contact morphism
\[
\xi|_{S^1\times D^2} \cong \ker \lambda =\ker (d \phi + r^2 d \theta), 
\]
where $\phi$ and $(r, \theta)$ denote the coordinate of $S^1$ and the polar coordinate of $D^2$ respectively. 
Define $\wt{N}_K := \pi^{-1}(N_K) = S^1\times D^2 $. 
Then, the projection $\pi|_{\wt{N}_K} :S^1\times D^2 = \wt{N}_K  \to N_K = S^1\times D^2 $ can be written as
\[
\pi: (\phi, r, \theta) \mapsto (\phi, r^n, n\theta). 
\]
(Strictly speaking, we should say that the branched cover is constructed using these coordinates.)
Then one has 
\[
\pi^* \lambda  = d \phi + n \cdot r^{2n} d \theta.
\]

Now we define the global 1-form $\lambda _f$ on $\Sigma_n(K)$ by the following way: 
\begin{align}\label{contact form on cov}
\lambda_f : = \begin{cases}
d \phi + f(r)  d \theta \text{  on  }\wt{N}_K= S^1\times D^2  \\ 
\pi^* \lambda  \text{  on  } \Sigma_n (K) \setminus \wt{N}_K, 
\end{cases} 
\end{align}
where $f : [0 , 1]\to [0 , \infty)$ is a strictly increasing smooth function such that $f(r) = r^2$ near $r = 0$ and $f(r) = n\cdot r^{2n}$ near $r = 1$. Then, the isotopy class of $\lambda_f$ does not depend on the choices of $f$. See \cite{Pl06} for more details on uniqueness.  We denote this contact structure by $\wt{\xi}$. Note that we can take a generator of the branched covering transformations $\tau$ so that 
\[
\tau |_{\wt{N}_K} : S^1\times D^2 \to S^1\times D^2 
\]
is written by 
\[
\tau(\phi, r  , \theta) = \left(\phi, r  , \theta+ \frac{2\pi}{n}\right). 
\]
Note that $\tau$ preserves the 1-form $\lambda_f$ on $\wt{N}_K$. Therefore, one has an isotopy class of a contact structure $\wt{\xi}$ on $\Sigma(K)$ satisfying the following conditions: 
\begin{itemize}
    \item[(i)] $\tau_* \wt{\xi}=\wt{\xi}$; 
    \item[(ii)] $\pi_*(\wt{\xi}|_{\pi^{-1}(Y\setminus K)} )  \cong  \xi|_{Y\setminus K}$. 
\end{itemize}
 
 Moreover, we have the following invariance:
 \begin{lem}[\cite{Pla06}]
 Let $K$ and $K'$ be transverse knots in $(Y, \xi)$. Suppose $K$ and $K'$ are transverse isotopic. 
 Then as $\Z_n$-equivariant contact manifolds, we have
 \[
 (\Sigma_n (K), \wt{\xi} ) \cong  (\Sigma_n (K'), \wt{\xi}' ). 
 \]
 \end{lem}
 
The following computation of  $d_3$-invariants is confirmed in \cite{It17}:
\begin{align}\label{Itoh formula}
d_3(\Sigma(K), \wt{\xi})  =  -\frac{3}{4}\sum_{w: w^n=1} \sigma_w(K)-\frac{n-1}{2}  \operatorname{sl}(K)-\frac{1}{2}n
\end{align}
for every transverse knot in $(S^3, \xi_{std})$. We shall use it in the computation of the grading of our equivariant contact invariant. 

Also, we observe which kinds of spin$^c$ structures appear as $\wt{\xi}$ on the branched covering spaces.
Recall that for an oriented 2-plane field $\xi$ on a closed oriented 3-manifold $Y$, an isomorphism class of spin$^c$ structure $\s_\xi=(S_\xi, \rho_\xi)$ is determined.
The spinor bundle $S_\xi$ is given by
\[
S_\xi=\underline{\C}\oplus \xi
\]
and thus
\[
c_1(\s_\xi)=c_1(\xi) \in H^2(Y; \Z )
\]
holds, where we regard $\xi$ as a complex line bundle on $Y$.
The following observation is a generalization of \cite[Lemma 1]{Pl06} for the case of general cyclic $n$-th branched coverings. 

\begin{prop}\label{spinness of contact str}
Let  $n$ be an integer greater than 1 (not necessarily a power of a prime number) and $Y$ be an oriented homology $3$-sphere. 
We also fix a knot $K$ in $Y$.
The induced spin$^c$ structure $\mathfrak{s}_{\wt{\xi}}$ satisfies the following equality
\[
c_1(\mathfrak{s}_{\wt{\xi}}) =0 \in H^2(\Sigma _n(K); \Z ), 
\]
where $\wt{K} = \pi^{-1}(K)$.
In particular, $\mathfrak{s}_{\wt{\xi}}$ is spin and invariant under $\Z_n$-action. 
\end{prop}
\begin{proof}
Let $E=Y\setminus \overset{\circ}{N}(K)$ be the knot exterior and $E_n \to E$ be the $\Z_n$ covering.
We write the $n$-th branched covering space $\Sigma_n(K)$ as the union $E_n \cup_{S^1\times S^1} S^1 \times D^2$. 
First, we can check $H^1(E_n; \Z)=\Z $, by decomposing $E_n$ into $n$ pieces along $n$ copies of Seifert surfaces and applying the Mayer-Veitoris long exact sequence inductively.
For example, when $n=2$, let us denote by
\[
E_2=E_{(0)} \cup E_{(1)}
\]
a decomposition along two copies of Seifert surfaces.
Now the Mayer-Veitoris long exact sequence 
\[
0 \to H^0(E_2) = \Z \to H^0(E_{(0)} \amalg E_{(1)}) = \Z \oplus \Z \to H^0(F \amalg F) = \Z \oplus \Z \to H^1(E_2) \to H^1(E_{(0)} \amalg E_{(1)})=  0 
\]
 shows $H^1(E_2; \Z )\cong \Z$.
The general $n$ case is similar.
\par
Next, we will show that the restriction map
\[
H^2(\Sigma_n(K))\to H^2(E_n)
\]
is injective, using the Mayer-Vietoris long exact sequence of $\Sigma_n(K)=E_n \cup D^2\times S^1$.
This injectivity implies
$c_1(\mathfrak{s}_{\wt{\xi}}) = \pi^* c_1 ( \s_\xi )=0 \in H^2(\Sigma_n (K))$, since its restriction to $ H^2(E_n)$ is the same.
The Mayer-Vietoris long exact sequence of $\Sigma_n(K)=E_n \cup D^2\times S^1$ : 
\[
\to H^1(E_n)\oplus H^1(D^2\times S^1)\xrightarrow{\cong} H^1(S^1\times S^1)\to H^2(\Sigma_n(K))\to H^2(E_n) \oplus H^2(D^2\times S^1)\to 
\]
can be computed as 
\[
\to\Z\oplus \Z\to \Z\oplus \Z \xrightarrow{0} H^2(\Sigma_n(K)))\to H^2(E_n) \oplus 0\to .
\]
This implies the injectivity that we wanted to show.
\end{proof}

\subsection{Symplectic structures on branched covering spaces }
In this section, we will prove the following:
\begin{thm}[\cref{double branch symp}]
Let $n\geq 2$ be an integer.
Let $(Y, \xi)$ be a contact integer homology 3-sphare and let $(X, \omega)$ be a weak symplectic filling of $(Y, \xi)$ with 
$H_1(X; \Z)=0$.
Let $S \subset (X, \omega)$ be a properly embedded symplectic surface in a compact symplectic 4-manifold satisfying $PD[S] \in n H^2(X; \Z)$ and $K:=\partial S$ is non-empty and connected, 
where $PD$ denotes the Poincar\'e duality. 
Let $\Sigma_n(S)\to X$ be the $\Z_n$  branched covering. 

Then there exists a $\Z_n$ invariant symplectic form  $\wt{\omega}$ on $\Sigma_n(S)$  with $[\wt{\omega}]=\pi^* [\omega] \in H^2(\Sigma_p(S);\R)$ and 
\[
c_1(\wt{\omega})=\pi^*c_1(\omega)+(1-n) PD[\wt{S}] \in H^2(\Sigma_n(S))
\]
such that $(\Sigma_n(S), \wt{\omega})$ is a weak symplectic filling of $(\Sigma_n(K), \wt{\xi})$, where $\wt{S}$ is the inverse image of S under the covering projection $\pi$.

\end{thm}
Again it is not hard to generalize this theorem to the link case without any change.

We use the following symplectic neighborhood theorem adapted to the surface with a boundary.
\begin{lem}\label{symplectic neighborhood}
Let $(Y, \xi)$ be a contact integer homology 3-sphere and let $(X, \omega)$ be a weak symplectic filling of $(Y, \xi)$.
Let $S \subset (X, \omega)$ be a properly embedded connected symplectic surface in a compact symplectic 4-manifold, and $K:=\partial S$ is non-empty.
Then there exist a tubular neighborhood $N_S$, its trivialization $N_S \cong S \times D^2$ and a symplectomorphism
\[
(N_S, \omega)\cong (S\times D^2, \omega|_S+\omega|_\text{fiber}).
\]
\end{lem}

\begin{proof}
This is a simple adaptation of the usual proof of the symplectic neighborhood theorem using the Moser method.
See \cite[Theorem 3.4.10]{MS17}, for example.
\end{proof}
Now let us prove \cref{double branch symp}
\begin{proof}[Proof of \cref{double branch symp}]
The following proof is based on Gompf's construction for closed symplectic manifolds \cite{Go98}.
In our case, we need to deal with the boundaries of the symplectic 4-manifold $X$ and the symplectic surface $S$ in addition. The assumption $\partial \neq \emptyset$ makes the argument simpler since this ensures that the normal bundle of $S$ is trivial. See also related constructions \cite{Wa97, Au00,Au05}. 

First, without loss of generality, using isotopy, we may assume $\om$ is a strong symplectic filling of $\wt{\xi}$ \cite{OO99}.  
Moreover, from \cref{symp surf boundary is trans}, again using isotopy, we can assume $\partial S =K$ is a transverse knot in $(Y, \xi)$. In the rest of the proof, we assume these two facts. 
As in \cref{Contact structures on branched coverring spaces}, we equip a $\Z_n$-invariant contact structure $\wt{\xi}$ on $\Sigma_n(K)$.

We first fix the following data: 
\begin{itemize}
    \item A tubular neighborhood $N_K= S^1\times D^2$ of $K$ in $Y$ such that 
\[
\xi|_{S^1\times D^2} \cong \ker \lambda=\ker (d \phi + r^2 d \theta), 
\]
where $\phi, r, \theta$ are the same coordinates as before. 
\item 
A tubular neighborhood $N_S$ of $S$ in $X$ and its identification with normal bundle $\nu_S$ of $S \subset X$.
Since we assume $\partial S $ is nonempty, we can take a trivialization of the normal bundle.

\item 
 A coordinate of the cyclic $n$-th branched cover so that the covering projection $\pi:S\times D^2 = N_{\wt{S}}  \to N_S = S\times D^2 $ is described as 
\[
\pi: (b, z) \mapsto (b, z^n), 
\]
for $b \in S$ and $(r, \theta) =z \in D^2$.
Now the $\Z_n$-action can be identified with $\Z_n\subset S^1\subset \C^\times$ action on the fiber of the unit disk bundle of the normal bundle $\C\to \nu_S \to S$.

Moreover, by symplectic neighborhood theorem \cref{symplectic neighborhood} above, we can  identify
\[
(N_S, S, \omega)\cong (S\times D^2 , S\times \{0\} , \om|_{S} +\frac{i}{2}dz \wedge d\bar{z}   )
\]
by a symplectomorphsim and an isotopy.
By composing the  contactomorphism obtained as the restriction of this symplectomorphism to the boundary, we furthermore have a contactomorphism
\[
\xi|_{S^1\times D^2} \cong \ker \lambda=\ker (d \phi + r^2 d \theta). 
\]
\item 
Fix a smooth non-negative  function
$\beta(r)$ ($r \geq 0$) which is $1$ on $r>1/3$ and $0$ near $r=0$.
Regard this as a function on both $X$ and $\Sigma_n(S)$ near the neighborhoods of $S$ in $X$ and $\pi^{-1} S$ in $\Sigma_n(S)$ respectively under the identification above and extend it by 1 outside of the neighborhoods.

\item 
Fix a smaller tublar neighborhood $N'_{\wt{S}}=\wt{S}\times D^2(\frac{1}{2}) \subset \pi^{-1} (N_S)$ of $\wt{S}$ on $\Sigma_n(S)$, and its open  covering 
\[
N'_{\wt{S}}=U_{\operatorname{int}} \cup U_\partial
\]
where $D^2(\frac{1}{2})$ is the two dimensional disk of radius $\frac{1}{2}$ and 
\[
U_\partial=S^1_\phi \times (1, 2]_s \times D^2\left(\frac{1}{2}\right) \subset \wt{S}\times D^2\left(\frac{1}{2}\right) = N'_{\wt{S}}
\]
is a neighborhood of $\partial S$ in $N'_{\wt{S}} $ and $(1, 2]_s $ is a coordinate of normal direction to the boundary $\partial \Sigma_n(S)$. The another open set $U_{\operatorname{int}}$ is chosen so that $\{ U_{\operatorname{int}}, U_\partial \}$ is an open covering of $N'_{\wt{S}}$ and $U_{\operatorname{int}} \cap \partial \Sigma_n (S) = \emptyset$.

By taking the intersection with $\wt{S}$, the covering above gives an open cover 
\[
\wt{S}=(\wt{S} \cap U_{\operatorname{int}}) \cup (\wt{S} \cap U_{\partial}).
\]
Fix a smooth partition of unity $\rho_{\operatorname{int}}, \rho_{\partial}$ for this covering $\{ \wt{S} \cap U_{\operatorname{int}} , \wt{S} \cap U_{\partial}\}$
and, by pull-back by the projection $\wt{S}\times D^2\to D^2$ and multiplying $1-\beta(r)$, extend $\rho_{\operatorname{int}}, \rho_{\partial}$ as the functions to $\Sigma_n(S)$.
\end{itemize}

Define several differential forms as follows: 
\begin{align*}
\Lambda&=\beta(r) d\theta  \text{ on } N_{\wt{S}} \\
\tau&=d\Lambda=\beta'(r)dr\wedge d\theta \text{ on } \Sigma_n (S)  \\
\delta_{\operatorname{int}} & =\frac{1}{2}r^2d\theta -\Lambda  \text{ on } U_{\operatorname{int}}  \\
\delta_{\partial}&=\frac{1}{2} s^2 \lambda_f-\Lambda  \text{ on } U_{\partial}\\
\eta&=\tau+d(\rho_{\operatorname{int}} \delta_{\operatorname{int}})+d(\rho_{\partial} \delta_{\partial}),
\end{align*}
where $\lambda_f$ is a contact form introduced in \eqref{contact form on cov}. 
Finally, for $t \in \R$, set 
\[
\wt{\omega}_t=\pi^* \omega+t \eta.
\]
We will see that for a sufficiently small $t>0$ this gives a desired symplectic form.
Notice that this is obviously a closed 2-form and $\Z_n$-invariant.
Thus it is enough to show
\begin{enumerate}
\item
\[
\wt{\omega}_t\wedge \wt{\omega}_t>0
\]
\item 
\[
\wt{\omega}_t|_{\wt{\xi}}>0
\]
for all sufficiently small $t>0$.
\end{enumerate}
First, let us show 
\[
\wt{\omega}_t\wedge \wt{\omega}_t>0.
\]
Since $\wt{\omega}_t=\pi^* \omega$ outside $N_{\wt{S}}$, it is enough to check
$\wt{\omega}_t\wedge \wt{\omega}_t>0$ on $N_{\wt{S}}$.
On $N_{\wt{S}}$ we have
\[
\wt{\omega}_t\wedge \wt{\omega}_t
=\pi^*\omega\wedge \pi^*\omega+2t (\pi^*\omega)\wedge \eta+t^2 \eta \wedge \eta.
\]

For sufficiently small $t>0$, it is enough to check 
$(\pi^*\omega)\wedge \eta>0$ on $\wt{S}$ since the term $t^2 \eta \wedge \eta$ has  order $t^2$ as $t \to 0$.
Since 
$\pi^*\omega|_{\wt{S}}=0$ on $\wt{S}$, it is enough to show $\eta|_{(\nu_{\wt{S}})_b = (N_{\wt{S}})_b }>0$ for every $b \in \wt{S}$.
Near $\partial S$, we have $\eta=\frac{1}{2}d(s^2 \lambda_f)$ so it is clear.
\par
Now let us denote by $\iota: D^2 = (N_{\wt{S}})_b \to N_{\wt{S}}$  an inclusion of a fiber.
Outside the neighborhood of $\partial S$, we have
\begin{align*}
\iota^*\eta&=\iota^*[\tau+d(\rho_{\operatorname{int}} \delta_{\operatorname{int}}))+d(\rho_{\partial} \delta_{\partial})]\\
&=\rho_{\operatorname{int}} \iota^*(\tau+d\delta_{\operatorname{int}})+\rho_{\partial} \iota^*(\tau+d\delta_{\partial})>0
\\
&=\rho_{\operatorname{int}} \iota^*(\frac{1}{2}d(r^2d\theta))+\rho_{\partial} \iota^*(\frac{1}{2}d(s^2\lambda_f))>0.
\end{align*}
Here we use the fact that $\rho_{\operatorname{int}}, \rho_{\partial}$ is constant along the fiber direction near $\wt{S}$ and thus $\iota^* d\rho_{\operatorname{int}}=\iota^* d\rho_{\partial}=0$.
\par

Second, let us check \[
\wt{\omega}_t|_{\wt{\xi}}>0.
\]
On $\Sigma_n (K) \setminus N_{\wt{K}}$, we have 
\[
\wt{\omega}_t=\pi^*\omega
\]
\[
\wt{\lambda}=\pi^*\lambda
\]
and thus
this follows from the assumption that $\omega$ is a weak filling of $\lambda$.
So it is enough to show $\wt{\omega}_t|_{\wt{\xi}}>0$ on $N_{\wt{K}}$.
Recall that we wrote
\begin{align*}
\wt{\xi}&=\ker(d\phi+f(r)d\theta)\\
&=\ker(d\phi+r^{-2}f(r)(x dy-ydx))\\
& =\operatorname{Span}\{e_1, e_2\},
\end{align*}
where $(x, y)$ denotes the coordinate of $D^2$ of $S^1 \times D^2 = N_{\wt{K}}$.  
Here we set 
\[
e_1=\partial_x +r^{-2}f(r)y \partial_\phi, \quad e_2=\partial_y-r^{-2}f(r)x\partial_\phi
\]
It is enough to show $\wt{\omega}_t(e_1, e_2)>0$.
On $N_{\wt{S}} \cap N_{\Sigma_n(K)}$, we have 
\begin{align*}
\eta&=\tau+d(\rho_{\partial} \delta_\partial)\\
&=\tau+(d\rho_\partial)\wedge \delta_\partial+\rho_\partial d\delta_{\partial}\\
&=(1-\rho_\partial)\tau+(d\rho_\partial)\wedge \delta_\partial+ \frac{\rho_\partial}{2}d(s^2\wt{\lambda})\\
&=\Omega+\frac{\rho_\partial}{2}d(s^2\wt{\lambda}).
\end{align*}
Here we set the positive basis of the plane field as
\[
\Omega=(1-\rho_\partial)\tau+(d\rho_\partial)\wedge \delta_\partial.
\]
Now we have
\begin{align*}
\wt{\omega}_t(e_1, e_2) &=(\pi^*\omega)(e_1, e_2)+t\Omega(e_1, e_2)+ \frac{t\rho_\partial}{2}d(s^2\wt{\lambda})(e_1, e_2)\\
&=n^2r^{2(n-1)}+t\Omega(e_1, e_2)+2t \rho_{\partial}
\end{align*}
Here we use $s=2$ on $\partial S$
and 
\begin{align*}
&(\pi^*\omega)(e_1, e_2)
\\
&=[\pi^*\omega|_S+n^2 r^{2(n-1)}dx \wedge dy](e_1, e_2)\\
&=n^2 r^{2(n-1)}dx \wedge dy (\partial_x +r^{-2}f(r)y \partial_\phi, \partial_y-r^{-2}f(r)x\partial_\phi)\\
&=n^2r^{2(n-1)}.
\end{align*}
For sufficiently small $t>0$, we have $\wt{\omega}_t|_{\wt{\xi}}>0$ outside the region with $\rho_\partial=1$.
On the region with $\rho_\partial=1$, we have $\Omega=0$, thus we have seen $\wt{\omega}_t|_{\wt{\xi}}>0$.
\par
Finally, our construction of the symplectic forms on the branched covering spaces is analogous to  \cite[Lemma 1]{Go98}, and $c_1(\wt{\omega})=\pi^*c_1(\omega)+(1-n) PD[\Sigma] $ is proven there. So, the same proof can work.
This completes the proof.
\end{proof}

\section{Rank theorem for equivariant Floer homology}
In this section, we recall the construction of Manolescu's Seiberg--Witten Floer stable homotopy type \cite{Man03}, and its equivariant version due to Baraglia--Hekmati \cite{BH}.
Let $p$ be a prime number and  $\wt{Y}$ be  a rational homology $3$-sphere.
Suppose $\wt{Y}$ is equipped with a \Spinc structure $\s$ and a smooth orientation preserving $\Z_p$-action $\tau$ and the isomorphism class of $\s$ such that 
\[
\tau^* \s \cong \s. 
\]
Let us call such a pair $(\wt{Y}, \s)$ a
$\Z_p$-Spin$^c$ rational homology 3-sphere. 
For a $\Z_p$-Spin$^c$ rational homology 3-sphere $(\wt{Y}, \s)$,  we have a 
$H^*(B(S^1\times \Z_p); \mathbb{F}_p)$ module 
\[
\wt{H}^*_{S^1\times \Z_p}(SWF(\wt{Y}, \s); \mathbb{F}_p)
\]
and a
$H^*(B\Z_p; \mathbb{F}_p)$ module 
\[
\wt{H}^*_{\Z_p}(SWF(\wt{Y}, \s); \mathbb{F}_p).
\]
By applying this to a cyclic branched covering along a knot, we obtain Seiberg--Witten Floer stable homotopy type for a knot in the 3-sphere.
When we consider the cyclic $p$-th branched coverings for a prime number $p$, this homotopy type has $S^1\times \Z_p$-action. 
Its $S^1\times \Z_p$-equivariant cohomology group was studied by Baraglia and Hekmati \cite{BH, Bar, BH2}, using a spectral sequence relating this $S^1\times \Z_p$ equivariant Floer cohomology group with the $S^1$-equivariant Floer cohomology group, and the latter can be computed by using techniques from Heegaard Floer homology and the isomorphism between $S^1$-equivariant Seiberg--Witten Floer cohomology group and Heegaard Floer homology. See \cite{KLTI, KLTII, KLTIII, KLTIV, KLTV, Ta10I, Ta10II, Ta10III, Ta10IV ,Ta10V, CGHI, CGHII, CGHIII, LM18}.  

What is noteworthy here is that we will restrict ourselves to $p=2$ and study the simpler $\Z_2$ equivariant cohomology group. 
We will prove that for any knot, this $\Z_2$-equivariant Floer cohomology group is a rank 1 module over $H^*(B\Z_2; \mathbb{F}_2)=\mathbb{F}_2[Q]$.
This is shown by adapting a technique used in singular instanton theory, due to Daemi--Scaduto \cite{DS23} based on Kronheimer's argument \cite{Kr97}.

\subsection{Review of equivariant Seiberg--Witten Floer stable homotopy type}
Let $p$ be a prime number and $(\wt{Y}, \s)$ be a $\Z_p$ Spin$^c$ rational homology 3-sphere.
In \cite{BH}, we review Baraglia--Hekmati's equivariant Seiberg--Witten Floer cohomology
\[
HSW_{\Z_p}^*  (\wt{Y}, \s) = \wt{H}^*_{S^1\times \Z_p} (SWF(\wt{Y}, \s); \mathbb{F}_p), 
\]
as an equivariant version of Manolescu's construction of Seiberg--Witten Floer stable homotopy type \cite{Man03}.
They discussed more general group action and coefficients (See \cite[Section 3.1]{BH}) but we concentrate on $\Z_p$-action and $\mathbb{F}_p$ coefficient for a prime $p$, and eventually on $p=2$.
In this subsection, we briefly review the construction of this invariant following \cite[Section 3]{BH}.
We fix a $\Z_p$-invariant Riemann metric $g$. 
\par
First, choose a reference \Spinc connection $A_0$ such that the associated connection on the determinant line bundle is flat.
As shown in \cite[Section 3.2]{BH}, for $\tau \in \Z_p$, we can choose a lift of it to the spinor bundle which preserves $A_0$. 
Here we use the assumption that the $\Z_p$-action preserves the isomorphism class of $\s$ and $b_1(\wt{Y})=0$. 
Let 
$G_\s$ be the set of unitary automorphisms $u: S\to S$ on the spinor bundle $S$ preserving $A_0$ and lifting the $\Z_p$-action on $\wt{Y}$.
Then we have an extension
\[
1\to S^1\to G_\s \to \Z_p \to 1.
\]
Note that this extension is always trivial in this case as shown in \cite[Section 5]{BH}. Therefore, we take a section and identify $G_\s\cong S^1\times \Z_p$.

Now, we have an action of $G_\s$ on the global Coulomb slice
\[
V=\ker d^* \oplus \Gamma(S) \subset i\Omega^1(\wt{Y}) \oplus \Gamma(S)
\]
and a formally self-adjoint elliptic operator
\[
l: V\to V \text{ defined as }  l(a, \phi)=(* da, D_{A_0}), 
\]
where $*$ is the Hodge star operator with respect to the metric $g$ and $D_{A_0}$ is the spin$^c$ Dirac operator with respect to $A_0$.
As usual, we take a finite-dimensional approximation $V^\mu_\lambda(g)$ obtained as the direct sum of all eigenspaces of $l$ in the range $(\lambda, \mu]$, again which as acted by $G_\s$. 

By finite-dimensional approximation of the Seiberg--Witten equation, we obtain a $G_\s$-equivareiant Conley index $I^\mu_\lambda(g)$ for sufficiently large real numbers $\mu, -\lambda$.
For the details of the construction, see \cite{Man03, BH}.
Now, a metric-dependent equivariant Floer homotopy type is defined as 
\[
SWF(\wt{Y}, \s, g):=\Sigma^{-V^0_\lambda(g)}I^\mu_\lambda(g).
\]
Notice that $I^\mu_\lambda (g )$ can be taken as a finite $S^1\times \Z_p$ CW complex. Therefore, equivariant (co)homologies with respect to any subgroup   $H \subset S^1\times \Z_p$  are finitely generated over $H^*_H(pt)$ by the existence of spectral sequence.  
Now we can define a version of an equivariant Floer cohomology:  
\[
\wt{H}^*_{S^1\times \Z_p} (SWF(\wt{Y}, \s)) := \wt{H}^{*+2n(\wt{Y}, \fraks, g)}_{S^1\times \Z_p} (SWF(\wt{Y}, \s, g); \mathbb{F}_p) = \wt{H}^{*+2n(\wt{Y}, \fraks, g)}_{G_\mathfrak{s}} (SWF(\wt{Y}, \s, g); \mathbb{F}_p), 
\]
where $n(\wt{Y}, \fraks, g)$ is the correction term introduced in \cite{Man03}. 
This
 is a module over the ring 
 \[
 H^*_{S^1\times \Z_p}  = H^*_{S^1\times \Z_p}(\operatorname{pt}; \mathbb{F}_p)
 =\begin{cases}  \mathbb{F}_2[U, Q] &\text{ $p$ is $2$} \\
\mathbb{F}_p[U, R, S] / (R^2) & \text{ otherwise}
\end{cases},
\]
where $\deg(U) = 2, \deg(Q) = 1$, $\deg(R) = 1$, and $\deg(S) = 2$.

We also use 
\[
\wt{H}^{S^1\times \Z_p}_* (SWF(\wt{Y}, \s)) := \wt{H}_{*+2n(\wt{Y}, \fraks, g)}^{S^1\times \Z_p} (SWF(\wt{Y}, \s, g); \mathbb{F}_p) 
\]
This
 is a module over the ring 
 \[
 H_{S^1\times \Z_p}:= H^*(B(S^1\times \Z_p); \mathbb{F}_p)\cong
\begin{cases}  \mathbb{F}_2[U_\dagger, Q_\dagger] &\text{ $p$ is $2$} \\
\mathbb{F}_p[U_\dagger, R_\dagger, S_\dagger] / (R^2_\dagger) & \text{ otherwise}
\end{cases},
\]
where $\deg(U_\dagger) =- 2, \deg(Q_\dagger) = -1$, $\deg(R_\dagger) = -1$, and $\deg(S_\dagger) = -2$.

We also forget the $S^1$-action and consider 
\[
\wt{H}^*_{\Z_2} (SWF(\wt{Y}, \fraks); \mathbb{F}_2):=\wt{H}^{*+2n(\wt{Y}, \fraks, g)}_{\Z_2} (SWF(\wt{Y}, \s, g); \mathbb{F}_p).
\]
This is a module over
 the ring 
 \[
 H^*_{\Z_2}  := H^*(B \Z_2; \mathbb{F}_2)
 =
 \mathbb{F}_2[Q], 
\]
where $\deg(Q) = 1$.

\begin{rem}\label{dependence with splitting}
    In order to define $\Z_2$-equivariant Floer cohomology or to define our concordance invariant $q_M$, we will fix a splitting  of 
    \[
1\to S^1\to G_\s \to \Z_2 \to 1.
\]
and via the splitting $\Z_2 \to  G_\s$, we obtain a $\Z_2$-action on the Floer homotopy type. 
However, if we fix a splitting $s: \Z_2 \to G_\s$, we also have another splitting $-s : \Z_2 \to G_\s$. (Notice that the actions of $s$ and $-s$ on spin$^c$ connections are the same.) 
Note that $s $ and $-s$ are the only possible choices of splittings \cite[Remark 3.8]{BH}.  If $\fraks$ is {\it spin structure}, we claim 
\begin{align}\label{isom indep splitting}
H^*_{\Z_2, s }(SWF(\wt{Y}, \fraks); \mathbb{F}_2) \cong H^*_{\Z_2, -s }(SWF(\wt{Y}, \fraks); \mathbb{F}_2) 
\end{align}
as $\mathbb{F}_2[Q]$-modules, where $H^*_{\Z_2, s }(SWF(\wt{Y}, \fraks); \mathbb{F}_2)$ denotes the $\Z_2$-equivariant cohomology with respect to the action comes from the splitting $s$.
Now, we realize $s$ and $-s$ as follows.  Since the involution $\tau$ preserved the isomorphism class of the spin structure $\fraks$ and the fixed point set is connected and codimension $2$, there are exactly two order four lifts of $\tau$: 
\[
\pm \wt{\tau} : P_{\fraks} \to  P_{\fraks}
\]
as a morphism of principal bundles. 
See \cite{AtBot68, Ka22} for the construction of such lifts. 
 We also denote the action of $ \wt{\tau}$ on the spinor bundle $S$ of $\fraks$ as the same notion  $ \wt{\tau}$. 
 Then, we define the complex linear order two lift of $\tau$ by 
 \[
\wt{\tau}_I  :=  \wt{\tau} \cdot i 
 \]
 on $S$. We see these two actions $\pm \wt{\tau}_I$ are the same as the actions come from $s$ and $ -s$. Now, we observe 
 \[
 j \cdot \wt{\tau}_I = - \wt{\tau}_I  \cdot j 
 \]
since $\wt{\tau}$ is $\H$-linear and $i$ and $j$ anticomute.
Therefore, an isomorphism \eqref{isom indep splitting} is given by $j$-action
\[
j : SWF(\wt{Y}, \fraks) \to SWF(\wt{Y}, \fraks),
\]
which is a self-homeomorphism of the Conley index commuting with $s$ and $-s$.
\footnote{The authors thank David Baraglia for pointing out this issue on the choice of splittings.}
\end{rem}

\subsection{Equivariant relative Bauer--Furuta invariant}
In this section, we review the cobordism maps in equivariant Seiberg--Witten Floer homotopy theory.

Let $p$ be a prime number.
Let $(\wt{Y}, \s),  (\wt{Y}', \s')$ be two $\Z_p$-Spin$^c$ rational homology 3-spheres.
The notion of $\Z_p$-spin$^c$ cobordism between them is defined in a similar way to the case of 3-manifolds.
For a smooth $\Z_p$-spin$^c$ cobordism 
\[
(\wt{W}, \mathfrak{s}_{\wt{W}}): (\wt{Y}, \s)\to (\wt{Y}', \s')
\]
with $b_1(\wt{W})=0$,  Baraglia--Hekmati gave an $S^1\times \Z_p$-equivariant stable homotopy class  
\[
BF_{(\wt{W}, \mathfrak{s}_{\wt{W}})}: SWF(\wt{Y}, \s)\wedge (\C^{\frac{1}{8} (c_1(\mathfrak{s}_{\wt{W}})^2 - \sigma(\wt{W}))})^+ \to (\R^{b^+(\wt{W})})^+  \wedge  SWF(\wt{Y}', \s').
\] 
Here we have not specified the $\Z_p$-actions of the domain and the codomain of the map $BF_{(\wt{W}, \mathfrak{s}_{\wt{W}})}$ in our notation.
However, there are some representation spaces corresponding to them. 
The map $BF_{(\wt{W}, \mathfrak{s}_{\wt{W}})}$ is obtained as a finite-dimensional approximation of the Seiberg--Witten map on $W$ with some Atiyah--Patodi--Singer type spectral boundary conditions. Again for the precise construction of it, see \cite{Man03, BH}. 
    It induces a module map called the {\it cobordism map} 
    \[
   BF^*_{(\wt{W}, \mathfrak{s}_{\wt{W}})} : \wt{H}^*_{S^1\times \Z_p} (SWF(\wt{Y}', \s); \mathbb{F}_p) \to \wt{H}^*_{S^1\times \Z_p} (SWF(\wt{Y}, \s); \mathbb{F}_p). 
    \]
  This is a  $H^*_{S^1\times \Z_p}$-module map and
  its grading shift is given by 
  \[
  \frac{c_1(\fraks_{\wt{W}})^2-\sigma(\wt{W})}{4} + b^+(\wt{W}) .   
  \]
  Moreover, Baraglia--Hekmati computed the cobordism map induced on the $S^1$-fixed point part of : 
  \[
  (BF^{S^1}_{(\wt{W}, \mathfrak{s}_{\wt{W}})} )^* =  Q^{b^+(\wt{W})} U^m: \wt{H}^*_{S^1\times \Z_p} (SWF(\wt{Y}', \s)^{S^1}; \mathbb{F}_p) \to \wt{H}^*_{S^1\times \Z_p} (SWF(\wt{Y}, \s)^{S^1}; \mathbb{F}_p)
  \]
  for some $m$ when $(H^+(\wt{W}))^{\Z_p}=0$ and $p=2$.
   In our construction, we forget the $S^1$-action and just use the $\Z_p$-action. Forgetting the $S^1$-action, we still have $H^*_{\Z_p}=H^*(B\Z_p; \mathbb{F}_p)$ module map
       \[
   BF^*_{(\wt{W}, \mathfrak{s}_{\wt{W}})} : \wt{H}^*_{\Z_p} (SWF(\wt{Y}', \s); \mathbb{F}_p) \to \wt{H}^*_{ \Z_p} (SWF(\wt{Y}, \s); \mathbb{F}_p)
    \]
    with the same grading shift as the one for $S^1\times \Z_p$ equivariant cohomology.

We also use the composition law of this equivariant Bauer--Furuta invariant. 

\begin{lem}\label{BFgluing}
Let $(\wt{W}_1, \s_0, \tau_0): (\wt{Y}_0, \s_0, \tau_0)\to (\wt{Y}_1, \s_1, \tau_1)$,  $(\wt{W}_1, \s_1, \tau_1): (\wt{Y}_1, \s_1, \tau_1)\to (\wt{Y}_2, \s_2, \tau_2)$ be two $\Z_p$ equivariant cobordism between $Spin^c$ rational homology 3-spheres.
Then we have the following equality as a $\Z_p$ equivariant stable homotopy class:
\[
BF_{\wt{W}_1, \s_1, \tau_1} \circ BF_{\wt{W}_0, \s_0, \tau_0}=BF_{\wt{W}_1 \circ \wt{W}_0, \s_1 \circ \s_0, \tau_1\circ \tau_0}.
\]

\end{lem}
\begin{proof}
The proof of this composition law is parallel to that of the usual relative  Bauer-Furuta invariants \cite{Man07, KLS'18}, so we omit the proof.
\end{proof}

    We shall put $\wt{Y}$ as the $\Z_p$-covering space along a given knot $K \subset S^3$ and $\fraks$ as the (unique) $\Z_p$-invariant spin structure on $\Sigma_p(K)$.
    Also, we consider an oriented connected surface cobordism $S$ embedded into a compact 4-manifold cobordism $W$ with $b_1(W)=0 =b^+(W)$ from $S^3$ to itself. Suppose $S$ is properly and smoothly embedded and $K=S \cap -S^3 $ and $K'=S \cap S^3$ are knots in $S^3$ respectively. In order to take the unique $\Z_p$-branched covering space, we suppose the homology class of $S$ is divisible by $p$. Then, we put $\wt{W}$ as the $p$-th branched covering space  $\Sigma_p(S)$. We also take an $\Z_p$ invariant spin$^c$ structure ${\fraks}_{S}$ on $\Sigma_p(S)$ which are spin on the boundaries $\Sigma_p(K)$ and $\Sigma_p(K')$. Then, we have the induced cobordism map
    \[
  BF^*_{(S, {\fraks}_{S})} = BF^*_{(\Sigma_p(S), {\fraks}_{S})} : \wt{H}^*_{S^1\times \Z_p} (SWF(\Sigma(K'), \s'); \mathbb{F}_p) \to \wt{H}^*_{S^1\times \Z_p} (SWF(\Sigma(K), \s); \mathbb{F}_p)
    \]
    and its $\Z_p$-equivariant version 
        \[
  BF^*_{(S,{\fraks}_{S}) }=  BF^*_{(\Sigma_p(S), {\fraks}_{S})} : \wt{H}^*_{\Z_p} (SWF(\Sigma(K'), \s'); \mathbb{F}_p) \to \wt{H}^*_{ \Z_p} (SWF(\Sigma(K), \s); \mathbb{F}_p).
    \]
    We often abbreviate these map as $BF^*_S$.
    They have the same grading shifts as before
      \begin{align*}
  &\frac{c_1({\fraks}_{S})^2-\sigma(\Sigma_p(S))}{4} + b^+(\Sigma_p(S))    \\
&=  \frac{c_1({\fraks}_{S})^2}{4} + \frac{p}{4}b_2(W) + \frac{p^2-1}{12p}[S]^2 -\frac{1}{4} (\sigma^{(p)} (K' )-\sigma^{(p)}(K)) +  (p-1) g(S) - \frac{p^2-1}{6p} [S]^2 + \\
& \frac{1}{2} (\sigma^{(p)} (K') -\sigma^{(p)} (K)) \\
&= \frac{c_1({\fraks}_{S})^2}{4} +\frac{p}{4}b_2(W)
+(p-1)g(S)-\frac{p^2-1}{12}[S]^2+\frac{1}{4}(\sigma^{(p)} (K') -\sigma^{(p)} (K)). 
    \end{align*}
Here we used the formulas
\begin{itemize}
    \item $b^+(\Sigma_p(S)) =  (p-1) g(S) - \frac{p^2-1}{6p} [S]^2 + \frac{1}{2} (\sigma^{(p)} (K') -\sigma^{(p)} (K))$ and 
    
    \item $\sigma (\Sigma_p (S)) = -pb_2(W) - \frac{p^2-1}{3p}[S]^2 + \sigma^{(p)} (K')- \sigma^{(p)}(K) $
\end{itemize}
which are proven in \cite{BH}. 

Moreover, the induced map on the $S^1$-fixed point part can be computed as 
  \[
   (BF^{S^1}_{(\Sigma_p(S), {\fraks}_{S})})^* = 
    Q^{ g(S) - \frac{1}{4} [S]^2 + \frac{1}{2} (\sigma (K') -\sigma(K))} U^m 
    \]
    for some $m$ and $p=2$. 

\subsection{Bauer--Furuta invariant for product cobordisms }\label{BF=id}

We will use the following result for product cobordisms:
\begin{thm}\label{product cobordism}
Let $p$ be a prime number and let $(\wt{Y}, \s, g)$ be a $\Z_p$ spin$^c$ rational homology 3-sphere with a fixed $\Z_p$ invariant metric.
Then $\Z_p\times S^1$-equivariant relative Bauer--Furuta map for the product cobordism 
\[
BF_{([0, 1]\times 
\wt{Y}, [0, 1]\times \s, [0, 1]\times g)}: SWF(\wt{Y}, \s, g)\to SWF(\wt{Y}, \s, g).
\]
is $S^1\times \Z_p$-equivariant stably homotopic to the identity.  
\qed
\end{thm}
The authors heard that a version of this result is known to J. Lin. 
For an $S^1$-equivariant version, the corresponding result is proven in the upcoming paper \cite{SSpre}. 
We can easily see that their all homotopies can be taken as $\Z_p$-equivariant as well. 

\subsection{Pull-back of spin$^c$ structures} \label{pull back spinc}

Let $p$ be a prime number.
Let $X$ be a compact oriented 4-manifold with boundary and with $H_1(X; \Z)=0$ and $S$ be a properly and smoothly embedded oriented surface whose homology class is divisible by $p$. Suppose $\partial X=Y$ is a homology 3-sphere. 
We will see there is an injective map
\[
\pi^* : \operatorname{Spin^c} (X) \to \operatorname{Spin^c} (\Sigma_p(S))
\]
whose image is $\operatorname{Spin^c} (\Sigma_p(S))^{\Z_p}$, which is the set of isomorphism classes of $\Z_p$-invariant spin$^c$ structures. 
We first define a map 
\[
\phi_{\s, \wt{\fraks}} : \operatorname{Spin^c} (X) \to \operatorname{Spin^c} (\Sigma_p(S))
\]
by putting 
\[
\fraks \otimes L \mapsto \wt{\fraks} \otimes \pi^* L
\]
where $\wt{\fraks}$ is the $\Z_p$-invariant spin$^c$ structure on $\Sigma_p(S)$ constructed in \cite[Proposition 2.3]{Ba22}, which coincides with $\pi^* \fraks$ outside the branched locus. Also, $\pi^* L$ is a corresponding complex line bundle whose first Chern class is $\pi^* c_1(L)$. 
\begin{lem}
The map 
    $\phi_{\fraks, \wt{\fraks}}$ is injective and surjective to $\operatorname{Spin^c} (\Sigma_p(S))^{\Z_p}$. 
\end{lem}
\begin{proof}
    This follows from \cite[Lemma 2.4]{Ba22} and the 
    fact that $\operatorname{Spin^c} (X)$ is a $H^2(X; \Z)$-torsor and $\operatorname{Spin^c} (\Sigma_p(S))^{\Z_p}$ is a $H^2(\Sigma_p(S); \Z)^{\Z_p}$-torsor.
\end{proof}

The injectivity implies if we take another choice of $\wt{\fraks}'$ for $\fraks$ by \cite[Proposition 2.3]{Ba22}, we have 
\[
\wt{\fraks} \cong \wt{\fraks}'.
\]
This implies the map $\phi_{\fraks, \wt{\fraks}}$ does not depend on the choices of $\wt{\fraks}$. 
We also see $\phi_{\fraks, \wt{\fraks}}$ does not depend on $\fraks$ as follows. If we take two spin$^c$ structures $\fraks$, $\fraks'$ on $X$, there is a line bundle $L$ on $X$ such that $\fraks\cong \fraks' \otimes L  $. From the definition of $\phi_{\fraks, \wt{\fraks}}$, we have 
\[
\phi_{\fraks, \wt{\fraks}}  (\fraks) = \wt{\fraks}   \text{ and } \phi_{\fraks', \wt{\fraks}'} ( \fraks ) = \wt{\fraks}' \otimes \pi^* L .
\]
Therefore it is sufficient to take a lift of $\fraks$ by $\wt{\fraks}' \otimes \pi^*L$ and indeed it is a lift since it coincides with $\pi^*\fraks \cong \pi^* (\fraks' \otimes L) = \pi^*\fraks' \otimes \pi^* L $ outside the branched locus. 
Then we just define 
\[
\pi^*  := \phi_{\s, \wt{\fraks}} : \operatorname{Spin^c} (X) \to \operatorname{Spin^c} (\Sigma_p(S)).
\]

\subsection{On Baraglia--Hekmati's invariants for 2-knots}
We shall prove equivariant Bauer--Furuta invariants for 2-knots in $S^4$ do not depend on the choices of 2-knots. 
\begin{thm}[The second part of \cref{vanishing BF}]\label{independence of 2-knots} 
    The $S^1\times \Z_p$-equivariant stable homotopy Bauer--Furuta invariant for a $2$-knot in $S^4$
\[
BF_{(S^4, S)} : S^V \to S^{V} 
\]
is stably $S^1\times \Z_p$ homotopic to the  identity up to sign, where $V$ is a suitable $S^1\times \Z_p$-representation space.
\end{thm}

Note that for a smoothly embedded 2-knot $S \subset S^4$, the $p$-th branched cover 
$\Sigma_p ( S)$ is a rational homology $S^4$.

Here we list of real/complex irreducible $\Z_p$-reprentations: 
\begin{itemize}
    \item all of the complex irreducible representations are given by 
    the complex representation $\C_{(j)}$ defined by 
    \[
    z \mapsto e^{2\pi i j /p }z. 
    \]
    \item when $p =2$ we have the unqiue real 1-dimensional non-trivial representation $\wt{\R}$ 
    \item when $p>2$, forgetting the complex structure of $\C_{(j)}$, we have an irreducible 2-dimensional representation $\R^2_{(j)}$ which gives all of real irreducible representations. 
\end{itemize}
Note the $S^1\times\Z_p$ equivariant Bauer--Furuta invariant for a 2-knot (with pull-backed spin$^c$ structure) defines 
\[
BF_{S\subset  S^4} : (V_\R \oplus W_\C)^+ \to (V'_\R \oplus W'_\C)^+. 
\]
Here $V_\R,V'_\R$ are real representation spaces of $\Z_p$ and  $W_\C, W'_\C$ are complex representations of $\Z_p$.
In order to use equivariant homotopy theory, we need to determine the numbers of irreducible representations that appeared in $V_\R \oplus W_\C$ and $V'_\R \oplus W'_\C$.
For the real parts $V_\R$ and $V'_\R$, since $b^+(\Sigma_p (S)) =0$, one can see all equivariant indeces are zero. It implies $V_\R \cong V'_\R$.
For the Dirac operator, we use the equivariant Atiyah--Singer index theorem stated in \cite[Theorem 6.16]{BG95} which is decribed as 
\begin{align*}
\operatorname{Tr} (e^{2\pi i j/p}; \ker D_A^+) - \operatorname{Tr} (e^{2\pi i j/p}; \Coker D_A^+) = \operatorname{Lef}_{e^{2\pi i j/p}\in \Z_p } ( D^+_A )  = -\frac{1}{2\pi} \int_{\pi^{-1}(S) } e^{2\pi i j/p} \cdot  F_A  
\end{align*}
for a spin connection $A$. Then the right-hand side is zero. 
Therefore, one can see $\operatorname{Lef}_{e^{2\pi i j/p}\in \Z_p } ( D^+_A ) =0 $
 for all $j$. This again implies $W_\C \cong W_{\C'}$ as complex $\Z_p$ representations.

We use the following two results in equivariant homotopy theory: 
 
\begin{prop}[\cite{To79}]
Let 
    $f : (\R^n \oplus \C^m)^+  \to (\R^n \oplus \C^m)^+ $ be an $S^1\times \Z_p $-equivariant map. 
    The following are equivalent: 
    \begin{itemize}
        \item $f$ is a $S^1\times \Z_p $ homotopy equivalent.
        \item For any subgroup $H$ in $S^1\times \Z_p $, $f^H$ is a homotopy equivalence. 
    \end{itemize}
    
\end{prop}

\begin{lem}[\cite{To79}]\label{tomDieck}
    $f : (\R^n \oplus \C^m)^+ \to (\R^n \oplus \C^m)^+  $ be an $S^1 $-equivariant map.
    Then we have 
    \[
    \deg (f ) = \deg (f^{S^1}). 
    \]
\end{lem}

\begin{proof}[Proof of \cref{independence of 2-knots}]
   All subgroups of $S^1\times \Z_p $: 
\begin{enumerate}
    \item  a subgroup $H\subset S^1\times \Z_p$ such that $H \cap (S^1\times \{e\}) \neq \{e\}$, 
    \item  $\{ e\} \times \Z_p$. 
    \item $\{e\} \subset S^1\times \Z_p$
\end{enumerate}
Let $BF_{(S^4, S)} : S^V \to S^{V} $ be the $S^1\times \Z_p$ equivariant Bauer--Furuta invariant for a given $2$-knot. 
Since $BF^{H}_{(S^4, S)}$ comes from the linearized Seiberg--Witten equation, the mapping degree is $\pm 1$.
For $BF^{\{ e\} \times \Z_p}_{(S^4, S)}$ and $BF^{\{ e\}}_{(S^4, S)}$, one can apply \cref{tomDieck}, so it is enough to show that $BF^{S^1 \times \Z_p}_{(S^4, S)}$ and $BF^{S^1}_{(S^4, S)}$ has mapping degree  $\pm 1$, but it is true because it comes from the linear isomorphism.
Therefore, it is a homotopy equivalence. 
\end{proof}

\subsection{Statements and proofs of rank theorems}
In this subsection, we prove that for any knot $K$ in $S^3$, the equivariant Floer homology of the double branched cover  $\wt{H}^*_{\Z_2} (SWF(\Sigma_2(K), \fraks_0))$ is a rank 1 module over $\mathbb{F}_2 [Q]$,  where  $\fraks_0$ is the unique spin structure on the branched double cover $\Sigma_2(K)$. Note that from \cref{dependence with splitting}, we see the $\Z_2$-equivariant Fleor cohomology does not depend on the choices of splittings. 
We also give a remark to extend this argument to knots in general homology 3-spheres.

\begin{thm}[\cref{rank-1}] \label{rank1 theorem}
    For any knot $K$ in $S^3$, we have 
    \[
\operatorname{rank}_{\mathbb{F}_2[Q]} \wt{H}^*_{\Z_2} (SWF(\Sigma_2(K)); \mathbb{F}_2) =1
    \]
    Moreover,
    \[
\operatorname{rank}_{\mathbb{F}_2[Q]} \operatorname{Ker} (U^i : \wt{H}^*_{S^1\times \Z_2} (SWF(\Sigma_2(K)); \mathbb{F}_2) \to \wt{H}^*_{S^1\times \Z_2} (SWF(\Sigma_2(K))); \mathbb{F}_2) =i. 
    \]
\end{thm}

We first take a sequence of crossing changes from $K$ to the unknot $U_0$ and regard it as a normally and smoothly immersed annulus $S$ from $K$ to $U_0$ in $[0,1]\times S^3$. 
 
By the use of $S$, we will define a cobordism map 
\[
BF^*_{S} : \wt{H}^*_{\Z_2} (SWF(\Sigma_2(U_0))) \to \wt{H}^*_{\Z_2} (SWF(\Sigma_2(K))) 
\]
with a certain grading shift. 
\begin{defn}
The cobordism maps for a normally immersed cobordism $S$ from $K$ to $K'$ in $[0, 1] \times S^3$ are defined to be equivariant cobordism maps for the embedded surface $S_b$ obtained as the proper transform of $S$ in the blow-up with respect to all immersed points of $S$;
\[
S_b \subset ([0, 1]\times S^3) \displaystyle \#_{s_++ s_- } \overline{\mathbb{C}P} ^2 , 
\]
where $s_+$ and $s_-$ are numbers of positive and negative immersed points. \footnote{We only use the existence of certain cobordism maps. Therefore, we do not need to care about the dependence of $S_b$ with respect to additional data for the proper transform.} 
In this case, the homology class of $S_b$ is represented as 
\[
[S_b]= (-2, \cdots, -2, 0, \cdots, 0) \in H_2 (\#_{s_++ s_- }\overline{\mathbb{C}P}^2; \Z) . 
\]
Since $[S_b]$ is divisible by $2$, we have the double branched covering space along $S_b$. We take a spin$^c$ structure $\s_b$ corresponding to the sum of the generators of $H_2 (\#_{s_++ s_- }\overline{\mathbb{C}P}^2)$ and "pullback" it to the double-branched covering space:
First, we consider the spin$^c$ structure $\s_b$ on $([0, 1] \times S^3) \#_{s_++ s_- }\overline{\mathbb{C}P}^2$is given by
\[
c_1(\s_b)=(1, \dots, 1) \in H^2(I \times S^3 \#_{s_++ s_- }\overline{\mathbb{C}P}^2).
\]
The $\Z_2$-invariant spin$^c$ structure on the branched cover is  chosen such that the condition 
\[
c_1(\wt{\fraks}) = \pi^* (c_1(\s_b) + \frac{1}{2} [S_b] ) 
\]
is satisfied.
See \cite[Proposition 2.5]{Ba22} for the existence of such spin$^c$ structures. 

We denote by 
\[
BF^*_{S_b} : \wt{H}^*_{\Z_2} (SWF(\Sigma_2(K')) \to  \wt{H}^*_{\Z_2} (SWF(\Sigma_2(K))
\]
the cobordism map with respect to the above spin$^c$ structures.

\end{defn}

\begin{rem}
    This is an analogous construction of invariants for immersed surfaces given in \cite{Kr97}. 
\end{rem}

Also, the normally immersed cobordism $-S: U_0 \to K$ obtained by reversing the orientation of $S$. Thus, we also have 
\[
BF^*_{-S} : \wt{H}^*_{\Z_2} (SWF(\Sigma_2(K))) \to \wt{H}^*_{\Z_2} (SWF(\Sigma_2(U_0))) 
\]
again with some grading shifts.

\begin{prop}\label{key rank one}
With respect to the above cobordism maps, if we consider the localization with $Q$, we have 
\begin{align*}
    ( BF^*_{S} \circ BF^*_{-S} )^{\operatorname{loc}}  = Q^i \\ 
 (BF^*_{-S} \circ BF^*_{S})^{\operatorname{loc}} = Q^i 
\end{align*}
for some $i$, where, for a $\mathbb{F}_2[U]$-module map $f : A \to B, f^{\operatorname{loc}}$ denotes the induced map on the localization 
\[
f^{\operatorname{loc}} : Q^{-1}A \to Q^{-1}B.
\]
\end{prop}

\begin{rem}
We can also consider variants of the cobordism maps by just replacing the unknot $U_0$ with $\#_{-\frac{1}{2}\sigma (K)} T(2,3)$. We call such a cobordism $S_{gr}$. 
In this case, the cobordism maps $BF^*_{S_{gr}}$ and $BF^*_{-S_{gr}}$ are grading preserving maps. In this case, we have
\begin{align*}
    (BF^*_{S_{gr}} \circ BF^*_{-S_{gr}})^{\operatorname{loc}}  =1 \\ 
 ( BF^*_{-S_{gr}}  \circ BF^*_{S_{gr}})^{\operatorname{loc}} = 1 . 
\end{align*}
\end{rem}

\begin{proof}[Proof of \cref{rank1 theorem}]
    \cref{rank1 theorem} just follows from the existence of module maps $BF^*_{S}$ and $BF^*_{-S}$ satisfying the conditions written in \cref{key rank one}. 
\end{proof}

For the proof of \cref{key rank one}, we prove the following general result: 
\begin{thm}\label{general indep}
 Let $S$ be a properly connected embedded smooth surface cobordism in a cobordism $W$ from a pair of a homology $3$-sphere $Y$ and a knot $K$ in Y to another pair $(Y', K')$ whose homology class is divisible by $2$.
Then, the $Q$-localized cohomological Bauer--Furuta invariant
\[
 BF^{* \operatorname{loc}}_{(W, S)} :Q^{-1}\wt{H}^*_{\Z_2} (SWF(\Sigma_2(K'))) \to Q^{-1} \wt{H}^*_{\Z_2} (SWF(\Sigma_2(K)))  
\]
depends only on the homotopy class of $S$, precisely, for two choices homotopic $S$ and $S'$ rel boundary, we have 
\[
BF^{*\operatorname{loc}}_{(W, S)}  = BF^{*\operatorname{loc}}_{(W, S')} .
\]
\end{thm}

\begin{proof}
The argument is similar to that of  Daemi--Scaduto \cite{DS23} for singular instantons, which is based on arguments by Kronheimer \cite{Kr97}.
Let 
\[
S, S' \subset I \times S^3
\]
be two homotopic normally immersed surface cobordims $ (S^3, K)\to (S^3, K')$ with the same genus. 
 It is known that there is a sequence of  moves from $S$ to $S' $ listed below: 
\begin{itemize}
   \item[(0)]  an ambient isotopy $T \to T'$ of the image in $I \times S^3$ rel $\partial I \times S^3$.
    \item[(i)] positive twist move $T \to T'$ ($T'$ has one more positive immersed point than $T$)
        \item[(ii)] negative twist move $T \to T'$, ($T'$ has one more negative immersed point than $T$) and 
    \item[(iii)] finger move $T \to T'$ ($T'$ has one more positive double point and one more negative double point than $T$)
\end{itemize}
or inverse of one of these.
The relation of the $\Z_2$-equivariant cohomological cobordism maps
under these moves can be written as 
\begin{itemize}
\item [(0)] $BF^*_T = BF^*_{T'} $
    \item[(i)] $Q BF^*_T = BF^*_{T'} $,
    \item[(ii)] $BF^*_T= BF^*_{T'}$ and 
    \item[(iii)] $Q BF^*_T = BF^*_{T'}$. 
\end{itemize}

Similar arguments are used by Kronheimer \cite{Kr97} for instantons. 
\par
For (0), $BF^*_T$ is a diffeomorphism invariant, so this is invariant under ambient isotopy. \par
For (i), 
from the construction of the blow-up, one can see 
\[
(\Sigma_2(T'_b), \iota_{T_b'}) \cong (\Sigma_2(T_b) \# S^2 \times S^2, \iota_T \# \iota ) 
\]
as $\Z_2$-equivariant manifolds. Note that the induced spin$^c$ structure on the $S^2\times S^2$-component is spin.

Therefore, the connected sum formula of equivariant Bauer--Furuta invariants, we see 
\[
BF_{(\Sigma_2(T'_b), \iota_{T_b'})}  = BF_{(\Sigma(T), \iota_T)}  \wedge BF_{( S^2 \times S^2, \iota)} . 
\]
From the equivariant index theorem and equivariant Hopf theorem written in [Tom--Dieck Trans Group Chapter II 4, Theorem 4.11(iv)], one can see 
\[
BF_{( S^2 \times S^2, \iota)} ^* = Q
\]
with respect to the spin structure on $S^2 \times S^2$.
This computation is still true for $S^1\times \Z_2$ case. 
This completes the proof of (i). 

For (ii), 
one can see 
\[
(\Sigma_2(T'_b), \iota_{T_b'}) \cong (\Sigma(T_b) \# \overline{\C P}^2 \# \overline{\C P}^2 , \iota_T \# \iota ) 
\]
as $\Z_2$-equivariant manifolds. Here, $\iota$ is given by the equivariant connected sum along two points of the flipping involution on $\overline{\C P}^2 \cup \overline{\C P}^2$ and the involution on $S^4$ arises as the double branched cover long trivial $2$-knot in $S^4$,  
 From the connected sum formula of equivariant Bauer--Furuta invariants, we see 
\[
BF_{(\Sigma_2(T'_b), \iota_{T_b'})}  = BF_{(\Sigma(T), \iota_T)}  \wedge BF_{( \overline{\C P}^2 \# \overline{\C P}^2 , \iota)} . 
\]
Note that the induced spin$^c$
structure on $\overline{\C P}^2 \# \overline{\C P}^2$ is corresponding to the sum of generators of the Poincar\'e duals of exceptional curves.  Again from the equivariant Hopf theorem written in [Tom--Dieck Trans Group Chapter II 4, Theorem 4.11(iv)], one can easily see 
\[
BF_{( \overline{\C P}^2 \# \overline{\C P}^2, \iota)} ^* =\pm \id 
\]
up to $S^1\times \Z_2$-stable homotopy.
This completes the proof of (ii).

Since (iii) is a combination of (i) and (ii), this completes the proof of (0)-(iii). 

So, using it and since there is a sequence of moves from $S$ to $S'$, we can write down 
\[
BF_{S} =  Q^j BF_{S' }\text{ in localization by $Q$}
\]
for some integer $j$.  
This completes the proof. 
\end{proof}

Now we give a proof of \cref{key rank one}: 
\begin{proof}[Proof of \cref{key rank one}]
We first consider $\Z_2$-equivariant cohomologies. 
    The composition $S \circ (-S)$ is an immersed genus $0$ cobordism from $U_0$ to itself. Note that $U_0 \times [0,1] $ is also a cobordism from $U_0$ to itself with genus $0$. 
So, using it and since there is a sequence of three kinds of moves from $S\circ -S$ to $U_0 \times [0,1] $, we can write down 
\[
BF^*_{S \circ -S} =  Q^j BF^*_{U_0 \times [0,1] } = Q^j \text{ in localization by $Q$}
\]
for some integer $j$. 
Here we used $BF^*_{U_0 \times [0,1] } = \id $ which will be proven in \cref{BF=id}. 
We can do completely the same discussion for $-S \circ S$ to see
\[
BF^*_{-S \circ S} =  Q^{i} BF^*_{K\times [0,1] } = Q^{j}  \text{ in localization by $Q$} . 
\]
 Here again we used $BF^*_{K \times [0,1] } = \id $.
From the grading reason, we see $i=j$. Moreover if we take $S$ as $S_{gr}$, then $i = j=0$ holds. 

Next, we consider $S^1\times \Z_2$ case. Note that $BF_{S}$ and $BF_{-S}$ are $U$-module maps. Therefore, they induce maps on $\ker U^i$ and the relations 
\[
BF^*_{-S \circ S} =  Q^{i} BF^*_{K\times [0,1] } = Q^{i}  \text{ in localization by $Q$} 
\]
are still true as maps on $\ker U^i$. 
This completes the proof. 
\end{proof}

We show some finiteness properties of Baraglia--Hekmati's $S^1\times \Z_2$ equivariant Floer homology and cohomology of double branched covers using \cref{rank1 theorem}.

\cref{towers} is a special case of the following general theorem: 
\begin{thm}
Let $K \subset S^3$ be a knot.
\begin{enumerate}
\item 
\begin{enumerate}
\item 
For $n \geq 1$, we have
\[
\rank_{\mathbb{F}[Q_\dagger ]}\ker(U^n_{\dagger}: \wt{H}^{S^1\times \Z_2}_*(SWF(\Sigma_2(K)))\to \wt{H}^{S^1\times \Z_2}_*(SWF(\Sigma_2(K))))=n
\]
for the Floer homology group.

\item 
For $n \geq 1$, we have
\[
\rank_{\mathbb{F}[Q]}\cok(U^n: \wt{H}_{S^1\times \Z_2}^*(SWF(\Sigma_2(K)))\to \wt{H}_{S^1\times \Z_2}^*(SWF(\Sigma_2(K))))=n.
\]
for the Floer cohomology group.
\end{enumerate}
\item
\begin{enumerate}
\item 
For $n \geq 1$, 
\[
\cok(U^n_{\dagger}: \wt{H}^{S^1\times \Z_2}_*(SWF(\Sigma_2(K)))\to \wt{H}^{S^1\times \Z_2}_*(SWF(\Sigma_2(K))))
\]
is a torsion $\mathbb{F}[Q_\dagger]$ module.
\item
 For $n \geq 1$, 
\[
\ker(U^n: \wt{H}_{S^1\times \Z_2}^*(SWF(\Sigma_2(K)))\to \wt{H}_{S^1\times \Z_2}^*(SWF(\Sigma_2(K))))
\]
is a torsion $\mathbb{F}[Q]$ module.
\end{enumerate}
\item 
\begin{enumerate}
\item 
The quotient
\[
\wt{H}^{S^1\times \Z_2}_*(SWF(\Sigma_2(K)))/\bigcap_{n=1}^\infty \im U^n_{\dagger}
\]
is finite. Thus, for all but finite elements, their $U_\dagger$-divitibility is infinite.
\item 
The number of $U$-torsion elements in
the Floer cohomology group $
\wt{H}_{S^1\times \Z_2}^*(SWF(\Sigma_2(K)))$
 is finite. 
 In particular, for sufficiently large $j \in \Z^{\geq 0}$, $\im Q^j \subset \wt{H}_{S^1\times \Z_2}^*(SWF(\Sigma_2(K)))$ does not contain any $U$-torsions and thus the quotient $\im Q^j/\im Q^{j+1}$ is a free $\mathbb{F}[U]$ module. 
\end{enumerate}
\end{enumerate}
\end{thm}
\begin{proof}
We only prove the claim for the cohomology group. The claim for the homology group can be straightforwardly obtained by considering the dual.
In the proof of this theorem, we abbreviate
\[
\wt{H}^*_{\Z_2}=\wt{H}^*_{\Z_2}((SWF(\Sigma_2(K)))).
\]
\[
\wt{H}^*_{\Z_2 \times S^1}=\wt{H}^*_{\Z_2 \times S^1}((SWF(\Sigma_2(K)))).
\]
to simplify the notation.
\par
(1) This is proved in \cref{rank1 theorem}.
\par

(2) We prove this by induction on $n$.\par
First let us consider $n=1$.
By the Thom-Gysin exact sequence
\[
\cdots \to \wt{H}^*_{\Z_2}\xrightarrow{\delta} \wt{H}^{*-1}_{\Z_2\times S^1}\xrightarrow{U} \wt{H}^{*+1}_{\Z_2\times S^1} \xrightarrow{\pi^*} \wt{H}^{*+1}_{\Z_2}\to \cdots 
\]
we have
\[
\ker U=\im \delta  \cong \wt{H}^*_{\Z_2}/\im \pi^*.
\]
Since $\rank_{\mathbb{F}[Q]}\wt{H}^*_{\Z_2}=1$ by \cref{rank1 theorem},  it is enough to show
$\rank_{\mathbb{F}[Q]}\im \pi^*=1$.
By using the Thom-Gysin exact sequence again, we have
\[
\im \pi^*\cong \wt{H}^*_{\Z_2\times S^1}/\ker \pi^*=\cok U.
\]
This has rank 1 over $\mathbb{F}[Q]$ by (1).
This proves the claim for $n=1$.
\par
Now assume $\ker U^n$ is finite and we will show $\ker U^{n+1}$ is also finite.
Suppose on the contrary that $\ker U^{n+1}$ is infinite.
Then there exists at least one $x \in \ker U^n$ such that there exist infinitely many $y_1, y_2, \dots \in \ker U^{n+1}$ such that $U y_i=x$ for all $i$. Then we have infinitely many elements $y_i-y_1 \in \ker U^n$ so this contradicts the assumption that $\ker U^n$ is finite. This proves (2).
\par
(3) Suppose on the contrary that we have infinitely many $U$-torsions
\[
x_1, x_2, \dots\in \wt{H}^*_{\Z_2 \times S^1}.
\]
For an element $x \in \wt{H}^*_{\Z_2 \times S^1}$
define its $U$-torsion order by 
\[
\operatorname{ord}_U (x)=\inf \{n \in \Z^{\geq 0}| U^n x=0\}.
\]
Now we can see that $\{\operatorname{ord}_U(x_i)\}$ is unbounded. This is because if there exists some $n_0$ with 
\[
\operatorname{ord}_U(x_i) \leq n_0
\]
for all $i$, then $\{x_i\}$ are infinitely many elements of $\ker(U^{n_0})$ and this contradicts (2).
Thus, we have $x_i$ with arbitrary large $\operatorname{ord}_U(x_i)$.
Now, since $\ker U$ is finite by (2), 
$\gr^\Q(\ker U)$ is finite. 
We always have $U^{\operatorname{ord}_U(x)-1} x \in \ker U$, and $U$ increases the grading by $2$, so we have $x_i$ with arbitrary small rational grading.  This contradicts the fact that $gr^\Q$ is bounded below on the whole cohomology group $\wt{H}^*_{\Z_2 \times S^1}$.
This completes the proof.
\end{proof}

\section{New concordance invariant $q_M(K)$ }

\subsection{Construction of $q_M(K)$}
In this subsection, we will introduce a new concrdance invariant $q_M(K)$ and several fundamental properties written in \cref{main slice-torus}. 

Let $K \subset S^3$ be a knot.
In this section, we define an integer-valued concordance invariant $q_M(K) \in \Z$.
We also prove that it is a slice-torus invariant.
This invariant is defined as follows.
Set
\[
q^\dagger_M(K) := \min \{ \gr^\Q(x) | Q^n x \neq 0,  \text{ for all }n\geq 0,   x \in H^*_{\Z_2} (SWF(\Sigma_2(K); \fraks_0))\}  
\]
where we assume $x$ is homogeneous and set
\[
q_M(K)=q^\dagger_M(-K)-\frac{3}{4}\sigma(K)
\]
and $\fraks_0$ is the unique spin structure, where $-K$ means the concordance inverse of $K$.
Note that we see the invariant $q_M$ is independent of the choices of splittings given in \cref{dependence with splitting}. 

\begin{lem}
For any knot $K$ in $S^3$, 
\[
\gr^\Q+\frac{3}{4}\sigma(K).
\]
is $\Z$-valued on $ \wt{H}^*_{\Z_2} (SWF(\Sigma_2(K)))$.
Thus  $q_M(K)$ is an integer-valued invariant.
\end{lem}
\begin{proof}
Let us recall that $ \wt{H}^*_{\Z_2} (SWF(\Sigma_2(K)))$ is defined as
\[
\wt{H}^*_{\Z_2} (SWF(\Sigma_2(K)))=\wt{H}^{*+\dim_\R V^0_{\lambda} +2n(\Sigma_2(K), \s_0, g)}(I^\mu_\lambda(g) ).
\]
Since $\dim_\R V^0_{\lambda}$ is an integer, this can be notrivial only when
\[
*+2n(\Sigma_2(K), \s_0, g) \in \Z, 
\]
where $n(\Sigma_2(K), \s_0, g)$ is the correction term introduced in \cite{Man03}. 
Therefore 
$\gr^{\Q}+2n(\Sigma_2(K), \s_0, g)$ is $\Z$-valued.
Now $2n(\Sigma_2(K), \s_0, g) $ can be computed as
\[
2n(\Sigma_2(K), \s_0, g)=2\ind^{APS}_\C D^+_{\Sigma_2(S)} +\frac{\sigma(\Sigma_2(S))}{4} 
\]
\[
=2\ind^{APS}_\C D^+_{\Sigma_2(S)} +\frac{\sigma(K)}{4}
\]
where $S \subset D^4$ is a properly embedded surface obtained by pushing off a Seifert surface of $K$.
The Atiyah--Patodi--Singer index term is an integer, so
$\gr^{\Q}+\frac{\sigma(K)}{4} $ is $\Z$-valued.
For any knot $K$ in $S^3$, the knot signature $\sigma(K)$ is an even number.  
Thus $\gr^{\Q}+\frac{3}{4}\sigma(K) $ is $\Z$-valued.
This completes the proof
\end{proof}
\par
First, we will prove that, for any knot $K$ in $S^3$ and its inverse $-K$ in the concordance group, we have  $-q_M(K)=q_M(-K)$. This follows from the corresponding formula for  $q^\dagger_M$, which is a consequence of a duality of the Floer cohomology groups.
\begin{prop}\label{duality} For a knot $K$ in $S^3$, we have 
    \[
    q_M(-K) = - q_M (K), 
    \]
    where $-K$ denotes the concordance inverse, i.e. its mirror with the opposite orientation. 
\end{prop}

\begin{proof}[Proof of \cref{duality}]
By \cref{rank1 theorem} and the definition of $q^\dagger_M$, we have 
  \[
\wt{H}^*_{\Z_2}  ( SWF(\Sigma_2(K))/\text{($Q$-torsions)} \cong \mathbb{F}_2[Q]_{(q^\dagger_M(K))}. 
\]  
and 
  \[
\wt{H}^*_{\Z_2}  ( SWF(\Sigma_2(-K))/\text{($Q$-torsions)} \cong \mathbb{F}_2[Q]_{(q^\dagger_M(-K))}. 
\]  
Thus, by the duality for orientations of $\Sigma_2(K)$ and the universal coefficient theorem, we have $q^\dagger_M(-K)= -q^\dagger_M(K)$.
By combining this with $\sigma(-K)=-\sigma(K)$, we have the desired relation $q_M(-K) = - q_M (K)$.
\end{proof}

The following gives the general genus bound: 
\begin{thm}\label{main slice-torus gen}

Let $S$ be a smoothly, properly and normally immersed connected oriented embedded surface cobordism in  $[0,1]\times S^3$ from a knot $K$ in $S^3$ to a knot $K'$ in $S^3$.  
   Then, we have the following genus bound: 
   \[
   q_M(K') \leq g(S)  +s_+(S) + q_M(K) .
   \]
\end{thm}
\begin{proof}
We consider the cobordism map 
\[
BF^{*}_{S}: \wt{H}^*_{\Z_2} (SWF(\Sigma_2(K'))) \to \wt{H}^*_{\Z_2} (SWF(\Sigma_2(K))) 
\]
defined by using proper transformations along immersed points and taking a certain $\Z_2$-invariant spin$^c$ structure on it. 
We first prove that $BF^{*, \text{loc}}_{S}$ is isomorphism as a homomorphism
\[
Q^{-1}\wt{H}^*_{\Z_2} (SWF(\Sigma_2(K'))) \to Q^{-1} \wt{H}^*_{\Z_2} (SWF(\Sigma_2(K))). 
\]
   The composition $S \circ (-S)$ is an immersed genus $2g(S)$ cobordism from $K$ to itself. Note that $K \times [0,1] \#_{2g(S)} T^2 $ is also a cobordism from $K$ to itself with genus $2g(S)$.
   These two cobordisms are homotopic rel boundaries.
   This can be checked as follows:
   Regard $S^3$ as $\R^3 \cup \{\infty\}$.
   Then $I \times \{\infty\}$ does not intersect with the surface cobordisms in general position in $ [0, 1]\times S^3$, so one can assume the surface cobordism is in $[0, 1] \times \R^3$. Now we have the linear homotopy connecting two surface cobordisms. 
   \par
From \cref{general indep}, we can write down 
\[
BF^*_{S \circ -S} =  Q^j BF^*_{U_0 \times [0,1] \#_{2g(S)} T^2  } = Q^j BF^*_{U_0 \times [0,1]} \cdot BF^*_{ \#_{2g(S)} T^2  }=  Q^{j+ 2g(S)} 
\]
for some integer $j$. Here we used the following: 
\begin{itemize}
\item equivariant connected sum formula of equivariant Bauer--Furuta invariants, 
    \item $BF^*_{T^2 \subset S^4} = Q$ which can be proven by combining the equivariant index theorem and the equivariant Hopf theorem, and 
    \item $BF^*_{U_0 \times [0,1] } = \id $ which will be proven in \cref{BF=id}.
\end{itemize}
This implies $BF^{*, \text{loc}}_{S}$ is non-zero. 
Thus the free part of the cobordism map
\[
BF^{*, \text{free}}_{S}: \wt{H}^*_{\Z_2} (SWF(\Sigma_2(K')))/Q\text{-Torsion} \to  \wt{H}^*_{\Z_2} (SWF(\Sigma_2(K)))/Q\text{-Torsion}
\]
is also non-trivial.
When we regard this as a map
\[
BF^{*, \text{free}}_{S}: \mathbb{F}_2[Q]_{q^\dagger(K')}\to \mathbb{F}_2[Q]_{q^\dagger(K)}, 
\]
it is given as
\[
BF^{*, \text{free}}_{S}=Q^{q^\dagger(K')-q^\dagger(K)+b^+(\Sigma_2(S_b))- \frac{1}{4 } (c_1(\wt{\fraks})^2 - \sigma(\Sigma_2(S_b)))}
\]
by the computation of grading.
Since $BF^{*, \text{free}}_{S} \neq 0$, the power of $Q$ is non-negative, so we have 
\[
-q^\dagger_M ( K') \leq  b^+(\Sigma(S))- \frac{1}{4 } (c_1(\wt{\fraks})^2 - \sigma(\Sigma(S))) - q^\dagger_M(K).
\]

Then, we do the following computations:  
\begin{align*}
&-q_M^\dagger ( K')\\ 
 &  \leq  b^+(\Sigma(S_b))  -  \frac{1}{4 } (c_1(\wt{\fraks})^2 - \sigma(\Sigma(S_b)))  -q_M^\dagger ( K) \\ 
   & \leq  g(S_b) -\frac{1}{4}[S_b]^2 + \frac{1}{2}\sigma(K' -K) -    \frac{1}{4 } ( \pi^* (c_1(\fraks)  -  \frac{1}{2}[S_b])^2 - 2 \sigma(X) + \frac{1}{2}[S_b]^2  - \sigma(K' -K)) - q^\dagger_M(K) \\ 
    & = g(S)  + \frac{1}{2}  (\langle c_1(\fraks), [S_b] \rangle - [S_b]^2   )   + \frac{3}{4}\sigma(K' -K)- q^\dagger_M(K)  .
    \end{align*}
    Note that for the blow-up of $S$ in $[0,1]\times S^3 \# \#_m -\C P^2$, one can see 
    \[
    \frac{1}{2}  (\langle c_1(\fraks), [S_b] \rangle - [S_b]^2   ) = s^+(S). 
    \]

    This completes the proof. 
\end{proof}

\begin{rem}
Even for an odd prime $p$, we can still define a similar invariant as 
 \begin{align*}
 q_M^{(p)} (K) :=\frac{1}{p-1} \min \{ \gr^\Q ( x) |    x \in H^*_{\Z_p} (SWF(-\Sigma_p(K)); \mathbb{F}_p):  \\ 
 \text{homogeneous  and } S^n x \neq 0  \text{ for all }n\geq 0\} - \frac{3}{4(p-1) } \sigma^{(p)}(K), 
  \end{align*}
  where 
  \[
  \sigma^{(p)}(K) := \sum_{1\neq\om \in U(1), \om^p=1} \sigma_\om (K)
  \]
  and $\sigma_\om (K)$ denotes the Tristram--Levine signature with respect to $\om\in U(1)$.
\end{rem}

\subsection{Basic properties of $q_M(K)$}

In this section, we prove the fundamental properties of the invariant $q_M(K)$ including relations with L-spaces, the duality, and the connected sum formula.

\begin{thm}
      If the double-branched cover of $K$ is L-space, then we have 
    \[
    \wt{H}^*_{\Z_2}(SWF(\Sigma_2(K); \mathfrak{s}_0); \mathbb{F}_2)\cong \mathbb{F}_2[Q]_{d=-\frac{\sigma(K)}{4}},
    \]
    where the subscript is the minimal grading and $\fraks_0$ is the unique spin structure. 
   Thus, 
   \[
   q_M(K) = -\frac{1}{2}\sigma(K).
   \]
   In particular, this holds for quasi-alternating knots. 
\end{thm}
\begin{proof}
Let us abbriviate $\mathbb{F}=\mathbb{F}_2$ in this proof.
The assumption that  $\Sigma_2(K)$ is an L-space
implies
\[
\wt{H}^*  ( SWF(\Sigma_2(K), \s_0;  \mathbb{F}) \cong \mathbb{F}_{(d)}, 
\]
where the grading
\[
d=d(\Sigma_2(K);\s_0)=2\delta(\Sigma_2(K); \s_0)=-\frac{\sigma(K)}{2}
\]
is the monopole Fr\o yshov invariant \cite{Fr10} of the double-branched covering space along $K$.
For the last equality, see  \cite[Corollary 6.3]{BH} (we follow their convention). 
We omit $\s_0$ from now on.
By the mod 2  Thom-Gysin exact sequence for the following $S^0=\Z_2$-bundle
\[
\Z_2 \to SWF(\Sigma_2(K)) \wedge (E\Z_2)_+ \xrightarrow{\pi} SWF(\Sigma_2(K)) \wedge_{\Z_2} (E\Z_2)_+, 
\]
we have 
\[
\wt{H}^{i}_{\Z_2}(SWF(\Sigma_2(K)))\xrightarrow[\cong]{Q}\wt{H}^{i+1}_{\Z_2}(SWF(\Sigma_2(K))), \quad i \neq d-1, d
\]
and an exact sequence
\[
0 \to \wt{H}^{d-1}_{\Z_2}(SWF(\Sigma_2(K)))\xrightarrow{Q} \wt{H}^{d}_{\Z_2}(SWF(\Sigma_2(K)))\xrightarrow{\pi^*_d}
 \mathbb{F}\to \wt{H}^{d}_{\Z_2}(SWF(\Sigma_2(K)))\xrightarrow{Q} \wt{H}^{d+1}_{\Z_2}(SWF(\Sigma_2(K))) \to 0.
\]
We can see $\pi^*_d\neq 0$.
Suppose on the contrary $\pi^*_d=0$.
Then we have $\wt{H}^{\leq d}_{\Z_2}(SWF(\Sigma_2(K)))=0$ but this contradicts the resulting exact sequence 
\[
0 \to \mathbb{F}\to \wt{H}^{d}_{\Z_2}(SWF(\Sigma_2(K)))\xrightarrow{Q} \wt{H}^{d+1}_{\Z_2}(SWF(\Sigma_2(K))) \to 0.
\]
Thus we have $\pi^*_d\neq 0$ and thus we obtain 
\[
\wt{H}^{i}_{\Z_2}(SWF(\Sigma_2(K)))\xrightarrow[\cong]{Q}\wt{H}^{i+1}_{\Z_2}(SWF(\Sigma_2(K))), \quad i \neq  d.
\]
and 
\[
0 \to \wt{H}^{d-1}_{\Z_2}(SWF(\Sigma_2(K)))\xrightarrow{Q} \wt{H}^{d}_{\Z_2}(SWF(\Sigma_2(K)))\xrightarrow{\pi^*_d}
 \mathbb{F}\to 0
 \]
 Therefore we have
 \[
 \wt{H}^*_{\Z_2}(SWF(\Sigma_2(K)))\cong \mathbb{F}[Q]_d
 \]
 and in particular $q^\dagger_M(K)=d=-\frac{\sigma(K)}{4}$.
 This implies
 \[
 q_M(K)=-q^\dagger_M(K)-\frac{3}{4}\sigma(K)=-\frac{\sigma(K)}{2}.
 \]
 The fact that the branched double cover of the quasi-alternating knots are L-spaces is due to Ozsv\'ath--Szab\'o \cite[Proposition 3.3]{OS05}. This completes the proof. 
\end{proof}

\begin{thm}\label{conn sum of qm}For two oriented knots $K$ and $K'$ in $S^3$, 
\[
    q_M(K \# K') = q_M (K) + q_M(K')
    \]
\end{thm}

We use the following two lemmas in the proof of this theorem.

The first lemma claims that the cohomological relative Bauer-Furuta invariant for $(S^0 \times D^4, S^0 \times D^2)$ and that of $(D^1 \times S^3, D^1 \times S^1)$
are the same Floer cohomology classes of $(S^0 \times S^3,  S^0\times S^1)$.

\begin{lem}\label{BFsurgery}
We have
\[
BF^*_{(S^0 \times D^4, S^0 \times D^2)}(1)=BF^*_{(D^1 \times S^3, D^1 \times S^1)}(1)
\in \wt{H}^*_{\Z_2}(SWF((S^0 \times S^3,  S^0\times S^1))). 
\]
\end{lem}
\begin{proof}
A similar gluing argument has been used in \cite[Subsection 3.7]{DSS23}.  Computation of $BF^*_{S^0 \times D^4, S^0, D^2}$ can be done by the proof of
\cref{independence of 2-knots}. 
\end{proof}
The second lemma is a localization theorem for $\Z_2$-action. 
\begin{lem}\label{localization}
The inclusion  
\[
SWF(\Sigma_2(K))^{\Z_2}\to SWF(\Sigma_2(K))
\]
induces an isomorphism
\[
Q^{-1}\wt{H}^*_{\Z_2}(SWF(\Sigma_2(K)))\to \wt{H}^*(SWF(\Sigma_2(K))^{\Z_2})\otimes \mathbb{F}[Q, Q^{-1}].
\]
\end{lem}
\begin{proof}
Note that 
$SWF(\Sigma_2(K))$ is a formal desuspension of a finite $\Z_2$-CW complex. Therefore, one can use a general localization theorem for $\Z_2$-action. 
See \cite[Theorem 2.1, page 44]{Ma98}, for example.
\end{proof}

Now, we prove \cref{conn sum of qm}. 
\begin{proof}[Proof of \cref{conn sum of qm}]

We use the following knot cobordisms: 
\begin{itemize}
    \item $(W_\#, S_\# )  : (S^3, K) \cup (S^3, K') \to (S^3, K\# K')$, 
    \item $(-W_\#, -S_\# )  :  (S^3, K\# K') \to (S^3, K) \cup (S^3, K')$, 
\end{itemize}
where $W_\#$ is a ($4$-dimensional) 1-handle cobordism from $S^3 \cup S^3 \to S^3$ and $S_\#$ is also a (2-dimensional)  1-handle cobordism from $K \cup K' \to K \# K'$. Associated with these cobordisms, we have (metric dependent) $\Z_2$-equivariant cohomological Bauer--Furuta invariants 
\[
BF_{(W_\#, S_\# )} : SWF(\Sigma(K)) \wedge SWF(\Sigma(K')) \to  SWF(\Sigma(K\# K'))
\]
and 
\[
BF_{(-W_\#, -S_\# ) } : SWF(\Sigma(K\# K')) \to SWF(\Sigma(K)) \wedge SWF(\Sigma(K')). 
\]
We claim the $Q$-localized induced maps $BF_{(W_\#, S_\# )}^{*, Q\text{-loc}} $ and $BF_{(-W_\#, -S_\# ) }^{*, Q\text{-loc}} $ are isomorphisms. 
In order to see this, we consider the composition: 
\[
(W_\#, S_\# )\circ(-W_\#, -S_\# ). 
\]

One can find an embedding of $(S^0 \times D^4, S^0 \times D^2) $ into $(-W_\#, -S_\# ) \circ (W_\#, S_\# ) $ such that the surgery along it is pairwisely diffeomorphic to $([0,1] \times S^3, [0,1]\times K  ) \cup ([0,1] \times S^3, [0,1]\times K'  )$. Again using a connected sum gluing theorem based on \cref{BFsurgery}, 
\[
BF_{(-W_\#, -S_\# ) \circ (W_\#, S_\# )} = BF_{([0,1] \times S^3, [0,1]\times K  ) \cup ([0,1] \times S^3, [0,1]\times K'  )}
\]
and 
we know $BF_{([0,1] \times S^3, [0,1]\times K  ) \cup ([0,1] \times S^3, [0,1]\times K'  )}^* $ is isomorphism from \cref{BF=id}. 

Therefore, 
\[
q_M (K\cup K'): = q_M( SWF(\Sigma(K)) \wedge SWF(\Sigma(K')))  , 
\]
here we define $q_M( SWF(\Sigma(K)) \wedge SWF(\Sigma(K')))$ similar to the knot case and use natural homological grading comes from the gradings of $K$ and $K'$.
Moreover, since the above two cobordisms give both directions of maps, we have 
\[
q_M(K \cup K') = q_M( K \# K'). 
\]
On the other hand, we see 
\[
q_M(K) + q_M(K') \leq q_M (K\cup K') 
\]
since we have a homomorphism 
\[
H^*_{\Z_2 } ( SWF(\Sigma K) \wedge SWF(\Sigma K')) \to H^*_{\Z_2 } ( SWF(\Sigma K) ) \otimes H^*_{\Z_2} ( SWF(\Sigma K'))
\]
which induces the isomorphism on the localizations
\[
Q^{-1} H^*_{\Z_2 } ( SWF(\Sigma K) \wedge SWF(\Sigma K')) \to Q^{-1} H^*_{\Z_2 } ( SWF(\Sigma K) ) \otimes H^*_{\Z_2} ( SWF(\Sigma K')). 
\]
This property follows from the naturality of the localization maps 
\[
  \begin{CD}
     H^*_{\Z_2 } ( SWF(\Sigma K) \wedge SWF(\Sigma K')) @>{}>> H^*_{\Z_2 } ( SWF(\Sigma K) ) \otimes H^*_{\Z_2} ( SWF(\Sigma K'))  \\
  @V{}VV    @VV{}V \\
   \mathbb{F}_2[Q^{\pm 1}] \otimes H^* ( SWF(\Sigma K) \wedge SWF( \Sigma K'))  @>{\cong }>>   \mathbb{F}_2[Q^{\pm 1}] \otimes  H^*( SWF(\Sigma K) ) \otimes H^*( SWF(\Sigma K')) 
  \end{CD}
\]
and non-equivariant Kunneth isomorphism.

Therefore, we have 
\[
q_M(K^* ) + q_M((K')^* ) \leq q_M ((K\cup K')^*) . 
\]
Using the mirroring formula \cref{duality}, we have opposite inequality. This gives the conclusion. 
\end{proof}

\subsection{A lower bound for $q_M(K)$}
Let $K \subset S^3$ be a knot.
In this subsection, we give a lower bound for $q_M(K)$ in terms of the Heegaard Floer homology group of the double-branched cover.
Set
\[
m(K):=\min \{ gr^\Q (x) | 0\neq x\in  \widehat{HF}(\Sigma_2(K); \fraks_0), x \text{ is homogeneous}  \}. 
\]

\begin{thm}
For any knot $K \subset S^3$, 
we have
\[
q_M(K)\geq m(-K)-\frac{3}{4}\sigma(K).
\]
\end{thm}

\begin{proof}
In this proof,  we abbreviate
\[
\wt{H}^*=\wt{H}^{*}(SWF(\Sigma_2(K); \s_0)))
\text{ and }
\wt{H}^*_{\Z_2}=\wt{H}^{*}_{\Z_2}(SWF(\Sigma_2(K); \s_0))).
\]
We have the isomorphism due to $HF=HM$  \cite{KLTI, KLTII, KLTIII, KLTIV, KLTV, Ta10I, Ta10II, Ta10III, Ta10IV ,Ta10V, CGHI, CGHII, CGHIII} and Lidman--Manolescu \cite{LM18}, 
\[
\widehat{HF}(\Sigma_2(K); \fraks_0)\cong\wt{H}^* 
\]
the invariant $m(K)$ can be rewritten in terms of Seiberg--Witten Floer stable homotopy type: 
\[
m(K)=\min \{ gr^\Q (x) | 0\neq x\in  \wt{H}^*, x \text{ is homogeneous}  \}.
\]
The fibration
\[
\Z_2 \to SWF(\Sigma_2(K); \s_0))\wedge (E \Z_2)_+\to  SWF(\Sigma_2(K); \s_0))\wedge_{\Z_2} (E \Z_2)_+
\]
gives the Thom--Gysin exact sequence
\[
\cdots \to \wt{H}^{*-1}_{\Z_2}\xrightarrow{Q}\wt{H}^{*}_{\Z_2}\xrightarrow{\pi^*} \wt{H}^{*}\to \cdots.
\]
and we have 
\[
\cok Q\cong \im \pi^* \subset \wt{H}^{*}.
\]
Thus, we have 
\begin{align*}
q^\dagger_M(K) &\geq\min \{ gr^\Q (x) | 0\neq x\in  \wt{H}^*_{\Z_2}, x \text{ is homogeneous}  \}
\\
&=\min \{ gr^\Q (x) | 0\neq x\in  \cok Q, x \text{ is homogeneous}  \}
\\
&\geq \min \{ gr^\Q (x) | 0\neq x\in \wt{H}^* , x \text{ is homogeneous}  \}
=m(K).
\end{align*}
Thus we obtain 
\[
-q^\dagger_M(K)=q^\dagger_M(-K)\geq m(-K).
\]
This gives
\begin{align*}
q_M(K)&=-q^\dagger_M(K)-\frac{3}{4}\sigma(K)\\
& \geq m(-K)-\frac{3}{4}\sigma(K), 
\end{align*}
which is the desired inequality.
\end{proof}
\subsection{$q_M(K)$ for links}
\label{qMforlink}
In this section, we observe our invariant $q_M(K)$ has a natural extension to invariants of oriented links with non-zero determinants. 
For a link $L$ in $S^3$ with non-zero determinant, the following one-to-one correspondence is known 
\[
\{ \text{ orientations on }L\}/\{\pm 1\} \cong \{\text{isomorphism classes of spin structures on } \Sigma(L) \}. 
\]
Moreover, it is checked in \cite{KMT23} that any spin structure on $\Sigma(L)$ is preserved by the covering involution. 
Therefore, depending on an orientation, we have a correspondence 
\[
(L, o) \mapsto SWF(\Sigma(L); \fraks_o), 
\]
where $o$ is an orientation of $L$ and $\fraks_o$ is the corresponding spin structure. 

Then we can repeat the construction of $q_M$: 
Set
\[
q^\dagger_M(L, o) := \min \{ \gr^\Q(x) | Q^n x \neq 0,  \text{ for all }n\geq 0,   x \in H^*_{\Z_2} (SWF(\Sigma_2(L); \s_o); \mathbb{F}_2)\}  
\]
where we assume $x$ is homogeneous and set
\[
q_M(L, o)=q^\dagger_M(-L,-o)-\frac{3}{4}\sigma(L, o).
\]
From now on, we abbreviate the notations for the orientations $o$.  In order to state the invariance of $q_M(L)$, it is convenient to use the notion of {\it oriented $\chi$-concordance} introduced in  \cite{DO12}. 
 Let us first review the definition of $\chi$-concordance. A {\it marked link} is a link in $S^3$ equipped with a marked component. 
For given oriented marked links $L_0$ and $L_1$, we call $L_0$ and $L_1$ are {\it $\chi$-concordant} if $-L^*_0\#L_1$ bounds a smoothly properly embedded oriented surface $F$ in $D^4$
such that
\begin{itemize}
    \item[(i)]  $F$ is a disjoint union of one disk together with annuli;
    \item[(ii)] the boundary of the disk component of $F$ is the marked component of $-L^*_0 \# L_1$. 
\end{itemize}
where $L^*$ means the mirror image of $L$ and $-L^*$ is the one with the opposite orientation to $L^*$.
the connected sum $-L^*_0\#L_1$ is taken along marked components. In \cite{DO12}, it is proven that the set $\tilde{\mathcal{L}}$ of all $\chi$-concordant classes of oriented marked links forms an abelian group with respect to the connected sum along marked components. The group $\tilde{\mathcal{L}}$ is called the {\it link concordance group}. In this paper, we focus on the subgroup $\tilde{\mathcal{F}}$ of $\tilde{\mathcal{L}}$ generated by oriented marked links whose determinants are {\it non-zero}.

\begin{thm}
    The invariant $q_M(L)$ descends to a homomorphism 
    \[
  q_M:   \tilde{\mathcal{F}} \to \Z.
    \]
    
\end{thm}
\begin{proof}
    It is proven in \cite{DO12} that if $L_0$ and $L_1$ are $\chi$-concordant, then the double branched cover of $-L^*_0 \# L_1$ bounds a spin rational homology ball $(W, \fraks)$. Moreover, the restrictions of the spin structure to the boundary $\Sigma_2(-L^*_0 \# L_1 ) = -\Sigma_2 (L_0)\# \Sigma_2 (L_1)$ coincide with the spin structure comes as the connected sum of spin structures on $-\Sigma_2 (L_0)$ and  $\Sigma_2 (L_1)$, which correspond to orientations of $L_0$ and $L_1$ respectively. 
    Moreover, if we write the $\chi$ concordance as $F$, $W$ is obtained as the branched cover along $F$ in $D^4$. Moreover, it is proven that isomorphism classes of spin structures on $W=\Sigma_2(F)$ have one to one correspondence with quasi-orientations of $F$. Moreover, it is not difficult to check these spin structures are invariant under the covering involutions. Therefore, the spin structure $\fraks$ is also preserved by the involution.
    So, the induced $\Z_2$-equivariant Bauer--Furuta invariant is described as 
    \[
   BF_F :  V^+ \to SWF(-L^*_0\# L_1) .
    \]
    This shows $0\leq q_M(-L^*_0\# L_1)$. From the duality and the connected sum formula, we see $q_M(L_0) \leq q_M(L_1)$. The opposite inequality is obtained by considering the opposite orientation of $W$.
    Therefore, $q_M$ defines a map from $\wt{\mathcal{F}}$ to $\Z$. The proof of the homomorphism property of $q_M$ for marked oriented links is completely similar to the knot case. So, we omit it. 
\end{proof}

\section{Equivariant stable homomotopical invariants from contact structures}

In this section, we construct two kinds of new invariants for equivariant spin$^c$ 4-manifolds with contact boundary and for equivariant closed contact 3-manifolds using invariant almost K\"alher cones and Seiberg--Witten equation on them.  

\subsection{Equivariant Bauer-Furuta type invariant for 4-manifold with contact boundary}
Let $G$ be a compact Lie group.
Let $(\wt{X}, \wt{\s})$ be a compact oriented Spin$^c$ 4-manifold  with a smooth $G$-action preserving the orientation and the Spin$^c$ structure.
We also fix a lift of $G$-action to the principal spin$^c$ bundle. 
Assume $b_3(\wt{X})=0$ as in \cite{I19}.
Notice that this condition on $b_3$ implies that the boundary $\wt{Y}=\partial \wt{X}$ is also connected.
Suppose that we have a $G$-invariant positive cooriented contact structure on the boundary and an $G$ equivariant isomorphism $\s|_{\partial \wt{X}} \cong \s_{\wt{\xi}}$.
In this section, we will define a $G$-equivariant Bauer--Furuta refinements of  Kronheimer--Mrowka invariants 
\[
\Psi_{G} (\wt{X},\wt{\xi}, \wt{\s}) : S^{\langle e({S}^+ , \wt{\Phi}_0),  [\wt{X}, \partial \wt{X} ]\rangle} \to S^0
\]
as a $G$-equivariant stable homotopy class. 
Here, the relative Euler number
$\langle e({S}^+ , \wt{\Phi}_0),  [\wt{X}, \partial \wt{X} ]\rangle$ will be explained later and this is equal to the virtual dimension of the Seiberg--Witten moduli space for 4-manifolds with contact boundary given by Kronheimer--Mrowka \cite{KM97}. By applying it to branched covering spaces along embedded surfaces, we shall also consider the invariants of surfaces. 
Recall also that for $G$-representations $V, W, V', W'$, we say two maps $f : V^+ \to W^+$ and $f' : V^+ \to W^+$ are $G$-stably homotopic if there exists a $G$-representation space $U$ such that $f \wedge \id_U : V^+ \wedge U^+ \to W^+ \wedge U^+$ are $G$-equivariantly homotopic, where $V$ and $W$ are some $G$-representation spaces. 
\subsection{Almost K\"ahler cone and $G$-action} 
In order to write a setup for the equivariant stable homotopy version of Kronheimer--Mrowka's invariants, we need to treat invariant almost K\"ahler structures on the open cone of the given contact 3-manifolds. 

Let $G$ be a compact Lie group and $\wt{Y}$ be an oriented rational homology 3-sphere with $G$-action and $\wt{\xi}$ be a $G$-invariant positive cooriented contact structure on $\wt{Y}$.

Take a $G$-invariant contact $1$-form $\lambda$ which is positive on the positively oriented normal field to $\wt{\xi}$ and a $G$-invariant complex structure $J$ on $\wt{\xi}$ compatible with the orientation. 
Note that we can always take a $G$-invariant contact $1$-form $\wt{\lambda}$ for a $G$-invariant contact structure $\wt{\xi}$ compatible with the coorientation by fixing an orbitrary contact form compatible with the coorientation and taking  the average over $G$.

We can take a $G$ invariant compatible complex structure $\wt{J}$ for the symplectic vector bundle $(\xi, d\lambda)$.
This is because there is a standard deformation retract
\[
r: \{\text{Euclid metric on } \xi\}\to \{\text{ compactible complex structure for } (\xi, d\lambda) \}, 
\]
which is $G$-equivariant
(See for example \cite[Proof of Lemma 2.5.5]{MS17}).
Thus if we take a Euclid metric $g$ on $\wt{\xi}$ and denote by $\wt{g}$ its average over $G$, then 
\[
\wt{J}=r(\wt{g}) 
\]
is a desired $G$-invariant complex structure.

Then we can define the $G$-invariant Riemann metric 
$
\wt{g}_1= \wt{\lambda} \otimes \wt{\lambda} + \frac{1}{2} d\wt{\lambda} (\cdot, J \cdot)|_\xi  $
 on $Y$. 
On $\R^{\geq 1}  \times \wt{Y}$, we consider the $G$-invariant Riemannian metric 
\[
\wt{g}_0 := ds^2 + s^2 \wt{g}_1,
\]
and the $G$-invariant symplectic form
\[
\wt{\om}_0 := \frac{1}{2} d(s^2 \wt{\lambda}),
\]
where $s$ is the coordinate of $\R^{\geq 1}$. 
This gives an almost K\"ahler structure on $\R^{\geq 1}  \times \wt{Y}$.

 Then a pair $(\wt{g}_0, \wt{\om}_0)$ determines an almost complex structure $\wt{J}$ on $\R^{\geq 1}  \times \wt{Y}$.  This defines a $Spin^c$ structure 
 \begin{align*}
 \wt{\s}_0:= ( {S}^+_{\R^{\geq 1}  \times \wt{Y}} = \Lambda_{\R^{\geq 1}  \times \wt{Y}}^{0,0} \oplus  \Lambda_{\R^{\geq 1}  \times \wt{Y}}^{0,2}, \ {S}^-_{\R^{\geq 1}  \times \wt{Y}} = \Lambda_{\R^{\geq 1}  \times \wt{Y}^{0,1}} , \\ \wt{\rho} : T^*(\R^{\geq 1}  \times \wt{Y}) \to \Hom ({S}^+_{\R^{\geq 1}  \times \wt{Y}} , {S}^-_{\R^{\geq 1}  \times \wt{Y}} ) )
  \end{align*}
 with natural lifts of the $G$-action, 
  where  
$
\wt{\rho}  = \sqrt{2} \operatorname{Symbol} (\overline{\partial} + \overline{\partial}^* ) . 
 $
 Let  $\wt{\Phi}_0$ be the positive spinor given by
 \[
 \wt{\Phi}_0=(1,0) \in \Om_{\R^{\geq1} \times \wt{Y}}^{0,0} \oplus  \Om_{\R^{\geq1} \times \wt{Y}}^{0,2}= \Gamma (\wt{S}^+|_{\R^{\geq1} \times \wt{Y}}).
 \]
   The {\it canonical $Spin^c$ connection} $\wt{A}_0$ on $\s$ is defined by the equation 
$
 D^+_{\wt{A}_0} \wt{\Phi}_0= 0$
 on $\R^{\geq 1} \times \wt{Y}$.
Note that $(\wt{A}_0, \wt{\Phi}_0)$ is invariant under $G$.

\subsection{Construction of invariants}

Now let $(\wt{X}, \wt{\xi} ,\wt{\fraks})$ be as in the introduction of this section.
In this subsection, we will construct a $G$-equivariant pointed stable homotopy class up to sign
\[
\Psi_{G} (\wt{X},  \wt{\xi},\wt{\fraks}): S^{\langle e({S}^+ , \wt{\Phi}_0),  [\wt{X}, \partial \wt{X} ]\rangle} \to S^0.
\]
of degree $d(\wt{X}, \wt{\xi}, \wt{\fraks} )=\langle e(S^+, \Phi_0) , [\wt{X}, \wt{Y}]\rangle$.

\par

Define a non-compact 4-manifold $\wt{X}^+$ with conical end 
\[
\wt{X}^+ := \wt{X} \cup_{\wt{Y}}  (\R^{\geq 1} \times \wt{Y}).
\]
Pick a $G$-invariant Riemann metric $g_{\wt{X}^+}$ on $\wt{X}^+$ such that $g_{\wt{X}^+}|_{\R^{\geq 1} \times Y}=g_0$. 
Fix a $G$-invariant \Spinc structure $\s_{\wt{X}^+}=(S^{\pm}_{\wt{X}^+}, \rho_{\wt{X}^+}) $ on $\wt{X}^+$ equipped with an isomorphism $\s_{\wt{X}^+} \to \wt{\s}_0$ on $\wt{X}^+\setminus \wt{X}$.
We also fix lifts of the $G$-action on spinor bundles on $\s_{\wt{X}^+}$ which coincide with the lifts taken in the previous section. 
We will omit this isomorphism in our notation.

Fix a smooth $G$-invariant extension of $(\wt{A}_0, \wt{\Phi}_0)$ on $\wt{X}^+$.
We also fix a nowhere zero proper extension $\sigma$ of $s \in \R^{\geq 1}$ coordinate to all of $\wt{X}^+$ which is $0$ on $\wt{X} \setminus \nu (\partial \wt{X})$, where $\nu (\partial \wt{X})$ is a small collar neighborhood of $\partial \wt{X}$ in $\wt{X}$.

On $\wt{X}^+$, {\it weighted Sobolev spaces} 
\[
\widehat{\cU}_{\wt{X}^+}=L^2_{k, \alpha, A_0 }(i \Lambda^1_{\wt{X}^+}\oplus S^+_{\wt{X}^+}) \text{ and }
\]
\[
\widehat{\cV}_{\wt{X}^+}=L^2_{k-1, \alpha, A_0 }(i\Lambda^0_{\wt{X}^+}\oplus i\Lambda^+_{\wt{X}^+}\oplus S^-_{\wt{X}^+})
\]
are defined using $\sigma$ for a positive real number $\al \in \R$ and $k \geq 4$, where $S^+_{\wt{X}^+}$ and $S^-_{\wt{X}^+}$ are positive and negative spinor bundles and the Sobolev spaces are given as completions of the following inner products: 
\begin{align}\label{inner}
\langle s_1, s_2\rangle_{L^2_{k, \alpha, A} } := \sum_{i=0}^k \int_{\wt{X}^+} e^{2\alpha \sigma} \langle \nabla^i_A s_1, \nabla^i_A s_2  \rangle  \operatorname{dvol}_{\wt{X}^+}, 
\end{align}
where the connection $\nabla^i_A$ is the induced connection from $A$ and the Levi-Civita connection. Note that $G$ also acts on the spaces $\widehat{\cU}_{\wt{X}^+}$ and $\widehat{\cV}_{\wt{X}^+}$.

Fix a sufficiently small positive real number $\al$. 
The invariant $\Psi_{G}(\wt{X}, \wt{\xi})$ is obtained as a $G$-invariant finite-dimensional approximation of the Seiberg--Witten map
\begin{align}\label{FX+}
\begin{split}
&\widehat{\mathcal{F}}_{\wt{X}^+}:  \widehat{\cU}_{\wt{X}^+}\to \widehat{\cV}_{\wt{X}^+}\\
&(a, \phi)\mapsto (d^{*_\alpha}a, d^+a-\wt{\rho}^{-1}(\phi\wt{\Phi}^*_0+\wt{\Phi}_0\phi^*)_0-\wt{\rho}^{-1}(\phi\phi^*)_0, D^+_{\wt{A}_0}\phi+\wt{\rho}(a)\wt{\Phi}_0+\wt{\rho}(a)\phi),  
\end{split}
\end{align}
where $d^{*_\alpha}$ is the $L^2_\alpha$ formal adjoint of $d$. 
By construction, the map $\widehat{\mathcal{F}}_{\wt{X}^+}$ is $G$-equivariant. 
\par
The finite-dimensional approximation goes as follows.
We decompose $\widehat{\cF}_{\wt{X}^+}$ as $\widehat{L}_{\wt{X}^+}+\widehat{C}_{\wt{X}^+}$ where 
\[
 \widehat{L}_{\wt{X}^+}(a, \phi)=(d^{*_\alpha}a, d^+a-\wt{\rho}^{-1}(\phi\wt{\Phi}^*_0+\wt{\Phi}_0\phi^*)_0, D^+_{\wt{A}_0}\phi+ \rho(a)\wt{\Phi}_0)
 \]
 and 
 \[
  \widehat{C}_{\wt{X}^+}(a, \phi) = (0, -\wt{\rho}^{-1}(\phi\phi^*)_0, \wt{\rho}(a)\phi). 
  \]
Then $\widehat{L}_{\wt{X}^+}$ is a linear Fredholm $G$-equivariant operator and  $\widehat{C}_{\wt{X}^+}$ is quadratic, compact and $G$-equivariant. (Here we used $\alpha>0$. )

Pick an increasing sequence of $G$-invariant finite-dimensional subspaces $\widehat{\cV}_{\wt{X}^+, n} \subset \cV_{\wt{X}^+}\, (n \in \Z^{\geq 1})$ such that 
\begin{itemize}
    \item For any $\gamma \in \widehat{\cV}_{\wt{X}^+}$, 
    \[
 \| \pr_{\widehat{\cV}_{\wt{X}^+, n}}( \gamma) - \gamma \|_{\widehat{\cV}_{\wt{X}^+}}   \to 0  \text{ as } n\to \infty
\]
and 
\item $\Coker \widehat{L}_{\wt{X}^+} := (\im \widehat{L}_{\wt{X}^+})^{\perp_{L^2_{k-1, \alpha}}} \subset \widehat{\cV}_{\wt{X}^+, 1}$.  
\end{itemize}
Let 
\[
\widehat{\cU}_{\wt{X}^+, n}=\wh{L}^{-1}(\widehat{\cV}_{\wt{X}^+, n})\subset \wh{\cU}_{\wt{X}^+},
\]
 and 
\[
\mathcal{F}_{\wt{X}^+, n}:=  \pr_{\widehat{\cV}_{\wt{X}^+, n}}   \circ \mathcal{F}_{\wt{X}^+}: \widehat{\cU}_{\wt{X}^+, n}\to \widehat{\cV}_{\wt{X}^+, n}. 
\]
We can show that for a large $R>0$, a small $\varepsilon$ and a large $n$, we have a well-defined $G$-equivariant map
\[
\mathcal{F}_{\wt{X}^+, n}: B(\widehat{\cU}_{\wt{X}^+, n}, R)/S(\widehat{\cU}_{\wt{X}^+, n}, R)\to B(\widehat{\cV}_{\wt{X}^+, n}, \varepsilon)/S(\widehat{\cV}_{\wt{X}^+, n}, \varepsilon).
\]
The stable homotopy class of $\mathcal{F}_{\wt{X}^+, n}$ defines {\it the Bauer--Furuta version of Kronheimer--Mrowka's invariant} 
\[
\Psi_{G} (\wt{X}, \wt{\xi}, \wt{\fraks} ) : S^{\langle e(S^+, \Phi_0) , [\wt{X}, \partial \wt{X}]\rangle } \to S^0 ,
\] where $e(S^+_{\wt{X}}, \Phi_0) \in H^4(\wt{X}, \partial \wt{X})$ is the relative Euler class of $S^+_{\wt{X}}$ with respect to the section $\Phi_0|_{\wt{Y}} $. It is not hard to see that the stable homotopy class of $\Psi_{G} (\wt{X}, \wt{\xi}, \wt{\fraks} )$ is independent of additional data.

Moreover, we have a diffeomorphism invariance:

\begin{lem}\label{Invariance of eq KM}
Let $(\wt{X}, \wt{\xi}, \wt{\fraks} )$ and $(\wt{X}', \wt{\xi}', \wt{\fraks}')$ be pairs of $G$-equivariant \Spinc 4-manifolds with contact boundary. Suppose there is a $G$-equivariant diffeomorphism $f : \wt{X} \to \wt{X}'$ so that 
$f_* \wt{\xi} =\wt{\xi}'$ and $f^* \wt{\fraks} = \fraks'$.
Then, we have 
\[
\Psi_{G} (\wt{X}, \wt{\xi}, \wt{\fraks} )=\pm \Psi_{G} (\wt{X}', \wt{\xi}', \wt{\fraks}' ) 
\]
up to $G$-equivariant stable homotopy. 
\end{lem}
\begin{proof}
The pull-back by $f$ induces a commutative diagram of $G$-equivariant maps: 
\[
  \begin{CD}
     S^{\langle e(S^+, \Phi_0) , [\wt{X}, \partial \wt{X}]\rangle } @>{\Psi_{G} (\wt{X}, \wt{\xi}, \wt{\fraks} )}>> S^0 \\
  @V{f^*}VV    @V{f^*}VV \\
     S^{\langle e(S^+, \Phi_0) , [\wt{X}', \partial \wt{X}']\rangle }   @>{\Psi_{G} (\wt{X}', \wt{\fraks}',\wt{\xi}')}>>  S^0. 
  \end{CD}
\]
Note that $f^*$ is $\Z_p$-euqivariant homeomorphism. 
This implies the conclusion. 
\end{proof}

Now we restrict ourselves to $G=\Z_p$ for a prime $p$.
In this case, let us write $\Psi_{G=\Z_p} (\wt{X}  , \wt{\xi}, \wt{\s})=\Psi_{(p)} (\wt{X}  , \wt{\xi}, \wt{\s})$.
We have the following non-vanishing theorem: 
\begin{prop}\label{non-van for eq KM}
Let $p$ be a prime number.
For a 4-dimensional $\Z_p$-equivariant weak symplectic filling $(\wt{X}, \wt{\om})$ of a $\Z_p$-equivariant contact
3-manifold $(\wt{Y} ,\xi)$,
 then 
$\Psi_{(p)} (\wt{X}  , \wt{\xi}, \frak{s}_{\wt{\om}})$ is $\Z_p$-stably homotopy equivalence, where $\fraks_{\wt{\om}}$ is the induced $\Z_p$-invariant spin$^c$ structure.

\end{prop}
\begin{proof}

By \cite{James-Segal,To79}, 
it is enough to prove the map $\Psi_{(p)}(\wt{X},\wt{\fraks},   \wt{\xi})^{\Z_p}$
induced on $\Z_p$ fixed points has mapping degree $\pm 1$
and $\Psi_{(p)} (\wt{X},\wt{\fraks},   \wt{\xi})$ is a homotopy equivalence as non-equivariant spaces.
The latter follows from  \cite[Corollary 4.3]{I19}, so it is enough to show the former claim.
When we perturb the Seiberg--Witten equation as in \cite{KM97}, 
the moduli space of the Seiberg--Witten solutions satisfies
\[
\mathcal{M}^{\Z_p} \subset \mathcal{M}=\{\operatorname{pt}\}, 
\]
and since the canonical solution is $\Z_p$ invariant by our construction of the almost K\"ahler structure, we have 
\[
\mathcal{M}^{\Z_p} =\{\operatorname{pt}\}.
\]
Furthermore, this solution is cut out transversally because
the linearization $L$ of the Seiberg--Witten equation with local slice satisfies
\[
\ker L^{\Z_p} \subset \ker L=0
\]
\[
(\ker L^*)^{\Z_p} \subset \ker L^*=0, 
\]
when we perturb the equation as above.
The mapping degree of the map
$\Psi_{(p)} (\wt{X},\wt{\fraks},   \wt{\xi})^{\Z_p}$
agrees with $\#\mathcal{M}^{\Z_p}=\pm 1$ by the argment of \cite[Theorem 4.1]{I19}. This completes the proof.

\end{proof}
We now define invariants of properly embedded surfaces in 4-manifolds with contact boundaries.

Let $X$ be a 4-manifold with contact homology $3$-sphere boundary $(Y, \xi)$ and $S$ be a properly embedded surface in $X$ such that $K= S \cap Y$ is a transverse knot in $Y$.
Suppose the homology class of $S$ is divisible by $p$. Suppose $b_3( \Sigma_p(S)) =0$. We also fix a spin$^c$ structure $\wt{\fraks}$ on $\Sigma_p(S)$ such that 
\[
\tau^* \wt{\fraks} \cong \wt{\fraks} \text{ and }\wt{\fraks}|_{\Sigma_p(K)} \cong \fraks_{\wt{\xi}}, 
\]
where $\fraks_{\wt{\xi}}$ is the $\Z_p$-invariant spin$^c$ structure on $\Sigma_p(K)$ induced from the Plamanevskaya's $\Z_p$-invariant contact structure. 
Depending on these data, we have $\Z_p$-invariant spin$^c$ 4-manifold $\wt{X} = \Sigma_p(S)$ with $\Z_p$-invaeriant contact boundary. 
Now, we define 
\[
\Psi_{\Z_p}( X, S, \xi, \wt{\fraks}) :=   \Psi_{\Z_p} (\wt{X}, \wt{\xi}, \wt{\fraks} )  . 
\]

Next, we observe non-vanishing results on $\Psi (X, S, \xi, \wt{\fraks} )$, which is a corollary of \cref{non-van for eq KM} combined with \cref{double branch symp}.

\begin{prop}
Let $(X, \xi, S)$ be a tuple of \Spinc 4-manifold $X$   with contact boundary $(Y, \xi)$ and a properly embedded surface $S$ in $X$ such that $\partial S \subset (Y, \xi)$ is a transverse knot and $[S]$ is divisible by $p$. Suppose S is isotoped rel boundary into a symplectic surface, then  
\[
\Psi_{\Z_p}( X, S, \xi, \wt{\fraks}) = \pm 1
\]
up to $\Z_p$-equivariant stable homotopy for a $\Z_p$-invariant spin$^c$ structure  $\wt{\fraks}$ obtained as the induced spin$^c$ structure arises as the $\Z_p$-invariant symplectic structure on $\Sigma_p(S)$.
\end{prop}
\subsection{Construction of stable homotopy transverse knot invariants}
Let $p$ be a prime number.
Let $(Y, \xi)$ be a rational homology 3-sphere, equipped with a positive cooriented contact structure and let $K \subset (Y, \xi)$ be a transverse knot.
 In this section, we will construct a transverse knot invariant as a stable homotopy class 
\[
\widecheck{\mathcal{C}}(Y, \xi, K): \Sigma^{d_3(\Sigma(K), \wt{\xi})+\frac{1}{2}}SWF(\Sigma(K), \s)  \to S^0
\]
for a transverse knot for each prime number $p$.
using the construction given in \cite{IT20}. 

We can define a cohomological transverse knot invariant for any $\Z_p$ equivariant homology theory and in particular, we will introduce the invariant for $\Z_p$ equivariant ordinary homology
\[
c_{(p)}(Y,  \xi, K)  \in HSW^{\Z_p}_*(Y, K; \Z_p)= \wt{H}_*^{\Z_p} (SWF(\Sigma_p(K); \fraks_0) )
\]
using $\widecheck{\mathcal{C}}(Y, \xi, K)$, where $\fraks_0$ is the unique $\Z_p$-invariant spin structure on $\Sigma_p(K)$. 
For a related construction, see \cite{Br23}. 

\subsection{Construction of the equivariant contact invariant}
The transverse knot invariant above is nothing but a $\Z_p$ equivariant version of the contact invariant introduced by the authors \cite{IT20}, applied to cyclic $p$-th brahcned coverings along a transverse knot.
In this section, we will review the construction in \cite{IT20} and see that it can be straightforwardly adapted to equivariant setting.
\par
Let $G$ be a compact Lie group.
Let $\wt{Y}$ be a rational homology 3-sphere with $G$-action and $\wt{\xi}$ be a $G$-invariant contact structure. As in the construction of $\Psi_{G} (\wt{X}, \wt{\xi}, \wt{\fraks} )$, we consider a $G$-invariant metric on 
$N^+:= \R^{\geq 0}  \times \wt{Y}$ which is a $G$-invariant extension of $\wt{g}_0$ and product on $[0,\frac{1}{2}] \times \wt{Y}$. We also call this metric $\wt{g}_0$.
 The Riemannian manifold $\wt{N}^+$ is what we mainly consider to define our invariant. 
We extend $\wt{\om}_0$ to a self-dual $2$-form with $|\wt{\om}_0(s, y)|= \sqrt{2}$ which is translation invariant on $[0, \frac{1}{2}] \times \wt{Y}$. 

We first fix the following data and notations: 
\begin{itemize}
    \item A $G$-invariant contact form $\wt{\lambda}$ and a $G$-invariant complex structure $J$ on $\wt{\xi}$.
    \item Denote by the almost K\"ahler structure on $\R^{\geq 1}\times \wt{Y}$ by 
    \[
    g_0 = ds^2 + s^2 g_1, \quad \om_0 = \frac{1}{2} d(s^2 \wt{\lambda} ).
    \] 
    \item An $G$-invariant extension of $\om_0$ to $N^+$ which is translation invariant on $[0, \frac{1}{2 }]\times \wt{Y}$.
    \item A $G$-invariant metric on $N^+$ which is an extension of $\wt{g}_0$ and product on $[0,\frac{1}{2}] \times \wt{Y}$. We also call this metric $\wt{g}_0$. 
    \item 
Denote by $S^\pm$ the spinor bundles constructed as in the previous section.
\item A $G$-invariant canonical configuration $(\wt{A}_0, \wt{\Phi}_0)$ on $[1, \infty)\times \wt{Y}$ constructed as in the previous section and a smooth extension of it to $N^+$ so that $(\wt{A}_0, \wt{\Phi}_0)$ is translation invariant in $[0, \frac{1}{2}]\times \wt{Y}$. 
\end{itemize}
For small positive real number $\alpha$, we define the {\it double Coulomb slice} by
\[
\mathcal{U}_{k, \alpha}(N^+)  :=    L^2_{k, \alpha}( i\Lambda_{N^+}^1 )_{CC} \oplus L^2_{k, \alpha} ( S^+_{N^+}) , 
\]
where $L^2_{k, \alpha}( i\Lambda_{N^+}^1 )$ and $L^2_{k, \alpha} ( S^+_{N^+})$ are the completions of the inner products with respect to $L^2_{k, \alpha, \nabla_{LC}}( i\Lambda_{N^+}^1 )  $ and $L^2_{k, \alpha, A_0} ( S^+_{N^+})$ and
\[
L^2_{k, \alpha}( i\Lambda_{N^+}^1 )_{CC} := \{ a  \in L^2_{k, \alpha}( i\Lambda_{N^+}^1) | d^{*_{\alpha}}  a=0, d^{*} {\bf t}a=0 \} . 
\]
Here ${\bf t}$ is the restriction of $1$-forms as differential forms and $d^{*_{\alpha}}$ is the formal adjoint of $d$ with respect to $L^2_{\alpha}$.

On $N^+$, we have the {\it Seiberg--Witten map}
\begin{align*}
&\mathcal{F}_{N^+}  :\mathcal{C}_{k, \alpha} (N^+)  \to  L^2_{k-1, \alpha}( i\Lambda_{N^+}^+ \oplus S^-_{N^+}) \text{ by } \\
&\mathcal{F}_{N^+} (A,  \Phi ) := \left( \frac{1}{2} F^+_{A^t}-\wt{\rho}^{-1} ( \Phi \Phi^*)_0 - (\frac{1}{2}  F^+_{\wt{A}^t_0}- \rho^{-1}( \wt{\Phi}_0 \wt{\Phi}^*_0)_0 ), D^+_A \Phi  \right), 
\end{align*}
which is $G$-equivariant. 
Again, we can write $\mathcal{F}_{N^+}$ as the sum of the linearization $L_{N^+}$ at $(\wt{A}_0, \wt{\Phi}_0)$ and the remaining part $C_{N^+}$. 
We carry out a finite-dimensional approximation of the map 
\[
\mathcal{F}_{N^+}  : \mathcal{U}_{k, \alpha}
 \to  \mathcal{V}_{k-1, \alpha} \oplus V (\partial N_+) , 
 \]
 where $\mathcal{U}_{k, \alpha}= L^2_{k, \alpha}( i\Lambda_{N^+}^1 )_{CC} \oplus   L^2_{k, \alpha} (S^+_{N^+})$ and $\mathcal{V}_{k-1, \alpha} = L^2_{k-1, \alpha}( i\Lambda_{N^+}^0\oplus i\Lambda_{N^+}^+ ) \oplus L^2_{k-1, \alpha}(S^-_{N^+})$.
In this section, we fix a small weight $\alpha \in (0, \infty)$ as in \cite{IT20}. 
  Take sequences of $G$-invariant subspaces
 \[
  \mathcal{V}_1 \subset  \mathcal{V}_2 \subset \cdots \subset  \mathcal{V}_{k-1, \alpha} \text{ and } V^{\lambda_1}_{-\lambda_1} \subset V^{\lambda_2}_{-\lambda_2} \subset \cdots   \subset V (\partial N_+)  
  \]
  such that 
  \begin{itemize}
  \item[(i)] $(\im L_{N^+} + p^{0}_{-\infty }\circ r)^{\perp_{\mathcal{V}_{k-1, \alpha} \oplus V (\partial N_+)}  } \subset \mathcal{V}_n \oplus  V^{\lambda_n}_{-\lambda_n} (\partial N_+)  $ for any $n$
  
 \item[(ii)] the $L^2$-projection $P_n :  \mathcal{V}_{k-1, \al}  \oplus V (\partial N_+) \to \mathcal{V}_n \oplus  V^{\lambda_n}_{-\lambda_n} (\partial N_+)$ satisfies
 \[
 \lim_{n \to \infty} P_n (v) =v 
 \]
 for any $ v \in \mathcal{V}_{k-1, \al}  \oplus V (\partial N_+)$.
  \end{itemize}
  Then we define a sequence of subspaces 
  \[
  \mathcal{U}_n :=  (L_{N^+}+ p^{\lambda_n}_{-\lambda_n}\circ r)^{-1} ( \mathcal{V}_n \oplus  V^{\lambda_n}_{-\lambda_n} ). 
  \]
 This gives a family of the approximated Seiberg-Witten map is given by 
  \[
  \{  \mathcal{F}_n : =  P_n (L_{N^+}+C_{N^+}, p^{\lambda_n}_{-\lambda_n} \circ r ) \colon   \mathcal{U}_n \to \mathcal{V}_n \oplus  V^{\lambda_n}_{-\lambda_n} (\partial N_+) \} ,  
  \]
  which is a sequence of $G$-invariant maps. 
 By combining compactness result ensured in \cite{IT20}, we can associate $G$-equivariant map
 \[
\overline{B}( \mathcal{U}_n ; R) / S( \mathcal{U}_n ; R)  \to ( \mathcal{V}_n/  \overset{\circ}{B}(\mathcal{V}_n, \varepsilon_n)^c) \wedge (N_n/ L_n) . 
\]
Then, by applying the formal (de)suspension 
\[
\Sigma^{(\frac{1}{2}-d_3(-\wt{Y}, [\xi])-2n(-\wt{Y}, g_{\wt{Y}}, \s))\R \oplus (-V^0_{-\lambda_n} ) \oplus (-\mathcal{V}_n) }
\]
 to 
\[
\overline{B}( \mathcal{U}_n ; R) / S( \mathcal{U}_n ; R)  \to ( \mathcal{V}_n/  \overset{\circ}{B}(\mathcal{V}_n, \varepsilon_n)^c) \wedge (N_n/ L_n) , 
\]
we obtain a map stably written by 
\begin{align}\label{homotopy}
\mathcal{C}_{G} (\wt{Y}, \wt{\xi}) : S^0 \to \Sigma^{ (\frac{1}{2}-d_3(-\wt{Y}, [\wt{\xi}]))\R}SWF(-\wt{Y}, \s_{\wt{\xi}}). 
\end{align}

The $G$-equivariant stable homotopy class of \eqref{homotopy} is called {\it equivariant Floer homotopy contact invariant}. When we put $G=\Z_p$, we write $\mathcal{C}_{G} (\wt{Y}, \wt{\xi})$ by $\mathcal{C}_{(p)} (\wt{Y}, \wt{\xi})$. Moreover, for a given transverse knot $K$ in a contact homology $3$-sphere, we can associate the $\Z_p$-invariant contact structure $\wt{\xi}$ on $\wt{Y}= \Sigma_p(K)$ described in \cref{Contact structures on branched coverring spaces}.
In this case, we denote $\mathcal{C}_{(p)} (\wt{Y}, \wt{\xi})$ by $\mathcal{C}_{(p)} (Y, \xi, K)$, which is called the {\it stable homotopy transverse knot invariant}.

\subsection{Equivariant Spaniel--Whitehead duality map}
For a rational homology 3-sphere $\wt{Y}$ with a spin$^c$ structure $\wt{\fraks}$.
Manolescu constructed the duality map $\epsilon$  \cite{Man03}: 
\[
\epsilon : SWF(\wt{Y},\wt{\fraks})\wedge  SWF(-\wt{Y},\wt{\fraks}) \to S^0.
\]
Suppose a compact Lie group $G$ acts on $\wt{Y}$ preserving the spin$^c$ structure $\wt{\fraks}$. 
Also, fix a lift of the $G$-action to the principal spin$^c$ bundle. In this case, is not so hard to construct $\epsilon$ equivariantly. 
We write this equivariant version by 
\[
\epsilon_G : SWF(\wt{Y},\wt{\fraks})\wedge  SWF(-\wt{Y},\wt{\fraks}) \to S^0, 
\]
where we are implicitly fixing a $G$-representation of $S^0$. This map satisfies the definition of $G$-$V$-dual. See  \cite{BH}. 

\subsection{Definition of cohomological invariants $c_{(p)}(Y, \xi, K)$}
Let $\wt{Y}$ be a rational homology $3$-sphere with $G$-action, where $G$ is a compact Lie group. Let $\wt{\xi}$ be a contact structure on $\wt{Y}$ preserved by $G$.

 By using the duality map $\epsilon_G$, we often regard \eqref{homotopy} as 
\[
\Sigma^{-\frac{1}{2} - d_3(Y, [\wt{\xi}]) } SWF(\wt{Y}, \s_{\wt{\xi}}) \xrightarrow{\mathcal{C} _G(\wt{Y}, \wt{\xi}) \wedge \id } 
\Sigma^{ \frac{1}{2}-d_3(-\wt{Y}, [\wt{\xi}])}SWF(-\wt{Y}, \s_{\wt{\xi}}) \wedge \Sigma^{-\frac{1}{2} - d_3(\wt{Y}, [\wt{\xi}]) } SWF(\wt{Y}, \s_{\wt{\xi}}) \xrightarrow{\epsilon_G} S^0 . 
\]
We write this composition by 
\[
\widecheck{\mathcal{C}}_{G} (\wt{Y}, \wt{\xi}) : \Sigma^{-\frac{1}{2} - d_3(\wt{Y}, [\xi]) } SWF(\wt{Y}, \s_\xi) \to S^0. 
\]

Now we define homological invariants $c_{G}(Y, K, \xi)$ as follows. 

\begin{defn}
Considering a $G$-equvariant map 
\[
\mathcal{C}_{G} (\wt{Y}, \wt{\xi}) : S^0 \to \Sigma^{ (\frac{1}{2}-d_3(-\wt{Y}, [\xi]))\R}SWF(-\wt{Y}, \s)
\]
and applying the Borel construction, we have the fiber bundle
\begin{align*}
S^0\wedge_{G} (EG)_+ \to \Sigma^{ (\frac{1}{2}-d_3(-Y, [\xi]))\R}SWF(-\wt{Y}, \s)\wedge_{G}  (EG)_+ . 
\end{align*}
Then, by applying the homology, we obtain a $H^*(BG)$-module homomorphism 
\[
{\mathcal{C}}_{G} (\wt{Y}, \wt{\xi})_* : \wt{H}_0^{G} (S^0) \to \wt{H}_{-\frac{1}{2} - d_3(\wt{Y}, [\xi]) }^{G} (SWF(-\wt{Y}) ). 
\]
We define 
\[
c_{G}(\wt{Y}, \wt{\xi}):= {\mathcal{C}}_{G} (\wt{Y}, \wt{\xi})_*(1) \in \wt{H}^{G} _{-\frac{1}{2} - d_3(\wt{Y}, [\xi]) } (SWF(-\wt{Y}, \fraks_{\wt{\xi}}) ) . 
\]
When $G=\Z_p$, we write it by $c_{(p)}(\wt{Y}, \wt{\xi})$.
Now let us assume $G=\Z_p$, so that we have the action of the trivial extension $S^1\times \Z_p$ on configurations.
Consider 
\begin{align*}
S^0\wedge_{\Z_p} (E\Z_p)_+ \to \Sigma^{ (\frac{1}{2}-d_3(-Y, [\xi]))\R}SWF(-\wt{Y}, \s)\wedge_{\Z_p}  (E\Z_p)_+ \wedge (ES^1)_+\\
\to \Sigma^{ (\frac{1}{2}-d_3(-\wt{Y}, [\xi]))\R}SWF(-\wt{Y}, \s) \wedge_{\Z_p \times S^1} (E(\Z_p\times S^1))_+. 
\end{align*}
Then, by applying the homology, we obtain a $H^*_{S^1\times \Z_p}$-module homomorphism 
\[
{\mathcal{C}}_{(p)} (\wt{Y}, \wt{\xi})_* : \wt{H}_0^{S^1\times \Z_p} (S^0) \to \wt{H}_{-\frac{1}{2} - d_3(\wt{Y}, [\xi]) }^{S^1\times \Z_p} (SWF(-\wt{Y}) ). 
\]

For a contact homology 3-sphere $(Y, \xi)$ with a transverse knot $K$ in $(Y, \xi)$ and a prime number $p$, we put $\wt{Y} = \Sigma_p(K)$ with the $\Z_p$ invariant contact structure $\wt{\xi}$ given in \cref{Contact structures on branched coverring spaces}. 

We define 
\[
c_{S^1\times \Z_p}(Y, \xi, K):= {\mathcal{C}}_{(p)} (\wt{Y}, \wt{\xi})_*(1) \in \wt{H}^{S^1\times \Z_p} _{-\frac{1}{2} - d_3(\wt{Y}, [\xi]) } (SWF(-\Sigma_p(K), \fraks_{\wt{\xi}}) ) , 
\]
where $\fraks_{\wt{\xi}}$ is the induced $\Z_p$-invariant spin$^c$ structure on $\Sigma_p(K)$, which is spin.
We also put 
\[
c_{(p)}(Y, \xi, K) := c_{(p)} (\Sigma_p(K), \wt{\xi}). 
\]
When $p=2$, we will also use the sequence of invariants 
\[
c^j_{(2)}(Y, \xi, K):= \mathcal{C}_{\Z_2} (\wt{Y}, \wt{\xi})^*(Q^{-j}) \in \wt{H}^{S^1\times \Z_p} _{-\frac{1}{2} - d_3(\wt{Y}, [\xi]) - j  } (SWF(-\Sigma_2 (K)) ) .
\]

From the constructions, we have 
\[
Q c^j_{(2)}(Y, \xi, K) = c^{j-1}_{(2)}(Y, \xi, K)
\]
if $j \geq 1$ and $Q c^{0}_{(2)}(Y, \xi, K)=0$.

\end{defn}
We call $\{ c_{S^1\times \Z_p}(Y, \xi, K)\} $ the {\it homological (equivariant) transverse knot invariants}. 
The Thom isomorphism implies the following well-definedness. 
\begin{prop}
The element $c_{G}(\wt{Y}, \wt{\xi})$ does not depend on the choices of representatives of maps 
\[
{\mathcal{C}}_{G} (\wt{Y}, \wt{\xi}) : S^0 \to \Sigma^{ (\frac{1}{2}-d_3(-\wt{Y}, [\wt{\xi}]))\R}SWF(-\wt{Y}, \s). 
\]
In particular, $c_{(p)}(Y, \xi, K)$ and $c_{(2)}^j(Y, \xi, K)$ depend only on the transverse isotopy class of $K$.
\end{prop}
 
\subsection{Thom--Gysin exact sequence} 
We use Thom--Gysin exact sequence to see that our transverse knot invariants lie in the $U$-towers in equivariant homology. 
In this section, we put $G=\Z_p$.

 We consider the following $S^1$-bundle: 
\begin{align}\label{Borel}
 \Sigma^{ (\frac{1}{2}-d_3(-\wt{Y}, [\xi]))\R}SWF(-\wt{Y}, \s)\wedge_{\Z_p}  (E(\Z_p\times S^1))_+
\to  \Sigma^{ (\frac{1}{2}-d_3(-\wt{Y}, [\xi]))\R}SWF(-\wt{Y}, \s) \wedge_{\Z_p \times S^1} (E(\Z_p\times S^1))_+.
\end{align}
Then we consider the Thom--Gysin exact sequence associated to the $S^1$-bundle \eqref{Borel}: 
 \begin{align*}
  \cdots \to \wt{H}_*^{\Z_p} (\Sigma^{\frac{1}{2}-d_3(-\wt{Y}, [\xi]) } SWF (-\wt{Y}, \s_{\wt{\xi}}) )) \to 
 \wt{H}_{*-1}^{\Z_p\times S^1} (\Sigma^{\frac{1}{2}-d_3(-\wt{Y}, [\xi]) } SWF (-\wt{Y}, \s_{\wt{\xi}}) )) \\ 
 \xrightarrow{U}  \wt{H}_{*+1}^{S^1\times \Z_p} (\Sigma^{\frac{1}{2}-d_3(-\wt{Y}, [\wt{\xi}]) } SWF (-\wt{Y}, \s_{\wt{\xi}}) ))\to \wt{H}_{*+1} ^{\Z_p} (\Sigma^{\frac{1}{2}-d_3(-\wt{Y}, [\wt{\xi}]) } SWF (-\wt{Y}, \s_{\wt{\xi}}) ))  \to \cdots . 
 \end{align*}
 This exact sequence implies the following vanishing property of $UQ^m c_{S^1\times \Z_2} (Y, \xi, K )$. 
 
\begin{lem}\label{Uc=0}
Let $(Y, \xi)$ be a contact homology 3-sphere and $K$ a transverse knot in $(Y, \xi)$. Fix a prime number $p$.
  We have 
    \[
 U_\dagger c_{S^1\times \Z_2}^j (Y, \xi, K ) =0 
    \]
    for $j \geq 0$ for $p=2$.
For a odd prime $p$, we have 
\[
U_\dagger  c_{S^1\times \Z_p} (Y, \xi, K ) =0.
\]
 for $m \geq 0$ and $m' \geq 0$. 
\end{lem}

\begin{proof}
From the constructions of $c_{S^1\times \Z_2}^j (Y, \xi, K )$ or $c_{S^1\times \Z_p} (Y, \xi, K )$, these are the images of the homomorphisms: 
\[
\wt{H}_*^{\Z_p} (\Sigma^{\frac{1}{2}-d_3(-\wt{Y}, [\wt{\xi}]) } SWF (-\wt{Y}, \s_{\wt{\xi}}) )) \to 
 \wt{H}_{*-1}^{\Z_p\times S^1} (\Sigma^{\frac{1}{2}-d_3(-\wt{Y}, [\wt{\xi}]) } SWF (-\wt{Y}, \s_{\wt{\xi}}) )). 
 \]
 Therefore, from the Thom--Gysin sequence above, we have the desired results. 
 
\end{proof}

Combined with the rank one theorem \cref{rank1 theorem}, we see that our transverse invariants lie in the $U$-towers for knots in $S^3$. 

\begin{thm}[\cref{c in Utower:intro}]\label{c in Utower}
 For any transverse knot $K$ in $ S^3$, the equivariant contact invariant lies in the $U$-tower, i.e. 
 \[
c^j_{S^1\times \Z_2}(S^3,\xi_{\operatorname{std}},  K) = c^j_{S^1\times \Z_2} (K) \in \bigcap_{i \geq 0} \operatorname{Im} U^i_\dagger 
 \]
 for any $j$.
\end{thm} 

\begin{proof}
   It is sufficient to see 
   \[
 c^{j_0}_{S^1\times \Z_2}(K) \in \bigcap_{i \geq 0} \operatorname{Im} U^i_\dagger  
 \]
 for a sufficiently large $j_0$ since $U_\dagger$ and $Q_\dagger$ commute. 
  The claim follows from \cref{towers}, which says that all but finite elements of $\wt{H}^{S^1\times \Z_p}(-\Sigma_p(K))$ can be divided by $U_\dagger$ infinitely many times. 
 \end{proof}

\section{Gluing theorem and a non-vanishing result for contact invariant}

In order to verify the non-triviality of our transverse invariants, it is important to have certain gluing theorems, which are the main topics in this section. 

\subsection{Statement of the gluing theorem} 

In this section, we claim the gluing result among equivariant Kronheimer--Mrokwa's invariant, the transvese knot invariant and the equivariant relative Bauer--Furuta invariant.

\begin{thm} \label{gluing}
Let $G$ be a compact Lie group.
Let $\wt{X}$ be a compact oriented $G$-$Spin^c$ 4-manifold with connected contact boundary $(\wt{Y}, \wt{\xi})$ and $\s_{\wt{X}}$ a $Spin^c$ structure whose restriction on the boundary is compatible with the $Spin^c$ structure induced by $\wt{\xi}$. Suppose $G$-action preserves $\wt{\xi}$ and lifts to a $G$-action on the principal spin$^c$ bundle.
Suppose $b_1(\wt{X})=0$. 
Then 
\[
\epsilon_{G} \circ ( BF_{G}(\wt{X}, \s_X)  \wedge \mathcal{C}_{G}(\wt{Y}, \wt{\xi})  ) = \Psi_{G}(\wt{X}, \s_{\wt{X}, \wt{\xi}}, \wt{\xi})
\]
holds as $G$-equivariant stable maps for the following maps:
\begin{itemize}
    \item $BF_{G}(\wt{X}, \fraks_{\wt{X}}) : (\C^{\frac{c_1^2(\mathfrak{s})-\sigma(\wt{X})}{4}})^+ \to (\R^{b^+(\wt{X})})^+ \wedge SWF(\wt{Y})$ is the equivariant relative Bauer--Furuta invariant defined in \cite[page 36]{BH} forgetting $S^1$-action, 
    \item $\epsilon_{G}: SWF(\wt{Y}, \wt{s}, g) \wedge SWF(-\wt{Y}, \wt{s}, g) \to V^+$ 
    is the $G$-equivariant Spaniel--Whitehead duality map defined in the previous section,
    \item $\mathcal{C}_{G} (\wt{Y}, \wt{\xi}) : S^0 \to \Sigma^{ (\frac{1}{2}-d_3(-\wt{Y}, [\xi]))\R}SWF(-\wt{Y}, \s) $ is the $G$-equivariant contact invariant, 
    \item $\Psi_{G}(\wt{X}, \s_{\wt{X}, \wt{\xi}}, \wt{\xi}) :S^{\langle e(S^+, \Phi_0) , [\wt{X}, \partial \wt{X}]\rangle } \to S^0  $ is the $G$-equivariant Bauer--Furuta type invariant for 4-manifold with contact boundary. 
\end{itemize}

\end{thm}

\begin{proof}
The proof of \cref{gluing} is essentially the same as the proof of the gluing theorem proven in \cite{IT20}. We only need to prove all homotopies given in  \cite{IT20} are $G$-equivariant, and indeed, every homotopy is given as some concrete linear homotopy which is obviously $G$-equivariant. Thus, we have an equivariant homotopy between them. 
\end{proof}

\subsection{Consequences from \cref{gluing}} 
As a consequence of the gluing theorem \cref{gluing}, we have the following pairing formula: 
\begin{thm}\label{paring formula}
Let $p$ be a prime number, $(X, \om)$ a compact symplectic filling bounded by an oriented homology 3-sphere $Y$ with a contact structure $ \xi$, and $S$ a properly embedded and connected smooth symplectic surface in $X$.
Suppose  homology class $[S] \in H_2(X, \partial X)$ is divisible by $p$ and  $H_1(X;\Z)=0$ so that the cyclic $p$-th branched cover $\Sigma_p(S)$ is uniquely determined.
We also assume $K= \partial S = S \cap X$ is a transverse knot in $(Y, \xi)$. 
Then, we have 
\[
\langle c_{(p)} (Y, \xi, K) , BF_{\Z_p}^* (1)\rangle =1, 
\]
where the pairing is given by the usual homological pairing
\[
\wt{H}^{\Z_p}_* (SWF(-\Sigma_p(K), \fraks_0)) \otimes \wt{H}^*_{\Z_p} (SWF(-\Sigma_p(K), \fraks_0)) \to \mathbb{F}_p. 
\]
\end{thm}
\begin{rem}
    We do not know whether a similar pairing formula for $S^1\times \Z_p$-equivariant theory. 
\end{rem}

Suppose $G=\Z_p$ for a prime $p$.
Let $\wt{X}$ be a compact oriented $G$-$Spin^c$ 4-manifold with connected contact boundary $(\wt{Y}, \wt{\xi})$ and $\s_{\wt{X}}$ a $Spin^c$ structure whose restriction on the boundary is compatible with the $Spin^c$ structure induced by $\wt{\xi}$. Suppose $G$-action preserves $\wt{\xi}$.
Suppose $b_1(\wt{X})=0$.

Under the assumptions, from \cref{gluing}, we have the following equality: 
\[
\epsilon_{G} \circ ( BF_{G}(\wt{X}, \s_X)  \wedge \mathcal{C}_{G}(\wt{Y}, \wt{\xi})  ) = \Psi_{G}(\wt{X}, \s_{\wt{X}, \wt{\xi}}, \wt{\xi})
\]
up to $\Z_p$-equivariant stable homotopy. By applying $\wt{H}^*_{\Z_p}(-; \Z_p)$, we have 
\begin{align}\label{=}
\langle c_{(p)} (Y, \xi, K) , BF_{\Z_p}^* (1)\rangle  = \Psi_{\Z_p}(\wt{X}, \s_{\wt{X}, \wt{\xi}}, \wt{\xi})_* (1). 
\end{align}

Now we put $\wt{X}$ as the cyclic $p$-th branched covering space $\Sigma_p(S)$ along $S$. 
From \cref{double branch symp}, we can equip a $\Z_p$-invariant weak symplectic filling structure $\wt{\omega}$ on $\Sigma_p(S)$ of the induced $\Z_p$-invariant contact structure $\wt{\xi}$ introduced in \cref{Contact structures on branched coverring spaces}. 
Put $\s_{\wt{X}}= \s_{\wt{\om}}$ and $ \wt{Y}= \Sigma_p(K). 
$ 
Then, we have seen the equivariant homotopical refinement of Kronheimer--Mrokwa's invariant $\Psi_{\Z_p}(\wt{X}, \s_{\wt{X}, \wt{\xi}}, \wt{\xi})$ is homotopic to $\pm \id$ when it admits a $\Z_p$-invariant weak symplectic structure. Thus, we have 
\[
\Psi_{\Z_p}(\wt{X}, \s_{\wt{X}, \wt{\xi}}, \wt{\xi})_* (1) =1. 
\]
This proves the first half of \cref{paring formula}. The second half is almost the same as the first case. 

\cref{paring formula} implies the following property of the transverse knot invariant: 
\begin{cor}\label{fundamental porp for t-inv(ii)}
    If $(Y, \xi)$ has a weak symplectic filling $(X, \om)$ with $b_1=0$ and $K$ bounds a properly embedded symplectic surface $S$ in $X$ with $0=[S]\in H_2(X; \Z_2)$, then
    \[
    Q^l \cdot c_{(2)} (Y, \xi, K ) \neq 0 
    \]
    for $l \geq 0$.
\end{cor}

\section{Proof of adjunction formula \cref{main theo:general}}

We now state our adjunction formula stated in \cref{main theo:general}.

\begin{thm}[\cref{main theo:general}]\label{Thom}
 Let $(X, \om)$ be a weak symplectic filling of $(S^3, \xi_{std})$ with $b_1 (X) = 0 = b^+(X)$ and $K$ bounds a properly embedded and connected symplectic surface $S$ in $X$ divisible by $2$. 
 Then, one has 
\begin{align}\label{main eq3}
q_M ( K) =   g(S)   
 + \frac{1}{2}  (\langle c_1(\fraks), [S] \rangle - [S]^2   )  
  \end{align}
  and 
  \begin{align}\label{main eq11}
q_M ( K) = -d_3 (\Sigma_2(K), \wt{\xi}) - \frac{1}{2} + \frac{3}{4} \sigma(K)=\frac{1}{2}sl(K)+\frac{1}{2}. 
  \end{align}
 
\end{thm}

\begin{rem}
The same result is true even for normally immersed symplectic surfaces with only positive double points. In this case, it is proven in \cite{Et20} in the symplectic category, we can replace one positive double point with one genus. 
\end{rem}

\begin{rem}
Let $(X, \om)$ be a symplectic cobordisms from the contact manifold $(Y_0, \xi_0)$ to $(Y_1, \xi_ 1)$, $K_0$ a transverse knot in $(Y_0, \xi_0)$, and $K_1$
a transverse knot in $(Y_1, \xi_1)$. Further, assume that $Y_0$ and $Y_1$
are homology spheres. The following is proven in \cite[Lemma 2.13]{Et20}: 
    If $S$ is any normally immersed symplectic surface with transverse double points in $(X, \om)$ with
boundary $-K_0 \cup K_1$, then
\[
\langle c_1(X, \om), [S]\rangle = \chi(S) - sl(K_0) + sl(K_1) + [S] \cdot [S] - 2D,
\]
where $[S]$ is the homology class of the closed surface $S$, $g(S)$ is the genus of $S$ and $D$ be the number of double points of $S$ counted with sign.

We apply it to 
\[
X = (D^4, \om_{std}) \#_n \overline{\mathbb{C}P}^2, Y_0= K_0=\emptyset, Y_1=S^3, \xi_1 = \xi_{std}, \text{ and } K_1=K. 
\]
Then, we obtain 
\begin{align}\label{Bennequin equ}
\langle c_1(X, \om), [S]\rangle = 1- 2 g(S) + \operatorname{sl}(K) + [S] \cdot [S] - 2D. 
\end{align}
The above two equalities imply \eqref{Bennequin equ}. 

\end{rem}

Thus we proved that symplectic surfaces have minimal genus in their relative homology class under the assumption of the theorem.
In particular, this implies the  Milnor conjecture).

\begin{proof}[Proof of  \cref{Thom}]
By removing a small ball, regards $S$ as a cobordism from the unknot:
\[
(W , S): (S^3, U)\to (S^3, K).
\]
By taking the upside-down cobordism, this can be regarded as
\[
(W , S): (S^3, -K)\to (S^3, U).
\]

We consider the cobordism maps satisfying the commutativity of the diagram: 
\[
  \begin{CD}
    \wt{H}^*_{\Z_2}(S^0) @>{ BF^*_{\Z_2}}>> \wt{H}^*_{\Z_2}(SWF(\Sigma_2(-K))) \\
     @V{\iota}VV    @V{\iota}VV \\
     Q^{-1} \mathbb{F}_2[Q]   @>{\neq 0}>>   Q^{-1}\wt{H}^*_{\Z_2}(SWF(\Sigma_2(-K))) 
  \end{CD}.
\]
We can see that 
\[
Q^j BF^*(W, S)(1_U) \neq 0 
\]
for any $j \geq 0$ and 
\[
BF^*(W, S)(1_U) \not\in Q \wt{H}^*_{\Z_2}(SWF(\Sigma_2(-K))).
\]
Indeed, these follow from the gluing result and the non-vanishing result: 
\[
\mathcal{C}_K \circ BF_S \sim_{\Z_2} \pm \id.
\]
Indeed, the injectivity of $\mathbb{F}_2[Q] \to \wt{H}^*_{\Z_2}(SWF(\Sigma_2(-K)))$ implies the former claim and the surjectivity of  $\mathcal{C}^*_{K}:\wt{H}^*_{\Z_2}(SWF(\Sigma_2(-K))) \to \mathbb{F}_2[Q]$ implies the latter claim.
Therefore we obtain $BF^*(W, S)(1_U)$ minimizes the $q_M$ invariant from the rank one theorem \cref{rank1 theorem}. 

Therefore from the computation done in the proof of \cref{main slice-torus gen}, we have 
\begin{align*}
q_M ( K)&  =  g(S)  + \frac{1}{2}  (\langle c_1(\fraks), [S] \rangle - [S]^2   ) 
    \end{align*}

This completes the proof.
\end{proof}

\subsection{On Baraglia's invariants}

In this section, we compare our concordance invariant $q_M(K)$ with the Baraglia invariant $\theta^{(2)}(K)$.

\subsubsection{The definition of $d_{\Z_p, j} $ via cohomology}
First, we review the construction of $\theta^{(2)}(K)$.

Put $G=\Z_p$ in Baraglia and Hekmati's setting. 
Let $X$ be a space of type $\Z_p$-SWF \cite[Definition 3.6]{BH} i.e. $X$ is a pointed finite $(S^1\times \Z_p)$-CW complex such that the followings hold: 
\begin{itemize}
    \item[(i)] The fixed point set $X^{S^1}$ is $\Z_p$-homotopy equivalent to a sphere $V^+$, where $V$ is a real representation of $\Z_p$.
    \item[(ii)]  The action of $S^1$ is free on $X \setminus X^{S^1}$. 
\end{itemize}

It is confirmed that the $S^1\times \Z_p$-equivariant Seiberg--Witten Floer homotopy type $SWF(\wt{Y}, \mathfrak{s}, g) $ is a space of type $\Z_p$-SWF type. 

The inclusion from the fixed point set $\iota: X^{S^1} \to X$ induces a map 
\[
\iota^*: U^{-1} H^*_{S^1\times \Z_p} (X;\mathbb{F}_p) \to U^{-1} \wt{H}^*_{S^1 \times \Z_p}(X^{S^1} )\cong U^{-1}  H^*_{S^1\times \Z_p}.
\]
Note that we have 
 \[
 H^*_{S^1\times \Z_p}  = H^*_{S^1\times \Z_p}(\operatorname{pt}; \mathbb{F}_p)
 =\begin{cases}  \mathbb{F}_2[U, Q] &\text{ $p$ is $2$} \\
\mathbb{F}_p[U, R, S] / (R^2) & \text{ otherwise}
\end{cases}.
\]
The localization theorem implies $U^{-1} H^*_{S^1\times \Z_p} (X;\mathbb{F}_p)$ is a free rank-$1$ $U^{-1}  H^*_{S^1\times \Z_p} $-module. 
Let $\tau$ denote a generator of $U^{-1}  H^*_{S^1\times \Z_p} $. Now we recall a sequence of invariants $ \{d_{\Z_p,j} (X)\}_{j \geq 0}$.
\begin{itemize}
    \item[(i)] 
Suppose $p=2$.
\[
d_{G,j} (X) := \min \{ i | \exists x \in \wt{H}^i_{S^1\times \Z_2}(X;\mathbb{F}_2) ,\ \iota^* x \equiv   U^k Q^j \tau \operatorname{mod} Q^{j+1} \text{  for some } k \geq 0 \}-j . 
\]
\item[(ii) ]Suppose $p$ is odd prime. 
\[
d_{G,j} (X) := \min \{ i | \exists x \in \wt{H}^i_{S^1\times \Z_p}(X;\mathbb{F}_p) ,\ \iota^*  x \equiv  S^j U^k  \tau  \operatorname{mod} S^{j+1}, RS^{j+1}  \text{  for some }  k \geq 0 \}-2j .
\]

\end{itemize}

This is not the original definition but an alternative description given in \cite[Proposition 3.14]{BH}.

\subsubsection{The invariants $\delta_{G,j}$ for 3-manifolds with $\Z_p$-action}

Let $G=\Z_p$ act on an oriented rational homology $3$-sphere $\wt{Y}$ and $\fraks$ be a $G$-invariant spin$^c$ structure. 
Then, the invariant $d_{G, j} (\wt{Y},\fraks) $ is defined as 
\[
d_{G,j} (\wt{Y}, \fraks) := d_{G, j} (SWF(\wt{Y}, \fraks, g) ) - 2n(\wt{Y}, \fraks, g),  
\]
where $n(\wt{Y}, \fraks, g)$ is the correction term introduced in \cite{Man03} and $SWF(\wt{Y}, \fraks, g)$ denotes the $S^1\times \Z_p$-equivariant Seiberg--Witten Floer stable homotopy type for a given $\Z_p$-invariant Riemann metric. 
Now we define 
\[
\delta_{G, j} (\wt{Y}, \fraks) := \frac{1}{2}d_{G,j} (\wt{Y}, \fraks).
\]
It is confirmed $\delta_{G, j} (\wt{Y}, \fraks)$ is constant when $j$ is sufficiently large. 

Define 
\[
j_G (\wt{Y}, \fraks) := \min \Set{ j \geq 0 | \delta_{G, j} (\wt{Y})= \delta_{G, \infty} (\wt{Y}) }.
\]

\subsection{Baraglia--Hekmati's knot concordace invariant $\delta^{(p)}_j(K)$}

Since we will also consider knots in general homology 3-spheres, we consider a generalization of Baraglia--Hekmati's knot concordance invariants of knots in other 3-manifolds. 
Let us review Baraglia--Hekmati's knot concordance invariant $\delta^{(p)}_j(K)$.

\subsection{The invariants $\delta^{(p)}_j(K)$ for a pair $(Y,K)$}
For a given knot $K$ in $S^3$ and a prime number $p$, Baraglia--Hekmati defined 
\[
\delta^{(p)}_j(K): = 4\delta_{\Z_p, j} (\Sigma_p(K), \fraks), 
\]
where $\fraks$ is the unique spin structure on $\Sigma_p(K)$. Obviously, this definition can be generalized to a knot in any oriented homology 3-sphere $Y$, which we shall write $\delta_j^{(p)}(Y, K)$.

We also define the invariants  $j
^{(p)}
(Y, K) $ are defined to be the smallest $j$ such that \[
\delta^{(p)}_j
(Y, K) = \delta ^{(p)}_\infty (Y, K).
\]
The followings are fundamental properties of the invariants $\delta_j^{(p)}(Y, K)$ and $j^{(p)}(Y, K)$. 
\begin{lem}
The invariants $\delta_j^{(p)}(Y, K)$ and $j^{(p)}(Y, K)$ are homology concordance invariant, i.e. if there are a homology cobordism $W$ from $Y$ to $Y'$ and there is a smoothly and properly embedded annulus $S$ in $W$ such that $\partial S= K \cup K'$, then one has 
\[
\delta_j^{(p)}(Y, K)= \delta_j^{(p)}(Y', K') \text{ and } j^{(p)}(Y, K)= j^{(p)}(Y', K'). 
\]
\end{lem}
\begin{proof}
We take the $p$-th branched covering space along the annulus $S$ and obtain $\Z_p$-equivariant $\Z_2$-homology cobordism and apply the inequality proven in \cite{BH}. 
\end{proof}
Therefore, if we write the homology concordance group by 
$\Theta^{(3,1)}_\Z$, we have the invariants 
\[
\delta_j^{(p)}: \Theta^{(3,1)}_\Z \to \Q, \quad j^{(p)}: \Theta^{(3,1)}_\Z \to \Z_{\geq 0}. 
\]

For a knot $K$ in $S^3$, we recall 
\[
\theta^{(2)}(K) := \max \{ 0, j^{(2)}(S^3, K^*) - \frac{1}{2} \sigma (K)\} . 
\]
We sometimes denote $\theta^{(2)}(K)$ as  $\theta(K)$. 
\subsubsection{$q_M (K) \leq \theta^{(2)}$}
In this section, we prove the following inequality.
\begin{thm}[\cref{qmtheta}]
For any knot $K \subset S^3$, we have
\[
q_M(K) \leq \theta^{(2)}(K).
\]
\end{thm}
\begin{proof}
By definition, it is enough to prove 
\[
q^\dagger_M(-K)-\frac{3}{4}\sigma(K)\leq j^{(2)}(-K)-\frac{\sigma(K)}{2}.
\]
By replacing $K$ with $-K$, it is sufficient to prove
\[
q^\dagger_M(K)\leq j^{(2)}(K)-\frac{1}{4}\sigma(K).
\]
By the definition of $j^{(2)}(K)$, we can take 
\[
x \in \tilde{H}^i_{S^1\times \Z_2}(SWF(\Sigma_2(K)))
\]
satisfying 
\[
i-j^{(2)}(K)=-\frac{\sigma(K)}{4}, 
\]
and
\begin{equation}\label{def of theta}
\iota x=Q^{j^{(2)}(K)} U^k \tau \text{ mod } Q^{j^{(2)}(K)+1}
\end{equation}
for some $k \in \Z^{\geq 0}$.
Now the Thom--Gysin exact sequence for the $S^1$-bundle
\[
S^1\to SWF(\Sigma_2(K))\wedge_{\Z_2} E(\Z_2 \times S^1)_+ \xrightarrow{\pi} SWF(\Sigma_2(K))\wedge_{\Z_2 \times S^1} E(\Z_2 \times S^1)_+
\]
is
\[
\cdots \to \wt{H}^*_{\Z_2 \times S^1}(SWF(\Sigma_2(K))) \xrightarrow{U}  
\wt{H}^*_{\Z_2 \times S^1}(SWF(\Sigma_2(K)))
\xrightarrow{\pi^*}     \wt{H}^*_{\Z_2 }(SWF(\Sigma_2(K))) \to \cdots.
\]
In order to prove the desired inequality, 
it is sufficient to see
\[
Q^n \pi^*(x)\neq 0
\]
for all $n \in \Z^{\geq 0}$.
This follows from the Thom--Gysin exact sequence above.
Indeed, suppose $Q^n \pi^*(x)=0$ for some $n >0$.
Then by the exactness of Thom--Gysin exact sequence, there exists some $z \in \wt{H}^{i-2}_{\Z_2 \times S^1}(SWF(\Sigma_2(K)))$ such that
\[
U z=Q^n x.
\]
By multiplying $Q^n$, we obtain from \eqref{def of theta} the relation
\[
\iota U z=Q^{j^{(2)}(K)+n} U^k \tau \quad \text{ mod } Q^{j_{(2)}(K)+n+1}
\]
Thus 
\[
\iota z=Q^{j^{(2)}(K)+n} U^{k-1} \tau \quad \text{ mod } Q^{j_{(2)}(K)+n+1}.
\]
Since $gr^\Q (z)=i-2$, 
this implies
\[
\delta_{j^{(2)}(K)+n}(K) <\delta_{\infty}(K).
\]
this contradicts the decreasing property of the sequence $\delta^{(2)}_{j}$.
\end{proof}

\section{Constraints on homology classes}
We next give constraints on homology classes represented by symplectic surfaces in symplectic fillings using \cref{rank1 theorem}.

\subsection{Homological constraints from $\theta(K, m)$}

We next give constraints on homology classes represented by symplectic surfaces in symplectic fillings using Baraglia's $\theta(K, m)$-invariant combined with adjunction equality. 
For the purpose, we introduce a sequence of invariants $l^{(p)}(K)$: 
\begin{defn}
    For a given knot $K$ in $S^3$ and a prime number $p$, we define 
    \[
    l^{(p)}(K) := \min \{i | HF^+_i(\Sigma_p (K), \mathfrak{s}_0 ) \neq 0 \}\in \Q, 
    \]
    where $\mathfrak{s}_0 $ is the unique spin structure which is invariant under $\Z_p$-action. 
\end{defn}

The following is the most general result in this section. 

\begin{thm}\label{divisibility of homology class}
Let $(X, \omega)$ be a symplectic filling of $(S^3, \xi_{std})$ and $S \subset X$ be a properly embedded symplectic surface bounded by $K$.
\begin{itemize}
\item[(i)]

Suppose 
\[
l^{(2)}(-K)-\frac{3 \sigma (K)}{4}+\frac{3}{8} [S]^2-\frac{1}{2}\eta([S]/2)
> \frac{1}{2}\operatorname{sl}(K)+\frac{1}{2}+\frac{1}{2}\langle c_1(TX), [S]\rangle+\frac{1}{2}[S]^2, 
\]
where
\[
\eta(x):=\min_{c\in H^2(X; \Z) }\left\{\ -(x+c)^2-b_2(X)\ \middle|\  c \equiv w_2(X) \text{ mod }2  \right\}
\]
for $x \in H_2(X , \partial X; \Z)$ and $\sigma(K)$ denotes the knot signature. 
Then $[S] \in H_2(X, \partial X; \Z)$ is not divisible by $2$.
\item[(ii)]
Let $p$ be an odd prime.
Suppose 
\[
\frac{l^{(p)}(-K)}{p-1}-\frac{3 \sigma ^{(p)}(K)}{4(p-1)}+\frac{p+1}{4p} [S]^2
> \frac{1}{2}\operatorname{sl}(K)+\frac{1}{2}+\frac{1}{2}\langle c_1(TX), [S]\rangle+\frac{1}{2}[S]^2, 
\]
where $\sigma^{(p)} (K)$ is the following sum of the Levine--Trsitram signature 
\[
\sigma^{{(p)}} (K) := \sum_{\om^p=1,\, \om \neq 1} \sigma_\om (K). 
\]
Then $[S] \in H_2(X, \partial X; \Z)$ is not divisible by $p$.
\end{itemize}
\end{thm}

For $i=0, 1$, let $(Y_i, \xi_i)$ be two contact  integer homology 3-spheres and $K_i \subset (Y_i, \xi_i) $ be transverse knots.
Let $(X, \om): (Y_0, \xi_0)\to (Y_1, \xi_1)$ be a symplectic cobordisms.

The following is proven in \cite[Lemma 2.13]{Et20}, which is called the adjunction equality.  
\begin{prop}
    If $S$ is any immersed connected symplectic surface with transverse double points in $(X, \om)$ with
boundary $-K_0 \cup K_1$, then
\[
\langle c_1(X, \om), [S]\rangle = \chi(S) - sl(K_0) + sl(K_1) + [S] \cdot [S] - 2D,
\]
where $[S]$ is the homology class of the closed surface $S$, $g(S)$ is the genus of $S$ and $D$ be the number of double points of $S$ counted with sign.
\end{prop}

We apply it to 
\[
X = (D^4, \om_{std}) \#_m \overline{\mathbb{C}P}^2, Y_0= K_0=\emptyset, Y_1=S^3, \xi_1 = \xi_{std}, \text{ and } K_1=K. 
\]
Then, we obtain 
\begin{align}\label{Bennequin equI}
\langle c_1(X, \om), [S]\rangle = 1- 2 g(S) + \operatorname{sl}(K) + [S] \cdot [S] . 
\end{align}

Next, we recall the following inequality due to Baraglia \cite[Corollary 6.5]{Ba22}. 

\begin{thm}[\cite{Ba22}]\label{BH full inequality}
Let $X$ be a compact negative definite $4$-manifold with $\partial X=S^3$ and $H_1 (X;\Z )=0$ and $S$ be a properly and smoothly embedded surface in $X$ bound by a knot $K$ in $S^3$. 
    \begin{itemize}
        \item If $[S]$ is divisible by $p$, then 
        \[
        g(S) \geq \frac{l^{(p)}(-K)}{p-1} - \frac{3 \sigma^{(p)}(K)}{4 (p-1) } + \frac{p+1}{4p } [S]^2
        \]
        \item If $p=2$,  
        \[
        g(S) \geq  {l^{(2)}(-K)} - \frac{3 \sigma(K)}{4  } + \frac{3}{8} [S]^2 -\frac{1}{2} \eta ( [S]/2) . 
        \]
    \end{itemize}
    
\end{thm}

\begin{proof}[Proof of \cref{divisibility of homology class}]
We just combine \cref{BH full inequality} with \eqref{Bennequin equ}. 
\end{proof}

\begin{proof}[Proof of \cref{Q-minimize statement}]
    If $c(\Sigma_p (K), \wt{\xi})$ minimizes the $\mathbb{Q}$-grading of all elements in $HF^+_* (-\Sigma_p (K)) $, we can see 
    \[
    l^{(p)}(-K)  = \operatorname{gr}^\Q ( c(\Sigma_p (K), \wt{\xi}) ) = -d_3(\Sigma_p(K), \wt{\xi} )- \frac{1}{2}. 
    \]
    From \cref{double branch symp}, $(\Sigma_p(K), \wt{\xi})$ has a symplectic filling. Therefore, the contact invariant $c(\Sigma_p (K), \wt{\xi}) $ (in Heegaard Floer homology) is non-trivial.
    Now we have Itoh's formula \cite[Theorem 1.1]{It17} 
    \[
    d_3(
\Sigma_p (K), \wt{\xi}) = -\frac{3}{4} \sigma^{(p)}(K) -\frac{p-1}{2}\operatorname{sl} (K) - \frac{1}{2}p . 
    \]

    So, by combining two equations, we get 
    \[
    l^{(p)}(-K) =  \frac{3}{4} \sigma^{(p)}(K) +\frac{p-1}{2}\operatorname{sl} (K) + \frac{1}{2}p  - \frac{1}{2}. 
    \]
    By using the above two equations, we see the assumptions of  \cref{divisibility of homology class} 
    \begin{align*}
l^{(2)}(-K)-\frac{3 \sigma (K)}{4}+\frac{3}{8} [S]^2-\frac{1}{2}\eta([S]/2)
> \frac{1}{2}\operatorname{sl}(K)+\frac{1}{2}-\frac{1}{2}\langle c_1(TX), [S]\rangle+\frac{1}{2}[S]^2 \\ 
\frac{l^{(p)}(-K)}{p-1}-\frac{3 \sigma ^{(p)}(K)}{4(p-1)}+\frac{p+1}{4p} [S]^2
> \frac{1}{2}\operatorname{sl}(K)+\frac{1}{2}-\frac{1}{2}\langle c_1(TX), [S]\rangle+\frac{1}{2}[S]^2
\end{align*}
are equivalent to   
    \begin{align*} 
    -\frac{1}{8} [S]^2-\frac{1}{2}\eta([S]/2)
> -\frac{1}{2}\langle c_1(TX), [S]\rangle  \text{ if } p=2 \\ 
    \frac{-p+1}{4p} [S]^2
> -\frac{1}{2}\langle c_1(TX), [S]\rangle \text{ if } p \neq 2. 
    \end{align*}
    This completes the proof. 
\end{proof}

\begin{proof}[Proof of \cref{computation for CP}]

We need to check the examples of knots in the list satisfy the assumption \cref{computation for CP}. It is proven that the Heegaard Floer homology $HF^+_*$ for $-\Sigma (2,3,6n \pm 1)$, $-\Sigma(2,5,7)$ and $-\Sigma(2,5,9)$ have minimum gradings as follows : 
 \begin{itemize}
     \item the minimum gradings of $-\Sigma (2,3,6n - 1)$ are $-2$,
     \item the minimum gradings of $-\Sigma (2,3,6n + 1)$ are $0$,
     \item the minimum gradings of $-\Sigma (2,5,7)$ and $-\Sigma(2,5,9)$ are $0$ and $-2$ respectively.
 \end{itemize}
 For example, see \cite{Tw13}. 
 On the other hand, one can easily see the grading of the contact element for the $\Z_2$ invariant contact structure for double-branched covering of corresponding torus knots coincide with the minimum degrees above. 
 Now, we use \cref{computation for CP}. 
Let $x$ be an element in $2 H_2(D^4 \# m\overline{\C P}^2;\Z)$ with an expression 
\[
x = \sum x_i e_i \in H_2 ( D^4 \# m\overline{\C P}^2; \Z).
\]
Suppose
    \[
\frac{1}{8} \sum x_i^2 -\frac{1}{2}\eta(x/2)
\leq -\frac{1}{2} \sum x_ i , 
\]
where 
\[
\eta(x/2)=\min_{(y_1, \cdots, y_n) =c\in H^2(X; \Z) }\left\{\ \sum (x_i/2 +y_i)^2-m\ \middle|\  c \equiv w_2(X) \text{ mod }2  \right\}. 
\]
This shows
\begin{align*}
& \sum x_i^2 + 4 \sum x_ i  \leq  \sum x_i^2 -4 \sum x_ i  -4\eta(x/2) \leq 0 .
\end{align*}
Therefore, we have 
\[
\sum  (x_i + 2)^2 \leq 4 m
\]
which implies $x_i= 0, -2 , -4$ for any $i$.
This completes the proof. 

\end{proof}

\subsection{Proof of \cref{main symp sur}}

Let us recall the statement of  \cref{main symp sur}. 
\begin{thm}[\cref{main symp sur}]\label{symplectic s3}
     Let $K$ be a transverse knot in $(S^3, \xi_{std})$. Suppose 
    \[
    g_4(K) + \frac{1}{2}\sigma(K) >0. 
    \]
Then there is no a properly embedded connected symplectic surface $S$ in $D^4 \#_n -\C P^2$ with $\partial S = K$ and $[S]$ is divisible by $2$ such that 
\[
g(S) - \frac{1}{4}[S]^2 + \frac{1}{2}\sigma (K)=0. 
\]
\end{thm}

\begin{proof}[Proof of \cref{symplectic s3}]
    Since $K$ is the boundary of symplectic surface $S$ in $D^4\# n \overline{\C P^2}$, one can see $K$ has a quasi-positive representation. It implies we can take a symplectic surface $S'$ in $(D^4, \om_{std}) $ bounded by $K$. 
    Now, we consider the (up-side-down) $S^1\times \Z_2$-equivariant cohomological Bauer--Furuta invariant 
    \[
    x= BF_{S}^* (1), y =  BF_{S'}^* (1) \in \wt{H}^*_{S^1\times \Z_p }(SWF ( - \Sigma_p(K)); \mathbb{F}_p). 
    \]
    From the results about localizations with $S^1$-action,  we see 
\begin{align*}
    \iota^* x = Q^{b^+(\Sigma_p (S)) } U^m =  Q^{g(S) - \frac{1}{4}[S]^2+ \frac{1}{2}\sigma (K)  } U^m  \text{ and } 
    \iota^* y  = Q^{b^+(\Sigma_p (S'))  } U^{m'} = Q^{g(S') + \frac{1}{2}\sigma (K) }U^{m'}
\end{align*}
for some $m$ and $m'$, where $\iota^* : H^{*}_{S^1\times \Z_p} (SWF(-\Sigma_p (K)), \mathfrak{s}_0; \mathbb{F}_p) \to H^{*}_{S^1\times \Z_p} (SWF(-\Sigma_p (K))^{S^1}, \mathfrak{s}_0; \mathbb{F}_p)$ induced from the inclusion $\iota : SWF(-\Sigma_p (K))^{S^1} \to SWF(-\Sigma_p (K))$.
Thus, we conclude $Q^{j'}U^jx  \neq U^j Q^{j'}y$ for any $j$ and $j'$.

    On the other hand, the Thom--Gysin exact sequence with respect to 
    \[
    \Z_2 \to SWF(\Sigma_2 (-K)) \wedge E(S^1\times \Z_2)_+  \to SWF(\Sigma_2 (-K)) \wedge_{S^1} ES^1_+
    \]
    implies $\cok Q$ in $S^1\times \Z_2$-equivariant cohomologies are rank $1$-module as $\mathbb{F}_2[U]$-modules and we have 
    \[
    \cok Q \cong \im \pi_* \subset H^*_{S^1}(SWF(-\Sigma_2(K), \fraks_0)) \cong  HF^+_*(-\Sigma_2(K), \fraks_0) 
    \]
  Note that 
    \[
    \{U^j x, U^j y\} \subset \cok Q
    \]
    are giving non-trivial elements.  We see $U^jx = U^jy$ in $\cok Q$ for some $j$ since it has only one non-trivial element in a sufficiently large degree.  This implies 
    \[
    U^jx - U^jy = Q z . 
    \]
    This implies 
    $U^j \iota^* x - U^j\iota^* y$ is divisible by $Q$. 
It implies
\[
Q^{g(S) - \frac{1}{4}[S]^2+ \frac{1}{2}\sigma (K)  } U^{m+j}  - Q^{g(S')+ \frac{1}{2}\sigma (K)  } U^{m'+j}  
\]
is divisible by $Q$. But we are assuming that 
\[
g(S')+ \frac{1}{2}\sigma (K) >0 \text{ and }  g(S) - \frac{1}{4}[S]^2+ \frac{1}{2}\sigma (K)=0
\]
which contradict each other. 
\end{proof}

Note that, from \eqref{Itoh formula}, we have 
\[
-d_3(\Sigma_2(K)) - \frac{1}{2} = \frac{3}{4} \sigma(K) +\frac{1}{2} sl(K) +  \frac{1}{2}. 
\]

\begin{proof}[Proof of \cref{ex for 510n+3}]
This follows from \cref{main symp sur} combined with the adjunction equality. 
\end{proof}

\begin{proof} [Proof of \cref{quasipos app}]
It is proven in \cite{Ru93} that the positively cusped Whitehead double of a strongly quasipositive knot is again strongly quasipositive. Since $Wh(K)$ is not isotopic to the unknot, we see 
\[
g_4( Wh(K) ) = g_3(Wh(K)) = 1 = \frac{1}{2} sl (Wh(K)) + \frac{1}{2} . 
\]
On the other hand, it is easy to see $\sigma(Wh(K))=0$. Therefore, one can use \cref{ex for 510n+3} and obtain 
\[
-\langle c_1(\om_{std}), [S]\rangle + \frac{1}{2}[S]^2+ \sigma(K)  \neq -2. 
\]
This implies 
\[
[S] \neq -2 [\C P^1]. 
\]

\end{proof}

\section{Proof of \cref{BP inequality}}
Let $\mathcal{L}$ and $\mathcal{T}$ be a Legendrian knot and a transverse knot in the standard contact 3-sphere.
Let $p$ be a prime and denote by $\theta^{(p)}$ the concordance invariant introduced by Baraglia \cite{Ba22}.
The Bennequin-type inequalities
\[
 sl(\mathcal{T})\leq 2\theta^{(p)}(\mathcal{T}
 )-1
\]
\[
tb(\mathcal{L})+|rot(\mathcal{L})| \leq 2\theta^{(p)}(\mathcal{L})-1
\]
for $p=2$ follows from a well-known result of slice-torus invariants, the fact that $q_M$ is a slice-torus invariant, and our inequality $q_M \leq \theta^{(2)}$.
In this section, we prove these inequalities for odd prime $p$ by modifying the argument known for slice-torus invariants.
\begin{thm}\label{AE theta}
Let $p$ be prime.
Let $K \subset (S^3, \xi_{\operatorname{std}}=\ker \lambda_{\operatorname{std}}) $ be a (quasi-positive) transverse knot, and suppose there is a connected embedded symplectic surface $S \subset (D^4, \omega_{\operatorname{std}})$ such that $\partial S=K$.
Suppose moreover that there is a symplectic cobordism
\[
S' \subset ([0, 1]\times S^3, d(e^t \lambda_{\operatorname{std}}))
\]
from $K$ to $T(a, b)$ for some positive coprime $(a, b)$.
Let $p$ be a prime not dividing $ab$.
Then 
\[
  \theta^{(p)} (K) =  g (S) 
  \]
  and
  \[ \theta^{(p)}  (T(a,b)) - \theta^{(p)}  (K) =  g(S') \] 
hold.

\end{thm}
\begin{proof}
For $S\cup S'$, by applying Baraglia--Hekmati's inequality, we obtain 
\[
\theta^{(p)} (T(a,b)) \leq g(S \cup S') . 
\]
By deforming the symplectic structure suitably, one can suppose $S\cup S'$ is a symplectic surface again in $(D^4, \om_{std})$. 
On the other hand, the slice-Bennequin equality implies 
\[
g(S \cup S') = \frac{1}{2} \operatorname{sl} (T(a,b)) + \frac{1}{2}, 
\]
for a transverse representation of $T(a,b)$ with a maximal self-linking number.
Also, Baraglia--Hekmati proved 
\[
\theta^{(p)} (T(a,b)) = g_4 (T(a,b)) = \frac{1}{2} \operatorname{sl} (T(a,b)) + \frac{1}{2}.
\]

Thus, the inequality is an equality. 
Baraglia--Hekmati's cobordism inequality implies
\begin{align*}
 &   \theta^{(p)} (K) \leq g (S) \\
 &   \theta^{(p)}  (T(a,b)) - \theta^{(p)}  (K) \leq g(S') . 
\end{align*}
The equality for the torus knot above implies that in fact both of them are equalities.
Indeed, if either inequality is strict we have
\[
\theta^{(p)} (K)+\theta^{(p)}  (T(a,b)) - \theta^{(p)}  (K) <g (S)+g(S')=g(S\cup S')
\]
and this contradicts the equality for torus knots.

This completes the proof. 
\end{proof}

\begin{thm}[\cref{BP inequality}]
Let $K$ be a transverse knot $(S^3 , \xi_{std})$. Then, we have 
\[
 \operatorname{sl} (K)  \leq 2\theta^{(p)}(K)-1 , 
\]
where $\operatorname{sl}(K)$ denotes the self-linking number of $K$. 

\end{thm}

\begin{proof}
We follow the standard strategy for slice-torus invariants, though $\theta^{(p)}$ are not additive under connected sum, so are not slice-torus invariants. See  \cite{Shu, Livingston, Le13, Iida23} for example.
    We will treat two cases: 
    \begin{itemize}
        \item $K$ has a  braid representation with no negative crossings 
        \item The general case. 
    \end{itemize}
    The first case follows from \cref{AE theta} using the following lemma 
\begin{lem}\label{construction of symp cob}
Let $\beta$ be a positive braid.
Denote by $n$ the number of strands of $\beta$ and by $k$ the number of crossings.
Let $p$ be a prime.
Then there exists a symplectic cobordism
\[
S' \subset ([0, 1]\times S^3, d(e^t \lambda_{std}))
\]
from the braid closure $\hat{\beta}$ to $T(a, b)$ with
$b=n$
and 
$a$ is coprime to $n$ and $p$ (For example take $a$ to be a sufficiently large prime number).
In particular, for any prime $p$ we can choose $a, b$ so that $p$ is not dividing $ab$.
\end{lem}

\begin{proof}[Proof of \cref{construction of symp cob}]
We follow the proof of \cite[Lemma 1.5]{Le13} combined with \cite[Lemma 2.8]{Et20}. 
\end{proof}
    Therefore, we consider the second case. 
  Take a braid representation $\beta$ of $K$.
  Let us denote the number of positive and negative crossings by $x_+$ and $x_-$ respectively and denote the number of strands by $n$.

  Let us denote $\beta^+$ be the braid obtained from $\beta$ by crossing changes for all negative crossings.
  Now $\beta^+$ is a positive braid, so we can apply the first case and obtain 
  \[
  2\theta^{(p)}(\hat{\beta^+})-1=sl(\hat{\beta^+})
  \]
  for the braid closure  $\hat{\beta^+}$ of $\beta^+$.

  Bennequin's formula implies
\begin{align*}
sl(K)& =x_+-x_--n \\
sl(\hat{\beta}^+) & =x_++x_--n,
\end{align*}
thus
\[
sl(K)=sl(\hat{\beta}^+)-2x_-.
\]
There is the following crossing change formula for $\theta^{(p)}$ \cite{Ba22}
\[
|\theta^{(p)} (K) - \theta^{(p)} (K')| \leq 1, 
\]
if there is a crossing change from $K$ to $K'$.
Thus 
\[
2\theta^{(p)}(K)-1\geq 2\theta^{(p)}(\hat{\beta}^+)-2x_- -1=sl(K).
\]
\end{proof}
By the standard push-off argument (See \cite[Proposition 3.5.36]{Gei08}), the following result follows from the result for transverse knots above.
\begin{cor}\label{Legendrian}
Let $p$ be a prime. 
Let $K$ be a Legendrian knot in $(S^3 , \xi_{std})$. Then, we have 
\[
 \operatorname{tb} (K) +|\operatorname{rot} (K)| \leq 2\theta^{(p)}(K)-1 
\]
where $\operatorname{tb}(K)$ denotes the Thurston-Bennequin number of $K$ and $\operatorname{rot} (K)$ is the rotation number of $K$. 
 \qed
\end{cor}

\subsection{Proof of \cref{main Montesinos}} 
We prove \cref{main Montesinos} in this section. The main tools of the proof are \cref{quasi theta} and \cite[Theorem 1.8]{BH24}.

\begin{proof}[Proof of \cref{main Montesinos}]
Let $K$ be a Montesinos knot as in the assumption in \cref{main Montesinos}. 
Then \cite[Theorem 1.8]{BH24} implies 
\[
\theta^{(2)}(K) + \frac{1}{2}
\sigma(K)  = 0\text{ or }1
\]
under the assumption. If $K$ is smoothly concordant to a quasipositive knot, we have 
\[
\theta (K)=g_4(K) 
\]
from \cref{quasi theta} combined with concordance invariance of $\theta$ and $g_4$. But this contradicts with
\[
g_4 (K) + \frac{1}{2}\sigma(K)\geq 2.
\]

This completes the proof. 
\end{proof}

\subsection{Proof of \cref{Rasmussen}}

\begin{proof}[Proof of \cref{Rasmussen}]
Note that $2q_M$ is a slice torus invariant and it has been proven in \cite{FLL22} that, for any squeezed knot $K$, the values of slice torus invariants coincide. Moreover, \cref{main slice-torus} implies $2q_M(K)= -\sigma(K)$ for a knot $K$ if $\Sigma_2(K) $ is a L-space. 

Next, we prove (ii). Let $K$ be a Montesinos knot satisfying the assumption. 
Then, in \cite{BH24}, it is proven that 
\[
\theta (K) + \frac{1}{2}\sigma(K) = 0 \text{ or }1. 
\]
Since $-K$ is also a Montesinos knot satisfying the same assumption, we also see 
\[
\theta (-K) - \frac{1}{2}\sigma(K) = 0 \text{ or }1. 
\]
Now, we use the inequalities 
\[
q_M( K) \leq \theta( K),  \quad  q_M(- K) \leq \theta(- K)
\]
to obtain 
\[
-1 - \frac{1}{2}\sigma(K) \leq q_M(K) \leq -\frac{1}{2}\sigma(K) +1. 
\]
Thus, we have 
\[
 |q_M(K) + \frac{1}{2}\sigma(K)|\leq 1. 
\]
Again for a squeezed knot, any slice-torus invariant coincides with each other. Thus obtain the conclusion. 
The proof of (iii) is similar. 
\end{proof}

\begin{proof}[Proof of \cref{squeezed Montesinos}]
We first prove (i).
For a squeezed knot $K$ satisfying the assumption, we have 
\[
q_M( K) = -\frac{1}{2}\sigma(K) =\wt{s} (K) . 
\]
Thus we have 
\[
\wt{s}(K)>0
\]
from the assumption. 
Now, using \cite[Theorem 1.5]{DISST22}, we have 
\[
h(S^3_1(K))<0, 
\]
where $h$ is the Fr{\o}yshov invariant for the instanton Floer homology.
Therefore, we can apply \cite[Theorem 1.8]{NST19} to see $\{S^3_{1/n}(K)\}$ is a linearly independent sequence.

Let us prove (ii) next. 
For a squeezed Montesinos knot satisfying the assumption, we have 
\[
q_M(K) = \wt{s}(K),
\]
where $\wt{s}$ is a slice-torus invariant introduced in equivariant singular instanton Floer theory \cite{DISST22}. From the assumption combined with $-1 - \frac{1}{2}\sigma(K) \leq q_M(K) $, we have 
\[
\wt{s}(K)>0. 
\]
The latter proof is the same as that of (i).  The proof of (iii) is similar. So we omit it. 
\end{proof}

\begin{rem}
Since the double-branched covers of quasi-alternating knots are L-spaces \cite{OS05}, one can treat squeezed quasi-alternating knots in \cref{squeezed Montesinos}(i). However since Heegaard Floer tau invariant is determined by its signature \cite{MaOz08}, one can also derive a similar result for squeezed quasi-alternating knots from Heegaard knot Floer theory.  On the other hand, examples of non-quasi-alternating knots whose double branched covering spaces are L-spaces have been known, see \cite{Gre09, Mo17}. Thus, we hope \cref{squeezed Montesinos}(i) can also prove linear independence of $1/n$-surgeries 
for more large class of knots. 
\end{rem}

\section{Conjectures} \label{conjectures}

We conjecture that
our invariant $q_M$ and Hendricks--Lipshitz-Sarkar's invariant
$q_\tau$ from the equivariant Heegaard Floer homology are related as follows.
\begin{conj}For any knot $K$ in $S^3$, the relation
\[
q_\tau(K)=2 q^\dagger_M(K),
\]
or equivalently
\[
q_\tau(K)=-2q_M(K)-\frac{3}{2}\sigma(K)
\]
holds.
\end{conj}

In \cite{Ka18}, using equivariant Heegaard Floer homology, Kang defined an {\it equivariant contact invariant} 
\begin{align}\label{kang inv}
c_{\Z_2} (\wt{\xi} ) \in HF^{\Z_2}_* (-\Sigma_2(K), \fraks; \mathbb{F}_2),
\end{align}
where $HF^{\Z_2}_* (-\Sigma(K); \mathbb{F}_2)$ is the equivariant Heegaard Floer homology based on Hendricks--Lipshitz--Sarkar's construction for the covering involution of the double branched covering space $\Sigma_2(K)$ and $\fraks$ is the unique spin structure on $-\Sigma_2(K)$. There is a $\mathbb{F}_2[Q]$-module structure on $HF^{\Z_2}_* (-\Sigma_2(K); \mathbb{F}_2)$.

Note that Roso \cite{Br23} constructed an 
 equivariant Floer cohomology class from transverse knots using his formulation of equivariant contact class. We also conjecture his class coincides with ours. 
As a counterpart of \eqref{kang inv}, we also defined the invariants 
\[
{c}_{(p)}(Y, \xi, K) \in \wt{H}^{\Z_p}_* (SWF(-\Sigma_p(K)); \mathbb{F}_p).  
\]

The following is an expected relation with Kang's equivariant contact invariant. 
\begin{conj}
There is a functorial  isomorphism 
\[ \
\Psi:  \wt{H}^{\Z_2}_* (SWF(\Sigma_2(K)); \mathbb{F}_2) \to HF^{\Z_2}_* (\Sigma(K); \mathbb{F}_2)
\]
as $\Z_2[Q]$-modules such that
\[
\Psi(c_2(S^3, \xi_{\operatorname{std}}, K)) = c_{\Z_2}(\wt{\xi}). 
\]
\end{conj}
A kind of naturality is proven in \cite{Ka18} for a certain class of symplectic cobordisms. Such a naturality should be true for our invariant as well. 
 Related to the above conjecture, we expect 
   \[
   c_{(2)}(Y, \xi, S(K)) = Qc_{(2)}(Y, \xi, K) , 
   \]
   where $S(K)$ denotes the stabilization of the transverse knot $K$.

\begin{ques}
    If $q_M(K)>0$, do we have $\delta(S_1^3(K)) <0$? Here $\delta(Y)$ is the monopole Fr\o yshov invariant introduced in \cite{Fr10}. 
\end{ques}
Note that if we consider $\tau(K)$ and $\wt{s}$ instead of $q_M(K)$, the corresponding statement has been already proven for Heegaard Floer $d$-invariant in \cite{HW16} (see also \cite{Sa18}) and instanton Fr\o yshov invariant \cite{Fr02} in \cite{DISST22}  respectively. 


\bibliographystyle{plain}
\bibliography{tex}

\end{document}